\documentclass[11pt,twoside,a4paper]{article}
\usepackage{amsmath,amssymb,mathrsfs}
\usepackage{amsthm}

\usepackage{inputenc}
\usepackage{avant}
\usepackage[english]{babel}
\usepackage{thmtools,thm-restate}
\usepackage[nottoc]
{tocbibind}
\usepackage{graphicx}
\usepackage{enumitem}
\usepackage{tikz}

\usepackage{bbm}
\usepackage{esint}
\usepackage{authblk}
\usepackage{mathtools}
\usepackage{dsfont}
\usepackage{interval}
\intervalconfig{soft open fences}

\usepackage{nicematrix}
\NiceMatrixOptions{cell-space-limits = 1pt}

\usepackage[inner=2.4cm,outer=2.4cm,top=2.8cm,bottom=2.8cm]{geometry}
\usepackage[all]{xy}
 \usepackage{todonotes}

\usepackage[bookmarks=true]{hyperref}
\usepackage{xcolor}
\hypersetup{
    colorlinks,
    linkcolor={red!50!black},
    citecolor={blue!50!black},
    urlcolor={blue!80!black},
}

\mathtoolsset{showonlyrefs}  % uncomment to tag only the referred equations (post-production)

\newcommand{\norm}[1]{\left\lVert#1\right\rVert}
\newcommand{\normdot}{{\|\!\cdot\!\|}}
\newcommand*\diff{\mathop{}\!\mathrm{d}}

\newcommand{\Leb}{\mathscr{L}}

\newcommand{\N}{\mathbb{N}}
\newcommand{\R}{\mathbb{R}}

\newcommand{\p}{\mathtt p} %projection
\newcommand{\de}{\ensuremath{\, \mathrm d}} % in the integrals

\newcommand\restr[2]{{% we make the whole thing an ordinary symbol
  \left.\kern-\nulldelimiterspace % automatically resize the bar with \right
  #1 % the function
  \right|_{#2} % this is the delimiter
  }}

\DeclarePairedDelimiter\abs{\lvert}{\rvert}%

\makeatletter
\let\oldabs\abs
\def\abs{\@ifstar{\oldabs}{\oldabs*}}
\newcommand{\sfD}{\mathsf D}
\newcommand{\Dreg}{\sfD^+_{reg}}

\newcommand{\cd}{\mathsf{CD}}
\newcommand{\bm}{\mathsf{BM}}

\newcommand{\MCP}{\mathsf{MCP}}

\newcommand{\X}{\mathsf{X}}

\newcommand{\M}{\mathcal{M}}
\newcommand{\J}{\mathcal{J}}

\newcommand{\di}
{\mathsf d} %standard mms distance notation
\newcommand{\m}{\mathfrak m} %standard mms measure notation 

\DeclareMathOperator{\Geo}{Geo}

\newcommand{\Prob}{\mathscr{P}}

\newcommand{\dis}{\mathcal D}

\newcommand{\sF}{sub-Fin\-sler }
\newcommand{\sr}{sub-Rie\-man\-nian }

\newcommand{\sinom}{\sin_\Omega}
\newcommand{\cosom}{\cos_\Omega}
\newcommand{\sinomp}{\sin_{\Omega^\circ}}
\newcommand{\cosomp}{\cos_{\Omega^\circ}}

\newcommand{\e}{{\rm e}}

\newcommand{\hei}{\mathbb{H}}

\renewcommand{\epsilon}{\varepsilon}
\renewcommand{\phi}{\varphi}

\usepackage{titlesec}
\titleformat{\section}
  {\scshape\large\centering}
  {\thesection}{0.5em}{}
\titleformat{\subsection}
  {\scshape\centering}
  {\thesubsection}{0.4em}{}

\title{\textbf{Measure contraction property and curvature-dimension condition on sub-Finsler Heisenberg groups}}
\date{\today}

\author{Samu\"el Borza\footnote{Scuola Internazionale Superiore di Studi Avanzati (SISSA, Trieste). \textit{E-mail}:  \href{mailto:sborza@sissa.it}{sborza@sissa.it}}, \ 
Mattia Magnabosco\footnote{Mathematical Institute, University of Oxford. \textit{E-mail}:  \href{mailto:mattia.magnabosco@maths.ox.ac.uk}{mattia.magnabosco@maths.ox.ac.uk}}, \ 
Tommaso Rossi\footnote{Laboratoire Jacques-Louis Lions, Sorbonne Universit\'e. \textit{E-mail}: \href{mailto:tommaso.rossi@inria.fr}{tommaso.rossi@inria.fr}} \ 
and Kenshiro Tashiro\footnote{Department of Mathematics, Tohoku University. \textit{E-mail}:  \href{mailto:kenshiro.tashiro.b2@tohoku.ac.jp}{kenshiro.tashiro.b2@tohoku.ac.jp}}}

\newtheoremstyle{remark}% <name>
        {10pt}% <Space above>
        {10pt}% <Space below>
        {}% <Body font>
        {}% <Indent amount>
        {\itshape}% <Theorem head font>
        {.}% <Punctuation after theorem head>
        {.4em}% <Space after theorem head>
        {}% <Theorem head spec (can be left empty, meaning 'normal')>

\newtheoremstyle{proof}% <name>
        {10pt}% <Space above>
        {10pt}% <Space below>
        {}% <Body font>
        {}% <Indent amount>
        {\itshape}% <Theorem head font>
        {.}% <Punctuation after theorem head>
        {.4em}% <Space after theorem head>
        {}% <Theorem head spec (can be left empty, meaning 'normal')>
        
\newtheoremstyle{definition}% <name>
        {10pt}% <Space above>
        {10pt}% <Space below>
        {}% <Body font>
        {}% <Indent amount>
        {\bfseries}% <Theorem head font>
        {.}% <Punctuation after theorem head>
        {.4em}% <Space after theorem head>
        {}% <Theorem head spec (can be left empty, meaning 'normal')>        

\newtheoremstyle{theorem}% <name>
        {10pt}% <Space above>
        {10pt}% <Space below>
        {\slshape}% <Body font>
        {}% <Indent amount>
        {\bfseries}% <Theorem head font>
        {.}% <Punctuation after theorem head>
        {.4em}% <Space after theorem head>
        {}% <Theorem head spec (can be left empty, meaning 'normal')>

\theoremstyle{theorem}
\newtheorem{theorem}{Theorem}[section]
\newtheorem{prop}[theorem]{Proposition}
\newtheorem{corollary}[theorem]{Corollary}
\newtheorem{lemma}[theorem]{Lemma}
\newtheorem{conj}[theorem]{Conjecture}

\theoremstyle{definition}
\newtheorem{definition}[theorem]{Definition}

\theoremstyle{remark}
\newtheorem{remark}[theorem]{Remark}
\newtheorem{example}[theorem]{Example}
\newtheorem{notation}[theorem]{Notation}

\theoremstyle{proof}
\newtheorem*{pro}{Proof}
 {\popQED\end{pro}}

\makeatletter
\renewcommand\xleftrightarrow[2][]{%
  \ext@arrow 9999{\longleftrightarrowfill@}{#1}{#2}}
\newcommand\longleftrightarrowfill@{%
  \arrowfill@\leftarrow\relbar\rightarrow}
\makeatother

\makeindex
\begin{document}

\maketitle

\begin{abstract}
    In this paper, we investigate the validity of synthetic curvature-dimension bounds in the \sF Heisenberg group, equipped with a positive smooth measure. Firstly, we study the measure contraction property, in short $\MCP$, proving that its validity depends on the norm generating the \sF structure. Indeed, we show that, if it is neither $C^1$ nor strongly convex, the associated Heisenberg group does not satisfy $\MCP(K,N)$ for any pair of parameters $K\in \R$ and $N\in (1,\infty)$. 
    %On the contrary, we prove that $\MCP(0,N)$ holds for some $N\in (1,\infty)$, in the \sF Heisenberg group, equipped with a $C^{1,1}$ strongly convex norm and with the Lebesgue measure.
    On the contrary, we prove that the \sF Heisenberg group, equipped with a $C^{1,1}$ and strongly convex norm, and with the Lebesgue measure, satisfies $\MCP(0,N)$ for some $N\in (1,\infty)$. Additionally, we provide a lower bound on the optimal dimensional parameter, and we also study the case of $C^1$ and strongly convex norms. Secondly, we address the validity of the curvature-dimension condition pioneered by Sturm and Lott-Villani, in short $\cd(K,N)$. We show that the \sF Heisenberg group, equipped with a $C^1$ and strongly convex norm, and with a positive smooth measure, does not satisfy the $\cd(K,N)$ condition for any pair of parameters $K\in \R$ and $N\in (1,\infty)$. Combining this result with our findings regarding the measure contraction property, we conclude the failure of the $\cd$ condition in the Heisenberg group for every \sF structure.
\end{abstract}

\tableofcontents

\section{Introduction}

In the seminal contributions \cite{sturm2006,sturm2006-1,lott--villani2009}, Sturm and Lott--Villani introduced the celebrated curvature-dimension condition, in short $\cd(K,N)$. This was the first successful notion of synthetic Ricci curvature bound in the non-smooth setting of metric measure spaces. The crucial observation was that a weighted Riemannian manifold has generalized Ricci curvature bounded below by $K\in\R$ and dimension bounded above by $N\in (1,+\infty]$ if and only if the so-called $N$-\emph{R\'enyi entropy functional} is $(K,N)$-convex along Wasserstein geodesics, cf.\ \cite{CEMCS,vonRenesseSt}. While the former is a differential notion, the latter relies only on the underlying metric structure and on a reference measure, thus it can be promoted to a \emph{definition} of synthetic curvature-dimension bound for metric measure spaces. Subsequently, it was showed in \cite{MR2546027} that the same relation between the $\cd$ condition and the (flag) Ricci curvature holds in smooth Finsler manifolds. 

While in Riemannian and Finsler geometry, the curvature-dimension condition \`a la Lott--Sturm--Villani is consistent with a lower bound on the Ricci curvature, the same does not hold in \sr and \sF geometry. The latter are a broad generalization of Riemannian and Finsler geometry, where a smoothly varying \emph{norm} is defined on a subset of preferred directions, called distribution. In a series of contributions \cite{juillet2010,juillet2020,magnaboscorossi,rizzistefani}, it has been showed that the $\cd$ condition fails in \sr geometry. More precisely, given a (truly) \sr manifold $M$, equipped with a smooth positive measure $\m$, the metric measure space $(M,\di_{SR},\m)$ does not satisfy the $\cd(K,N)$ condition, for any choice of parameters $K\in\R$ and $N\in (1,+\infty]$. Note that, the positivity of the measure is not merely a technical assumption as there are examples of \sr manifolds, equipped with smooth measures vanishing at some points, that satisfy the $\cd$ condition, see \cite{panwei,Pan23}. In addition, recent developments \cite{borzatashiro,magnabosco2023failure} have shown that the $\cd$ condition fails also in a large class of \sF manifolds, corroborating the general belief that it should fail in \sF geometry. 

These results show that the classical $\cd$ condition is not suitable to study curvature in the setting of sub-Riemannian and \sF manifolds. Motivated by this, Barilari, Mondino and Rizzi \cite{barilari2022unified} introduced and studied a generalized version of the $\cd$ condition, defined for gauged metric measure spaces. They showed in particular that any compact fat \sr manifold, with a suitable gauge, satisfies their version of the $\cd$ condition. 

A different classical curvature-dimension bound, which instead holds in some examples of sub-Rieman\-nian and \sF manifolds, is the so-called \emph{measure contraction property}, or $\MCP(K,N)$ for brevity. This condition was introduced by Ohta in \cite{ohta2007} and it prescribes a control (depending on $K$ and $N$) on the contraction rate of volumes along geodesics. Although the $\MCP(K,N)$ condition is \emph{weaker} than the $\cd(K,N)$ condition, in the setting of (unweighted) Riemannian manifolds, it is equivalent to having Ricci curvature bounded below by $K\in\R$ and dimension equal to $N\in (1,+\infty)$. As mentioned, the $\MCP$ condition holds in a large class of \sr manifolds, including many Carnot groups, cf.\ \cite{MR3110060,MR3502622,MR3848070,MR4245620,borza2022,Nicolussi-Zhang2023}, and in the (sub-Finsler) $\ell^p$-Heisenberg group for $1<p\leq2$, as shown by the first- and fourth-named authors in \cite{borzatashiro}. In particular, for $p=2$, they recovered a result of \cite{juillet2010}, where it was shown that the \sr Heisenberg group, equipped with the Lebesgue measure $\Leb^3$, satisfies $\MCP(0,5)$ with sharp constants.

%The \sr Heisenberg group $\hei$ is a stratified nilpotent Lie group on $\R^3$, where a scalar product is defined on the first layer of its Lie algebra. In this case, the metric measure space $(\hei,\di_{SR},\Leb^3)$ satisfies $\MCP(0,5)$ with sharp constants.    

In this paper, we complete the study initiated in \cite{borzatashiro,magnabosco2023failure}, investigating whether the measure contraction property $\MCP(K,N)$ holds in the \sF Heisenberg group, equipped with a \emph{general} norm. The \sF Heisenberg group is a stratified real Lie group of dimension 3, where a norm $\normdot$ is defined on the first layer of its Lie algebra. In this setting, we prove both positive and negative results, showing that the validity of $\MCP(K,N)$ depends on the \emph{smoothness} and \emph{convexity} properties of the reference norm $\normdot$. This is particularly interesting if compared to what happens in $\R^n$ (with the Lebesgue measure $\Leb^n$), which satisfies $\cd(0,n)$, and thus $\MCP(0,n)$, if equipped with any distance induced by a norm, see the appendix of \cite{villani}.

Firstly, we study in detail the geometry of the \sF Heisenberg group, making use of convex trigonometry (cf. Subsections \ref{sec:convex_trigtrig} and \ref{sec:tre3}), in the way that it was introduced in \cite{lok}. Building upon \cite{Bereszynski,lok2}, we provide a precise and general description of geodesics and their uniqueness: one of our main results (Proposition \ref{prop:summary}) highlights that uniqueness of geodesics depends on how ``flat'' the polar of the unit ball of the reference norm is and extends \cite{breuillard--ledonne2013}. Secondly, this allows us to identify a \sF exponential map and study its regularity properties, in relation to the properties of $\normdot$. The main observation is that the smoothness of the norm $\normdot$ influences the positivity of the Jacobian of the exponential map, while the convexity of $\normdot$ determines its regularity. As the Jacobian of the exponential map controls the infinitesimal volume distortion of geodesics (cf. Subsection \ref{subsec:jacobian}), we use it to develop many useful criteria to address the validity of the $\MCP$ condition in the \sF Heisenberg group.
% A central object in our analysis is the Jacobian of the exponential map, as it controls the infinitesimal volume distortion of geodesics (cf. Subsection \ref{subsec:jacobian}). The main observation is that the regularity of the norm $\normdot$ influences the positivity of the Jacobian, while the convexity of $\normdot$ determines the regularity of the Jacobian. Hinging upon this observation, we develop many useful criteria to address the validity of the $\MCP$ condition in the \sF Heisenberg group. 
Our first result in this direction (Proposition \ref{prop:MCPgeneral}) is valid for any norm and provides a general necessary condition for $\MCP$ to hold, in term of the Jacobian of the exponential map. Our second result (Proposition \ref{prop:phoenix}) shows that, if the reference norm $\normdot$ is $C^1$ and strictly convex, the same condition is also sufficient. Finally, a refined analysis of the differentiability properties of the Jacobian allows us to provide further \emph{differential characterizations} for $\MCP$, if the reference norm $\normdot$ is $C^1$ and strongly convex, cf. Corollary \ref{prop:diffcharacMCP} and Proposition \ref{prop:MCPlimsup}.

Relying on these criteria, we deduce our main result showing the failure of the measure contraction property.

\begin{theorem}\label{thm:MAINnomcp}
     Let $\hei$ be the \sF Heisenberg group, equipped with a norm $\normdot$ which is not $C^1$ or not strongly convex, and let $\m$ be a positive smooth measure on $\hei$. Then, the metric measure space $(\hei, \di, \m)$ does not satisfy the measure contraction property $\MCP(K,N)$ for every $K\in \R$ and $N\in (1,\infty)$.
\end{theorem}

Remarkably, Theorem \ref{thm:MAINnomcp} shows that the properties of the reference norm $\normdot$ do not only influence the optimal constants $K$ and $N$ for which the measure contraction $\MCP(K,N)$ holds, but the validity of the condition itself. 
%Moreover, we highlight that there are two main properties of $\normdot$ causing the failure of $\MCP$, the irregularity and the lack of (strong) convexity. In particular, each of these two properties are reflected in a specific \emph{singular} behavior of geodesics that we can exploit to show the failure of $\MCP$, with two different strategies.
The two main properties of $\normdot$ causing the failure of $\MCP$ are the non-smoothness and the lack of (strong) convexity and, notably, each of these two properties is reflected in a \emph{specific singular behavior} of geodesics that we can exploit to show the failure of $\MCP$, with two different strategies.

Indeed, if the reference norm $\normdot$ is not $C^1$, we show that geodesics can branch, even though they are unique, cf. Theorem \ref{thm:noCDnonC1}. This behavior, which was already observed in \cite{magnabosco2023failure} for strictly convex not $C^1$ norms, has independent interest, as examples of branching spaces usually occur when geodesics are not unique. Instead, when the reference norm $\normdot$ is not strongly convex, we take advantage of the loss of regularity of the Jacobian of the exponential map. The proof of this case is divided into two parts: the first part (Theorem \ref{thm:noMCPnonstrctly}) addresses the scenario where the norm is not even strictly convex, while the second part (Theorem \ref{thm:noMCPC1strictlynotstrongly}) closes the remaining gap. More in details, when the norm is not strictly convex we exploit discontinuity points of the Jacobian and contradict the necessary condition of Proposition \ref{prop:MCPgeneral}. While, if the norm is strictly but not strongly convex, we contradict the equivalent characterization of Proposition \ref{prop:phoenix}, exploiting the fact that the Jacobian, while continuous, fails to be Lipschitz.

%The proof of Theorem \ref{thm:MAINnomcp} is described in Section \ref{sec:nomcp}. In particular, in Subsection \ref{sec:nomcp1} we prove if for irregular reference norms, while in Subsections \ref{sec:nomcp2} and \ref{sec:nomcp3} we deal with not strongly convex norms.

\medskip 

Subsequently, we investigate the validity of the measure contraction property in the \sF Heisenberg group, when the reference norm is \emph{at least} $C^1$ and strongly convex. The central object of our analysis is the \emph{angle correspondence map} $C_\circ:\R\to\R$, which represents the duality map from $(\R^2,\normdot_*)$ to $(\R^2,\normdot)$, interpreted at the level of generalized angles, see Subsection \ref{sec:convex_trigtrig} for a precise definition. According to classical convex geometry,
the duality map corresponds to the differential of the dual norm and,
therefore, the angle correspondence map $C_\circ$ has its regularity tied to the regularity of the norm $\normdot$, see Proposition \ref{prop:regularityCo}. 
In our analysis, the angle correspondence appears in the asymptotic expansion of the Jacobian of the exponential map. Hence, when the norm $\normdot$ is $C^{1,1}$ and strongly convex, we shall observe that the map $C_\circ$ is bi-Lipschitz and we obtain an important positive result (Corollary \ref{cor:C11andstrongly}).

% In Section \ref{sec:afterstrongly} we investigate necessary and/or sufficient condition for the validity of the measure contraction property $\MCP(K,N)$. A remarkable consequence of our findings is the following result. 

\begin{theorem}\label{thm:MAINmcp}
     Let $\hei$ be the \sF Heisenberg group, equipped with a $C^{1,1}$ and strongly convex norm $\normdot$, and let $\Leb^3$ be the Lebesgue measure on $\hei$. Then, the metric measure space $(\hei, \di, \Leb^3)$ satisfies the measure contraction property $\MCP(0,N)$ for some $N\in (1,\infty)$.
\end{theorem}

% Heuristically, as the norm $\normdot$ is $C^{1,1}$ and strongly convex, the Jacobian of the exponential map is positive and Lipschitz, and one can then check the validity of $\MCP$ using the equivalent characterization of Proposition \ref{prop:phoenix}. In practice, we obtain Theorem \ref{thm:MAINmcp} as a byproduct of a refined analysis of the differential of the Jacobian (cf. Section \ref{sec:afterstrongly}).

\noindent Note that, as the Heisenberg group admits a one-parameter family of dilations (cf.\ \cite{MR3283670}), by the scaling property of $\MCP$, it is sufficient to investigate the validity of the measure contraction property with $K=0$.

Combining Theorem \ref{thm:MAINnomcp} and Theorem \ref{thm:MAINmcp}, we obtain an almost complete picture describing the validity of $\MCP$ in the \sF Heisenberg group. However, these results do not cover the case when the norm $\normdot$ is strongly convex and $C^1$ but not $C^{1,1}$. In Section \ref{sec:afterstrongly}, we identify different behaviors for this intermediate case, showing that $\MCP$ can both hold or fail. The findings of Section \ref{sec:afterstrongly} supporting the validity of $\MCP$ can be summarized in the following theorem (cf.\ Theorem \ref{thm:MCPinstrongly}).

\begin{theorem}
\label{thm:stronglysmallo?}
    Let $\hei$ be the \sF Heisenberg group, equipped with a $C^{1}$ and strongly convex norm $\normdot$ and with the Lebesgue measure $\Leb^3$.
    If the derivative of $C_\circ$ is asymptotically and uniformly equivalent to a fractional polynomial at its zero points, see \eqref{eq:fractionalcondition} and \eqref{eq:uniformABalpha},
    then $(\hei,\di,\Leb^3)$ satisfies the $\MCP(0,N)$ for some $N\in(1,\infty)$.
\end{theorem}
% \noindent {\red Theorem \ref{thm:stronglysmallo?} generalizes the validity of $\MCP$ for $\ell^p$-norm with $p\in(1,2)$ given in \cite{borzatashiro}, and is based on a formal asymptotic expansion of the Jacobian of the exponential map (see Proposition \ref{prop:formulaJRdJRint}).
% In \cite{borzatashiro},
% the authors used the (fractional) Taylor polynomial of the generalized $p$-trigonometric functions,
% which is not given in general norm.
% Instead of the expansion based on the trigonometric function,
% we shall use the expansion with integral,
% where the integrand is explicitly written by the differential of $C_\circ$.
% It allows us to apply the characterization of $\MCP$ given in Proposition \ref{prop:diffcharacMCP}.} 
The proof of Theorem \ref{thm:stronglysmallo?} hinges upon the characterization of $\MCP$ given in Proposition \ref{prop:diffcharacMCP} and a novel Taylor expansion of the Jacobian of the exponential map with integral remainder (see Proposition \ref{prop:formulaJRdJRint}). This theorem generalizes the validity of $\MCP$ for $\ell^p$-norm with $p\in(1,2)$ that was obtained in \cite{borzatashiro}. We remark that we indeed obtain Theorem \ref{thm:MAINmcp} as a corollary of Theorem \ref{thm:stronglysmallo?}.

On the one hand, the sufficient conditions identified in Theorem \ref{thm:stronglysmallo?} are not necessary. Indeed, we provide an example (cf. Example \ref{example:monotone}) of \sF Heisenberg group, equipped with a $C^{1}$ and strongly convex norm $\normdot$, satisfying the $\MCP$ condition, but where $C_\circ'$ does not behave as a fractional polynomial. On the other hand, as displayed by Example \ref{example:oscilliation}, removing the uniformity assumption may lead to the failure of the $\MCP$ condition.  

\medskip

An additional relevant byproduct of our study is a lower bound on the so-called \emph{curvature exponent} $N_{\mathrm{curv}}$ of the \sF Heisenberg group. The curvature exponent is the minimal parameter $N$ such that $\hei$ (equipped with the Lebesgue measure $\Leb^3$) satisfies $\MCP(0,N)$. To get an estimate of the curvature exponent, we need a slightly stronger assumption on $C_\circ$ (cf. Theorem \ref{thm:MCPinstronglysmallo}).
\begin{theorem}\label{thm:curvature_exponent}
    Let $\hei$ be the \sF Heisenberg group, equipped with a $C^{1}$ and strongly convex norm $\normdot$ and with the Lebesgue measure $\Leb^3$.
    If the derivative of $C_\circ$ has maximal fractional order $s\geq 0$, see \eqref{eq:smallocondition},
    then $(\hei,\di,\Leb^3)$ satisfies the $\MCP(0,N)$ for some $N\in (1,\infty)$, and $N_{\mathrm{curv}} \geq 2s + 5$.
\end{theorem}

An important consequence of Theorem \ref{thm:curvature_exponent} is that if the reference norm is $C^2$ and strongly convex, then the curvature exponent $N_{\mathrm{curv}}\geq 5$ (cf. Corollary \ref{cor:C11andstrongly}).
However, the lower bound $5$ is not sharp for a general $C^2$ norm.
Indeed, we observe that there is a $C^2$ and strongly convex norm such that $N_{\mathrm{curv}}>5$ (cf. Example \ref{ex:greaterthan5}).
These facts lead us to formulate the following conjecture.

\begin{conj}
    The metric measure space $(\hei,\di,\mathscr{L}^3)$ satisfies $\MCP(0,5)$ if and only if the reference norm is the $\ell^2$-norm, that is to say $(\hei,\di)$ is the sub-Riemannian Heisenberg group.%\todo{Sam supports the stronger conjecture: MCP(0, 5) iff sub-Riemannian Heisenberg}
\end{conj}

% Under the same assumptions of Theorem \ref{thm:stronglysmallo?}, if $C_0'$ has maximal fractional order $s\geq 0$, then the curvature exponent satisfies $N_{\mathrm{curv}}\geq 2s+5$ (cf.\ Theorem \ref{thm:MCPinstronglysmallo}). In particular, we deduce that $N_{\mathrm{curv}} \geq 5$, showing that the curvature exponent is always greater than or equal to the one of the \sr Heisenberg group.

In the last section of this paper we prove the failure of the $\cd$ condition in the \sF Heisenberg group, equipped with a $C^1$ strongly convex norm $\normdot$ and with a smooth measure $\m$ (cf. Theorem \ref{thm:noCDstrongly}). In light of Theorems \ref{thm:MAINmcp} and \ref{thm:stronglysmallo?}, this is a completely non-trivial result as the weaker $\MCP$ condition may hold in this case. Our argument is a substantial refinement of the one presented in \cite{magnabosco2023failure} for $C^{1,1}$ and strictly convex norms, based on an improved analysis of the correspondence map $C_\circ$. Combining the findings of Theorem \ref{thm:noCDstrongly} with Theorem \ref{thm:MAINnomcp}, we obtain the following negative result, valid for \emph{any} reference norm.
 \begin{theorem}\label{thm:MAINnocd}
  Let $\hei$ be the \sF Heisenberg group, equipped with a norm $\normdot$ and with a positive smooth measure $\m$. Then, the metric measure space $(\hei,\di,\m)$ does not satisfy the $\cd(K,N)$ condition, for every $K\in\R$ and $N\in (1,\infty)$.
\end{theorem}

The \sF Heisenberg groups are the unique (up to isometries) \sF Carnot groups with Hausdorff dimension less than $5$ (or with topological dimension less than or equal to $3$), see \cite[Def.\ 10.3]{ABB-srgeom} for a precise definition of Carnot group. Therefore, Theorem \ref{thm:MAINnocd} corroborates the following conjecture, already formulated in \cite{magnabosco2023failure}. 

\begin{conj}\label{conj:carnot}
    Let $G$ be a \sF Carnot group, endowed with a positive smooth measure $\m$. Then, the metric measure space $(G,\di_{SF},\m)$ does not satisfy the $\cd(K,N)$ condition for any $K\in\R$ and $N\in(1,\infty)$.
\end{conj}

 Our interest in Carnot groups stems from the fact that they are the only metric spaces that are locally compact, geodesic, isometrically homogeneous and self-similar (i.e. admitting a dilation), cf. \cite{MR3283670}. According to this property, sub-Finsler Carnot groups naturally arise as (unique) metric tangents of metric measure spaces, as showed in \cite{MR2865538}. As the metric measure tangents of a $\cd(K,N)$ space are $\cd(0,N)$, the study of the $\cd(K,N)$ condition in \sF Carnot groups, and especially the validity of Conjecture \ref{conj:carnot}, has the potential to provide deep insights on the structure of tangents of $\cd(K,N)$ spaces. This could be of significant interest, particularly in connection with Bate's recent work \cite{MR4506771}, which establishes a criterion for rectifiability in metric measure spaces, based on the structure of metric tangents.

\subsection*{Acknowledgments} 
This project has received funding from the European Research Council (ERC) under the European Union’s Horizon 2020 research and innovation program (grant agreement No. 945655). M.M. acknowledges support from the Royal Society through the Newton International Fellowship (award number: NIF$\backslash$R1$\backslash$231659). T.R. acknowledges support from the DFG through the collaborative research centre ``The mathematics of emerging effects'' (CRC 1060, Project-ID 211504053) and  from the ANR-DFG project ``CoRoMo'' (ANR-22-CE92-0077-01). K.T. is partially supported by JSPS KAKENHI grant numbers 18K03298, 19H01786, 23K03104.

\section{Preliminaries}

\subsection{The \texorpdfstring{$\cd(K,N)$}{CD(K,N)} and the \texorpdfstring{$\MCP(K,N)$}{MCP(K,N)} conditions} \label{sec:CD}

A metric measure space is a triple $(\X,\di,\m)$ where $(\X,\di)$ is a complete and separable metric space and $\m$ is a locally finite Borel measure on it. In the following, we denote by $C([0, 1], \X)$ the space of continuous curves from $[0, 1]$ to $\X$. For every $t \in [0, 1]$ we call $e_t \colon C([0, 1], \X) \to \X$ the evaluation map, i.e. $e_t(\gamma) := \gamma(t)$. A curve $\gamma\in C([0, 1], \X)$ is said to be a \textit{geodesic} if 
\begin{equation}
    \di(\gamma(s), \gamma(t)) = |t-s| \cdot  \di(\gamma(0), \gamma(1)) \quad \text{for every }s,t\in[0,1].
\end{equation}
We denote by $\Geo(\X)$ the space of all geodesics on $(\X,\di)$. The metric space $(\X,\di)$ is said to be geodesic if every pair of points $x,y \in \X$ can be connected with a curve $\gamma\in \Geo(\X)$. 
We denote by $\Prob(\X)$ the set of Borel probability measures on $\X$ and by $\Prob_2(\X) \subset \Prob(\X)$ the set of those having finite second moment. We endow the space $\Prob_2(\X)$ with the Wasserstein distance $W_2$, defined by
\begin{equation}
\label{eq:defW2}
    W_2^2(\mu_0, \mu_1) := \inf_{\pi \in \mathsf{Adm}(\mu_0,\mu_1)}  \int \di^2(x, y) \, \de \pi(x, y),
\end{equation}
where $\mathsf{Adm}(\mu_0, \mu_1)$ is the set of all admissible transport plans between $\mu_0$ and $\mu_1$, namely all the measures $\pi \in \Prob(\X\times\X)$ such that $(\p_1)_\sharp \pi = \mu_0$ and $(\p_2)_\sharp \pi = \mu_1$, where $\p_i$, for $i=1,2
$, is the projection onto the $i$-th factor. The metric space $(\Prob_2(\X),W_2)$ is itself complete and separable, moreover, if $(\X,\di)$ is geodesic, then $(\Prob_2(\X),W_2)$ is geodesic as well. In this case, every geodesic $(\mu_t)_{t\in [0,1]}$ in $(\Prob_2(\X),W_2)$ can be represented with a measure $\eta \in \Prob(\Geo(\X))$, i.e.\ $\mu_t = (e_t)_\# \eta$.

\paragraph{The \texorpdfstring{$\cd(K,N)$}{CD(K,N)} condition.} 

We present the curvature-dimension condition, or $\cd(K,N)$ for brevity, firstly introduced by Sturm and Lott--Villani in \cite{sturm2006-1,sturm2006,lott--villani2009}. For every $K \in \R$, $N\in (1,\infty)$ and $t\in [0,1]$, the \emph{distortion coefficients} are the functions:

\begin{equation}\label{eq:tau}
    \tau_{K,N}^{(t)}(\theta):=t^{\frac{1}{N}}\left[\sigma_{K, N-1}^{(t)}(\theta)\right]^{1-\frac{1}{N}},\qquad\forall\,\theta\geq 0
\end{equation}
where
\begin{equation}
\sigma_{K,N}^{(t)}(\theta):= 
\begin{cases}

\displaystyle  +\infty & \textrm{if}\  N\pi^{2}\leq K\theta^{2}, \crcr
\displaystyle  \frac{\sin(t\theta\sqrt{K/N})}{\sin(\theta\sqrt{K/N})} & \textrm{if}\  0 < K\theta^{2} < N\pi^{2}, \crcr
t & \textrm{if}\ %K \theta^{2}<0 \ \textrm{and}\ N=0, \ \textrm{or  if}\ 
K =0,  \crcr
\displaystyle   \frac{\sinh(t\theta\sqrt{-K/N})}{\sinh(\theta\sqrt{-K/N})} & \textrm{if}\ K < 0.
\end{cases}
\end{equation}

\begin{definition}[$\cd(K,N)$ condition]\label{def:CD}
A metric measure space $(\X,\di,\m)$ is said to be a $\cd(K,N)$ \emph{space} (or to satisfy the $\cd(K,N)$ condition) if for every pair of measures $\mu_0=\rho_0\m,\mu_1= \rho_1 \m \in \Prob_2(\X)$, absolutely continuous with respect to $\m$, there exists a $W_2$-geodesic $(\mu_t)_{t\in [0,1]}$ connecting them and induced by $\eta \in \Prob(\Geo(\X))$, such that $\mu_t =\rho_t \m \ll \m$ for every $t\in [0,1]$ and the following inequality holds for every $N'\geq N$ and every $t \in [0,1]$
\begin{equation}\label{eq:CDcond}
    \int_\X \rho_t^{1-\frac 1{N'}} \de \m \geq \int_{\X \times \X} \Big[ \tau^{(1-t)}_{K,N'} \big(\di(x,y) \big) \rho_{0}(x)^{-\frac{1}{N'}} +    \tau^{(t)}_{K,N'} \big(\di(x,y) \big) \rho_{1}(y)^{-\frac{1}{N'}} \Big]    \de\pi( x,y),
\end{equation}
where $\pi= (e_0,e_1)_\# \eta$.
\end{definition}

\paragraph{The Brunn--Minkowski inequality.} One of the main merits of the $\cd(K,N)$ condition is that it is sufficient to deduce geometric and functional inequalities that hold in the smooth setting. An example, which is particularly relevant to this paper, is the so-called Brunn--Minkowski inequality, whose definition in a metric measure space requires the notion of midpoints.

\begin{definition}[Midpoints]
     Let $(\X,\di)$ be a metric space and let $A,B \subset \X$ be two Borel subsets. Then for $t\in (0,1)$, we define the set of $t$-\emph{midpoints} between $A$ and $B$ as
    \begin{align}
        M_t(A,B) := 
        \{
            x \in \X \, : \, x = \gamma(t)
                \, , \, 
            \gamma \in \Geo(\X)
                \, , \, 
            \gamma(0) \in A  
                \ \text{and} \  
            \gamma(1) \in B
        \} 
            \, .
    \end{align}
\end{definition}

\begin{definition}[Brunn--Minkowski inequality]
    \label{def:bruno}
    Given $K \in \R$ and $N\in (1,\infty)$, we say that a metric measure space $(\X,\di, \m)$ satisfies the \emph{Brunn--Minkowski inequality} $\bm(K,N)$ if, for every nonempty $A,B \subset \text{spt}(\m)$ Borel subsets and every $t \in (0,1)$, we have
        \begin{align}   \label{eq:bm}
            \m \big(M_t(A,B)\big) \big)^ \frac{1}{N} 
                \geq 
            \tau_{K,N}^{(1-t)} (\Theta(A,B)) \cdot \m(A)^ \frac{1}{N} 
                + 
            \tau_{K,N}^{(t)} (\Theta(A,B)) \cdot \m(B)^ \frac{1}{N}
                \, ,
        \end{align}
     where 
    \begin{align}   \label{eq:def_Theta}
        \Theta(A,B):=
        \left\{
        \begin{array}{ll}\displaystyle    \inf_{x \in A,\, y \in B} \di(x, y) 
                & \text { if } K \geq 0 \, , \\ 
        \displaystyle
            \sup _{x \in A,\, y \in B} \di(x, y) 
                & \text { if } K<0 \, .
        \end{array}
        \right. 
    \end{align}
\end{definition}

As already mentioned, the Brunn--Minkowski inequality is a consequence of the $\cd(K,N)$ condition, in particular we have that 
\begin{equation*}
    \cd(K,N)\quad \implies\quad \bm(K,N),
\end{equation*}
for every $K\in \R$ and every $N\in (1,\infty)$. In Section \ref{sec:noCD}, we will prove the failure of the $\cd(K,N)$ condition for every choice of the parameters $K\in \R$ and $N\in (1,\infty)$, by contradicting the Brunn--Minkowski inequality $\bm(K,N)$. A priori, this is a stronger result than the one stated in Theorem \ref{thm:MAINnocd}, as in principle the Brunn--Minkowski inequality is weaker than the $\cd(K,N)$ condition. However, recent developments (cf. \cite{seminalpaper,Magnabosco-Portinale-Rossi:2022b}) suggest that the Brunn--Minkowski $\bm(K,N)$ could be equivalent to the $\cd
(K,N)$ condition in a wide class of metric measure spaces. 

\paragraph{The \texorpdfstring{$\MCP(K,N)$}{MCP(K,N)} condition.}

A way to relax the condition of $\cd(K,N)$ involves requiring it only when the first marginal degenerates to $\delta_x$, a delta-measure at $x\in\text{spt}(\m)$, and the second marginal is $\frac{\m|_A}{\m(A)}$, for some Borel set $A\subset\X$ with $0<\m(A)<\infty$.
This is the idea behind behind the so-called measure contraction property, introduced by Ohta in \cite{ohta2007}. 

\begin{definition}[$\mathsf{MCP}(K,N)$ condition]
\label{def:mcp}
    Given $K\in\R$ and $N\in (1,\infty)$, a metric measure space $(\X,\di,\m)$ is said to satisfy the \emph{measure contraction property} $\mathsf{MCP}(K,N)$ if for every $x\in\text{spt}(\m)$ and every Borel set $A\subset\X$ with $0<\m(A)<\infty$, there exists a $W_2$-geodesic induced by $\eta \in \Prob(\Geo(\X))$ connecting $\delta_x$ and $\frac{\m|_A}{\m(A)}$ such that, for every $t\in[0,1]$,
    \begin{equation}
    \label{eq:mcp_def}
        \frac{1}{\m(A)}\m\geq(e_t)_\#\Big(\tau_{K,N}^{(t)}\big(\di(\gamma(0),\gamma(1))\big)^N\eta(\text{d}\gamma)\Big).
    \end{equation}
\end{definition}

\noindent The $\MCP(K,N)$ condition is weaker than the $\cd(K,N)$ one, i.e. 
\begin{equation*}
    \cd(K,N)\quad \implies \quad\MCP(K,N),
\end{equation*}
for every $K\in \R$ and every $N\in (1,\infty)$, cf. \cite[Lem.\ 6.13]{MR4309491}.

\begin{remark}
\label{rmk:SIUUUUUUU}
Let us recall a useful equivalent formulation of the inequality \eqref{eq:mcp_def}, which holds whenever geodesics are unique, we refer the reader to \cite[Lem.\ 2.3]{ohta2007} for further details. Consider $x\in\text{spt}(\m)$ and a Borel set $A\subset\X$ with $0<\m(A)<\infty$. Assume that for every $y\in A$, there exists a unique geodesic $\gamma_{x,y}:[0,1]\to \X$ joining $x$ and $y$. Then, \eqref{eq:mcp_def} is verified for the marginals $\delta_x$ and $\frac{\m|_A}{\m(A)}$ if and only if 
\begin{equation}
\label{eq:tj_pantaloncini}
    \m\big(M_t(\{x\},A'))\big)\geq \int_{ A'}\tau_{K,N}^{(t)}(\di(x,y))^N \de\m(y),\qquad\text{for any Borel set }A'\subset A.
\end{equation}
\end{remark}

We recall below the definition of the curvature exponent.

\begin{definition}[Curvature exponent] \label{def:curvatureexponent}
    Let $(\X,\di,\m)$ be a metric measure space satisfying the $\MCP(0,N)$ for some $N\in (1,+\infty)$. The \emph{curvature exponent} of $\X$ is defined as 
    \begin{equation}
        N_{\mathrm{curv}}:=\inf\{N\in  (1,+\infty):(\X,\di,\m) \text{ is } \MCP(0,N)\}.
    \end{equation}
\end{definition}

\paragraph{Scaling and stability properties.} The $\cd(K,N)$ condition and the measure contraction property $\MCP(K,N)$ enjoy several properties that validate them as synthetic curvature dimension bounds. Among them, we only mention the ones necessary for our purposes:
\begin{itemize}
    \item \textit{scaling property} (cf.\ \cite{sturm2006}): If $(\X,\di,\m)$ is a $\cd(K,N)$ (resp. $\MCP(K,N)$) space, for every $\alpha,\beta>0$ the scaled space $(\X,\alpha \di,\beta \m)$ is a $\cd(\alpha^{-2}K,N)$ (resp. $\MCP(\alpha^{-2} K,N)$) space.
    \item \textit{$\mathrm{pmGH}$-stability} (cf.\ \cite{GigMonSav}) Let $\{(\X_n,\di_n, \m_n,p_n)\}_{n\in \N}$ be a sequence of pointed metric measure spaces (i.e.\ metric measure spaces with a distinguished point) converging to $(\X_\infty,\di_\infty,\m_\infty,p_\infty)$ in the pointed measured Gromov-Hausdorff convergence. Assume that for every $n\in\N$, $(\X_n,\di_n, \m_n)$ is a $\cd(K_n,N_n)$ (resp. $\MCP(K_n,N_n)$) space, for two sequences $(K_n)_{n\in \N} \subset \R$ and $(N_n)_{n\in \N} \subset (1,\infty)$ converging to $K_\infty \in \R$ and $N_\infty\in (1, \infty)$, respectively. Then, the limit space $(\X_\infty,\di_\infty,\m_\infty)$ is a $\cd(K_\infty,N_\infty)$ (resp. $\MCP(K_\infty,N_\infty)$) space.
\end{itemize}

\subsection{Convex trigonometry}
\label{sec:convex_trigtrig}

In this section, we recall the definition and main properties of the convex trigonometric functions, firstly introduced in \cite{lok}. 
%In this subsection we follow the presentation of \cite{lok2}.
Let $\Omega\subset \R^2$ be a convex, compact set, such that $O:=(0,0)\in \text{Int} (\Omega)$ and denote by $\pi_{\Omega}$ its surface area. 

\begin{definition}[Convex trigonometric functions]
    Let $\theta\in\R$ denote a generalized angle. If $0\leq \theta < 2 \pi_{\Omega}$ define $P_\theta$ as the point on the boundary of $\Omega$, such that the area of the sector of $\Omega$ between the rays $Ox$ and $OP_{\theta}$ is $\frac{1}{2}\theta$ (see Figure \ref{fig:convextrig1}). Moreover, define $\sinom(\theta)$ and $\cosom(\theta)$ as the coordinates of the point $P_\theta$, i.e.
    \begin{equation*}
        P_\theta = \big( \cosom(\theta), \sinom(\theta) \big).
    \end{equation*}
    Finally, extend these trigonometric functions outside the interval $[0,2\pi_{\Omega})$ by periodicity (of period $2 \pi_{\Omega}$), so that for every $k\in \mathbb Z$
    \begin{equation*}
        \cosom(\theta)= \cosom(\theta+2k \pi_{\Omega}), \quad \sinom(\theta)= \sinom(\theta+2k \pi_{\Omega}) \quad \text{and}\quad P_\theta = P_{\theta +2k\pi_{\Omega}}.
    \end{equation*}
\end{definition}
\noindent In particular, the maps $P,\sinom,\cosom$ are well-defined on the quotient $\R/ 2 \pi_\Omega \mathbb Z$.

 Observe that by definition $\sinom(0)=0$ and that when $\Omega$ is the Euclidean unit ball we recover the classical trigonometric functions.

\begin{lemma}
\label{lem:facile_cazzo}
    The map $P:\R/ 2 \pi_\Omega \mathbb Z \to  \partial \Omega \subset \R^2$ that associate to the angle $\theta$ the vector $P_\theta$, is bi-Lipschitz.
\end{lemma}

\begin{proof}
    First of all, observe that, by convexity of $\Omega$, the map $P$ is bijective and thus invertible. %To prove the lemma is then sufficient to show that $P$ is bi-Lipschitz.
    We now prove that $P$ is Lipschitz. 
   % For every $\theta_1 \neq \theta_2 \in [0,2 \pi_{\Omega})$ let $l_{\theta_1,\theta_2}$ is the line joining $P_{\theta_1}$ and $P_{\theta_2}$, observe that if 
    By compactness of $\partial \Omega$, we can find $\epsilon>0$ and a positive constant $K$ such that, if $\theta_1, \theta_2 \in \R$ with $0<|\theta_1-\theta_2|<\epsilon$, we have 
    \begin{equation}
        \di_{eu}(O,l_{\theta_1,\theta_2}) \geq K ,
    \end{equation}
    % To prove this fact, we pursue the following strategy: 
    % 1. we find an angle $\vartheta>0$ such that if $|\theta_1-\theta_2|<\vartheta$, then $P_{\theta_1}, O$ and $P_{\theta_2}$ are not aligned.
    % 2. we take $\epsilon:=\vartheta/2$ and we let $K:=\min \di_{eu}(O,l_{\theta_1,\theta_2})$ where the minimum is taken over the angles $\theta_1, \theta_2 \in [0,2 \pi_{\Omega})$ with $0<|\theta_1-\theta_2|<\epsilon$. 
    % 3. we must have $K>0$ otherwise, if $K=0$, the angles realizing the minimum (which exist by compactness) would have the corrisponding points aligned. 
    where $l_{\theta_1,\theta_2}$ is the line joining $P_{\theta_1}$ and $P_{\theta_2}$. Then, given any $\theta_1, \theta_2 \in \R$ such that $0<|\theta_1-\theta_2|<\epsilon$, we can deduce that  
    \begin{equation}
        \frac{1}{2} |\theta_2-\theta_1 |\geq \Leb^2\Big(\triangle P_{\theta_1} O P_{\theta_2}\Big) \geq \frac{1}{2} K \cdot \norm{P_{\theta_2}-P_{\theta_1}}_{eu},
    \end{equation}
    where $\triangle P_{\theta_1}OP_{\theta_2}$ denotes the triangle of vertices $P_{\theta_1}, O$ and $P_{\theta_2}$. This proves that $P$ is locally $K$-Lipschitz everywhere, and thus $K$-Lipschitz, since $\R/ 2 \pi_\Omega \mathbb Z$ is compact. 

    We are left to show that $P^{-1}$ is Lipschitz. Let $\theta_1 \neq \theta_2 \in \R$ such that $\angle P_{\theta_1}O P_{\theta_2} <  \frac \pi 2$, where $\angle P_{\theta_1}O P_{\theta_2}$ is the Euclidean angle between $O P_{\theta_1}$ and $O P_{\theta_2}$. In this case, we consider the quantities 
    \begin{equation}
        r:= \min\{  \norm{x}_{eu}: x\in \partial \Omega\} \quad \text{and} \quad  R:= \max\{ \norm{x}_{eu}: x\in \partial \Omega\},
    \end{equation}
    and observe that the section of $\Omega$ between the rays $OP_{\theta_1}$ and $OP_{\theta_2}$ is contained in the triangle $ \frac{2R}{r} \cdot \triangle P_{\theta_1} O P_{\theta_2}$. In fact, every point in the line segment joining $ \frac{2R}{r} P_{\theta_1}$ and $ \frac{2R}{r} P_{\theta_2}$ has Euclidean norm strictly bigger than $R$. Hence, we deduce that, for every $\theta_1 \neq \theta_2 \in \R$ such that $\angle P_{\theta_1}O P_{\theta_2} < \frac \pi 2$, we have
    \begin{equation}
        \frac{1}{2} |\theta_2-\theta_1 |\leq \Leb^2\bigg(  \frac{2R}{r} \cdot \triangle P_{\theta_1} O P_{\theta_2}\bigg) =  \frac{4R^2}{r^2}  \Leb^2\Big( \triangle P_{\theta_1} O P_{\theta_2}\Big)  \leq  \frac{2R^2}{r^2}  R \cdot \norm{P_{\theta_2}-P_{\theta_1}}_{eu} .
    \end{equation}
    Therefore, we can conclude that the map $P^{-1}$ is locally $\frac{4 R^3}{r^2}$-Lipschitz, thus it is also $\frac{4 R^3}{r^2}$-Lipschitz.
\end{proof}

\begin{remark}
\label{rmk:Lip}
    Since the projection
    $\R^2 \ni (x,y)\mapsto x \in \R$ (resp. $\R^2 \ni(x,y)\mapsto y\in \R$) is Lipschitz,
    we deduce that the generalized trigonometric function $\cosom$ (resp. $\sinom$) is Lipschitz continuous.
\end{remark}

\begin{figure}[ht]
    \begin{minipage}[c]{.47\textwidth}

    \centering
    \begin{tikzpicture}[scale=0.8]

    \draw[white](1.1,3.25)--(0.92142,4.75);
    \fill[color=black!10!white](3,0)--(0,0)--(1.1,3.25)--(3.6,4);
    \fill[white](3,0) .. controls (3.6,2.8) and (2.4,3.8)..(0,3)--(0,4)--(4,4);
    \draw[->] (-4,0)--(4,0);
    \draw[->] (0,-4)--(0,4);
    \draw[very thick] (3,0) .. controls (3.6,2.8) and (2.4,3.8)..(0,3)..controls (-1.2,2.6) and (-2,2.4).. (-2.6,0)..controls (-3.4,-3) and (-2,-3.4).. (0,-3.2)..controls (1.4,-3) and (2.4,-2.4).. (3,0);
    \draw[dotted,blue,thick] (1.1,3.25) --(0,3.25);
    \draw[very thick] (1.1,3.25) --(0,0);
    \draw[very thick] (3,0) --(0,0);
    \draw[very thick,blue](0,3.25)--(0,0);
    \draw[very thick,red](1.1,0)--(0,0);
    \draw[dotted,red,thick] (1.1,3.25) --(1.1,0);
    \filldraw[black] (0,0) circle (1.5pt);
    \filldraw[blue] (0,3.25) circle (1.5pt);
    \filldraw[red] (1.1,0) circle (1.5pt);
    %\filldraw[black] (0.5,1.47) circle (1.5pt);
    \filldraw[black] (1.1,3.25) circle (1.5pt);

    \node at (-0.9,2)[label=south:${\color{blue}{\sin_\Omega(\theta)}}$] {};
    \node at (0.9,0)[label=south:${\color{red}{\cos_\Omega(\theta)}}$] {};
    \node at (-0.3,0.1)[label=south:$O$] {};
    \node at (2,2)[label=south:$\frac 12 \theta$] {};
    \node at (-2,-1.8)[label=south:$\Omega$] {};
    \node at (1.2,4.2)[label=south:$P_\theta$] {};

    \end{tikzpicture}
    \caption{Values of the generalized trigonometric functions $\cosom$ and $\sinom$.}
    \label{fig:convextrig1}
    
\end{minipage}%
\hfill
\begin{minipage}[h]{.47\textwidth}

    \centering
    \begin{tikzpicture}[scale=0.8]

    \fill[color=black!10!white](3,0)--(0,0)--(1.1,3.25)--(3.6,4);
    \fill[white](3,0) .. controls (3.6,2.8) and (2.4,3.8)..(0,3)--(0,4)--(4,4);
    \draw[->] (-4,0)--(4,0);
    \draw[->] (0,-4)--(0,4);
    \draw[very thick] (3,0) .. controls (3.6,2.8) and (2.4,3.8)..(0,3)..controls (-1.2,2.6) and (-2,2.4).. (-2.6,0)..controls (-3.4,-3) and (-2,-3.4).. (0,-3.2)..controls (1.4,-3) and (2.4,-2.4).. (3,0);
    \draw[very thick] (1.1,3.25) --(0,0);
    \draw[very thick] (3,0) --(0,0);
    \filldraw[black] (0,0) circle (1.5pt);
    %\filldraw[black] (0.5,1.47) circle (1.5pt);
    \filldraw[black] (1.1,3.25) circle (1.5pt);
    \draw(1.1,3.25)--(-1,3);
    \draw(1.1,3.25)--(3.2,3.5);
    \draw[very thick, ->](1.1,3.25)--(0.92142,4.75);

    \node at (-0.3,0.1)[label=south:$O$] {};
    \node at (2,2)[label=south:$\frac 12 \theta$] {};
    \node at (-2,-1.8)[label=south:$\Omega$] {};
    \node at (1.5,3.3)[label=south:$P_\theta$] {};
    \node at (1.5,4.7)[label=south:$Q_{\psi}$] {};

    \end{tikzpicture}
   \caption{Representation of the correspondence $\theta\xleftrightarrow{\Omega} \psi$.}
    \label{fig:convextrig2}
    
\end{minipage}
\end{figure}

Consider now  the polar set:
\begin{equation*}
    \Omega^\circ := \{p\in \R^2\, :\, \langle p,x\rangle\leq 1 \text{ for every }x\in \Omega\},
\end{equation*}
which is itself a convex, compact set such that $O\in \text{Int}(\Omega^\circ)$. Therefore, we can consider the trigonometric functions $\sinomp$ and $\cosomp$. Observe that, by definition of polar set, it holds that 
\begin{equation}
\label{eq:NO_PYTHAGOREAN_IDENTITY_COMETIPERMETTI}
    \cosom(\theta) \cosomp(\phi) + \sinom(\theta)\sinomp(\phi)\leq 1,\qquad \text{for every } \theta,\phi\in \R.
\end{equation}

\begin{definition}[Correspondence]
    We say that two angles $\theta,\phi\in \R$ \emph{correspond} to each other and write $\theta \xleftrightarrow{\Omega} \phi$ if the vector $Q_\phi:= (\cosomp(\phi),\sinomp(\phi))$ determines a half-plane containing $\Omega$ (see Figure \ref{fig:convextrig2}).
\end{definition}

By the bipolar theorem \cite[Thm.\ 14.5]{Rockafellar+1970}, it holds that $\Omega^{\circ \circ}=\Omega$ and this allows to prove the following symmetry property for the correspondence just defined.

\begin{prop}\label{prop:correspondence}
    Let $\Omega\subset \R^2$ be a convex and compact set, with $O\in \text{Int} (\Omega)$. Given two angles $\theta,\phi\in \R$, $\theta\xleftrightarrow{\Omega} \phi$ if and only if $\phi\xleftrightarrow{\Omega^\circ} \theta$.
    Moreover, the following analogous of the Pythagorean equality holds:
    \begin{equation}\label{eq:pytagorean}
        \theta\xleftrightarrow{\Omega} \phi \qquad \text{ if and only if }\qquad \cosom(\theta) \cosomp(\phi) + \sinom(\theta)\sinomp(\phi)= 1.
    \end{equation}
\end{prop}

\noindent The correspondence $\theta\xleftrightarrow{\Omega} \phi$ is not one-to-one in general, in fact if the boundary of $\Omega$ has a corner at the point $P_\theta$, the angle $\theta$ corresponds to an interval of angles (in every period). Nonetheless, we can define a monotone multi-valued
map $C^\circ$ that maps an angle $\theta$ to the maximal closed interval containing angles corresponding to $\theta$ i.e. $\theta\xleftrightarrow{\Omega} C^\circ(\theta)$.
This function has the following periodicity property:
\begin{equation}
\label{eq:periodicity_Ccirc}
    C^\circ(\theta+2\pi_{\Omega} k)=  C^\circ(\theta) +2\pi_{\Omega^\circ} k \qquad \text{ for every }k\in \mathbb Z,
\end{equation}
where $\pi_{\Omega^\circ}$ denotes the surface area of $\Omega^\circ$. Analogously, we can define the map $C_\circ$ associated to the correspondence $\phi\xleftrightarrow{\Omega^\circ} \theta$, and it satisfies an analogue of \eqref{eq:periodicity_Ccirc}. Note that $C_\circ$ and $C^\circ$ are monotone and multi-valued maps, thus their composition is monotone and multi-valued as well. In particular, $C_\circ \circ C^\circ(\theta)$ is an interval containing $\theta$, according to Proposition \ref{prop:correspondence}. The analogous property holds for the composition $C^\circ \circ C_\circ(\phi)$, and, for the sequel, we define the functions $\delta_{\pm}:\R/2\pi_{\Omega^\circ}\mathbb Z\to [0,\infty)$ so that
\begin{equation}\label{eq:interval}
    C^\circ \circ C_\circ(\phi)=[\phi-\delta_-(\phi),\phi+\delta_+(\phi)].
\end{equation}
Observe that the set $\{Q_\psi:\psi\in [\phi-\delta_-(\phi),\phi+\delta_+(\phi)]\}$ is the maximal segment of $\partial \Omega^\circ$ containing $Q_\phi$.

\begin{prop}\label{prop:difftrig}
    Let $\Omega\subset\R^2$ as above. The trigonometric functions $\sinom$ and $\cosom$ are differentiable almost everywhere (cf. Remark \ref{rmk:Lip}). Their differentiability points coincide with the angles $\theta$ where $C^\circ$ is single-valued and it holds that
    \begin{equation*}
        \sinom'(\theta)= \cosomp(C^\circ(\theta)) \qquad \text{and}\qquad \cosom'(\theta)= - \sinomp(C^\circ(\theta)).
    \end{equation*}
    Naturally, the analogous result holds for the trigonometric functions $\sinomp$ and $\cosomp$.
\end{prop}

According to the previous proposition, letting $\mathsf{D}_0 \subset \R/2\pi_{\Omega^\circ}\mathbb{Z}$ be the set of differentiability points of $\sinomp$ and $\cosomp$, $\mathsf{D}_0$ is a $\Leb^1$-full-measure set and corresponds to the set where $C_\circ$ is a single-valued map. From now on, in order to ease the notation, we will sometimes use the shorthand:
\begin{equation}
    \quad \phi_\circ = C_\circ(\phi),
\end{equation}
for the angles $\phi\in \mathsf{D}_0$, where the correspondence map $C_\circ$ is single-valued.

\subsection{Convex trigonometry associated with a norm}\label{sec:tre3}

In the following, we study the convex trigonometry associated with the closed unit ball of a norm $\normdot$ on $\R^2$, i.e. $\Omega:= \bar B^{\norm{\cdot}}_1(0)$. In this case, the polar set $\Omega^\circ$ is the closed unit ball $\bar B^{\norm{\cdot}_*}_1(0)$ of the dual norm $\normdot_*$.

We say that a function $f:\R^2\to \R$ is \emph{strictly convex} if, for any $x,y\in\R^2$ such that $x\neq y$, 
\begin{equation}
    f(tx+(1-t)y)<tf(x)+(1-t)f(y),\qquad\forall\, t\in (0,1).
\end{equation}
Furthermore, let $\normdot:\R^2\to \R_{\geq 0}$ be a norm, then we say that $f$ is \emph{strongly convex} with respect to $\normdot$ if there exists $\alpha>0$ such that, for every $x,y\in\R^n$, 
\begin{equation}
    f(tx+(1-t)y)\leq tf(x)+(1-t)f(y)-\frac{\alpha}{2}t(1-t)\norm{x-y}^2,\qquad\forall\,t\in[0,1].
\end{equation}

Let $\normdot:\R^2\to\R_{\geq 0}$ be a norm on $\R^2$ and define $f_\Omega:\R^2\to \R_{\geq 0}$ to be the function given by $f_\Omega(x):=\frac{1}{2}\|x\|^2$. Similarly, we define $f_{\Omega^\circ}:\R^2\to \R_{\geq 0}$ as $f_{\Omega^\circ}(x):=\frac{1}{2}\|x\|_*^2$, where $\normdot_*$ is the dual norm. 

\begin{definition}[Strictly and strongly convex norm]
    We say that $\normdot$ is \emph{strictly convex} if the function $f_\Omega$ is strictly convex. Similarly, we say that $\normdot$ is \emph{strongly convex} if $f_{\Omega}$ is strongly convex with respect to $\normdot$. 
\end{definition}

Note that, according to \cite[Prop.\ 1.6]{MR1079061}, $\normdot$ is strictly convex if and only if the associated unit ball $\Omega=\bar B^{\norm{\cdot}}_1(0)$ is a strictly convex set. Whereas, \cite[Prop.\ 2.11]{MR1079061} implies that $\normdot$ is strongly convex if and only if the associated unit ball $\Omega$ is a uniformly convex set. We recall below a well-known result on the relation between a norm $\normdot$ and its dual norm $\normdot_*$, cf. \cite[Prop.\ 2.6]{MR1623472}.

\begin{prop}\label{prop:propunderduality}
    Let $\normdot:\R^2\to\R_{>0}$ be a norm on $\R^2$, and let $\normdot_*$ be its dual norm, then:
    \begin{itemize}
        \item [(i)] $\normdot$ is a strictly convex norm if and only if $\normdot_*$ is a $C^1$ norm, i.e.  $f_{\Omega^\circ}$ is $C^1$;
        \item [(ii)] $\normdot$ is a strongly convex norm if and only if $\normdot_*$ is a $C^{1,1}$ norm, i.e.  $f_{\Omega^\circ}$ is $C^{1,1}$.
    \end{itemize}
\end{prop}

Furthermore, we can relate the regularity of the angle correspondence $C_\circ$ with the regularity of the norm. Define the map $Q:\R/ 2 \pi_{\Omega^\circ} \mathbb Z \to  \partial \Omega^\circ \subset \R^2$ that associate to the angle $\phi$ the vector $Q_\phi$ (this is the analogous map of the one defined in Lemma \ref{lem:facile_cazzo} for $\Omega$). 

\begin{lemma}\label{lem:subdifferential}
    Let $\normdot$ be a norm and let $\phi\in \R/2\pi_{\Omega^\circ}\mathbb{Z}$ be an angle such that $Q_\phi\in\partial\Omega^\circ$ is a differentiability point of $\normdot_*$. Then, the angle correspondence map $C_\circ$ from $\Omega^\circ$ to $\Omega$ satisfies
    \begin{equation}
    C_\circ(\phi)=P^{-1}\circ \diff\normdot_*\circ Q(\phi)=P^{-1}\circ \diff_{Q_\phi}\normdot_*.
    \end{equation}
\end{lemma}

\begin{proof}
    According to the Pythagorean identity \eqref{eq:pytagorean}, $\phi\xleftrightarrow{\Omega^\circ} \theta=C_\circ(\phi)$ if and only if $P_\theta$ is a dual vector of $Q_\phi$. Thus, if $Q_\phi$ is a differentiability point of $\normdot_*$, \cite[Lem.\ 3.6]{magnabosco2023failure} ensures that
    \begin{equation}\label{eq:correspondence}
        (\cosom(\theta),\sinom(\theta))=P_\theta = \diff_{Q_\phi} \normdot_*.
    \end{equation}
    The thesis follows from the definition of the maps $P,Q$, recalling also that $P$ is invertible.
\end{proof}

\begin{prop}\label{prop:regularityCo}
     Let $\normdot$ be a norm on $\R^2$ and $C_\circ$ be the angle correspondence map from $\Omega^\circ$ to $\Omega$, then:
    \begin{itemize}
        \item[(i)] $\normdot$ is a $C^1$ norm if and only if $C_\circ$ is strictly increasing,
        \item[(ii)] $\normdot$ is a strictly convex norm if and only if $C_\circ$ is single valued at every angle and continuous,
        \item[(iii)] $\normdot$ is a strongly convex norm if and only if  $C_\circ$ is Lipschitz continuous.
    \end{itemize}
\end{prop}

\begin{proof}
     $\textsl{(i)}$ According to Proposition \ref{prop:propunderduality}, the norm $\normdot$ is $C^1$ if and only if the dual norm $\normdot_*$ is strictly convex. The equivalence between strict convexity of the reference set (in this case $\Omega^\circ$) and strict monotonicity of the angle correspondence has been already observed in \cite[Sec.\ 3]{lok}. \\
     $\textsl{(ii)}$ According to Proposition \ref{prop:propunderduality}, the norm $\normdot$ is strictly convex if and only if the dual norm $\normdot_*$ is $C^1$. This is equivalent to asking that $\diff\normdot_*$ is continuous. The thesis follows from Lemma \ref{lem:subdifferential} and Lemma \ref{lem:facile_cazzo}.\\
     $\textsl{(iii)}$ The thesis can be proven similarly to item $\textsl{(ii)}$, observing that the map $C_\circ=P^{-1}\circ \diff\normdot_* \circ Q$ is Lipschitz if and only if $\diff\normdot_*$ is Lipschitz, as a consequence of Lemma \ref{lem:facile_cazzo}.  
\end{proof}

Finally, since in this case the set $\Omega$ is symmetric with respect to the origin, we have the following identities for the generalized trigonometric functions
\begin{equation}\label{eq:symmetry1}
    \cosomp(\phi \pm \pi_{\Omega^\circ})=-\cosomp(\phi) \quad \text{ and }\quad \sinomp(\phi \pm \pi_{\Omega^\circ})=-\sinomp(\phi)\qquad \forall \phi\in \R.
\end{equation}
The same symmetry guarantees the following property of the correspondence $C^\circ$:
\begin{equation}\label{eq:symmetry2}
    C_\circ(\phi \pm \pi_{\Omega^\circ})= C_\circ (\phi) \pm \pi_{\Omega}.
\end{equation}

\begin{remark}
    All the properties listed above, namely Lemma \ref{lem:subdifferential}, Proposition \ref{prop:regularityCo}, as well as the analogous of \eqref{eq:symmetry1} and \eqref{eq:symmetry2}, hold for $\cosom$, $\sinom$ and $C^\circ$.
\end{remark}

\section{The \sF geometry of the Heisenberg group}
\label{sec:saymyname}

We present here the \sF Heisenberg group and study its geodesics. Let us consider the Lie group $M=\R^3$, equipped with the non-commutative group law, defined by
\begin{equation}
    (x, y, z) \star (x', y', z') = \bigg(x+x',y+y',z+z'+\frac12(xy' - x'y)\bigg),\qquad\forall\,(x, y, z), (x', y', z')\in\R^3,
\end{equation}
with identity element $\e=(0,0,0)$. We define the left-invariant vector fields
\begin{equation}
    X_1:=\partial_x-\frac{y}2\partial_z,\qquad X_2:=\partial_y+\frac{x}2\partial_z.
\end{equation}
The associated distribution of rank $2$ is $\dis:=\text{span}\{X_1,X_2\}$. It can be easily seen that $\dis$ is bracket-generating. Then, letting $\normdot:\R^2\to\R_{\geq 0}$ be a norm, the \emph{\sF Heisenberg group} $\hei$ is the Lie group $M$ equipped with the \sF structure $(\dis,\normdot)$. For further details on \sF geometry, we refer to \cite[Sec.\ 2.2]{magnabosco2023failure}. We define the associated left-invariant norm on $\dis$ as 
\begin{equation}
    \|v\|_\dis:=\|(u_1,u_2)\|,\qquad \text{for every }v=u_1X_1+u_2X_2\in \dis. 
\end{equation}
A curve $\gamma: [0,1]\to \hei$ is \emph{admissible} if its velocity $\dot\gamma(t)$ exists almost everywhere and there exists a function $u=(u_1,u_2)\in L^2([0,1];\R^2)$ such that
\begin{equation}
\label{eq:admissible_curve}
    \dot\gamma(t)= u_1(t)X_1(\gamma(t))+u_2(t)X_2(\gamma(t))\in\dis_{\gamma(t)},\qquad\text{for a.e. }t\in [0,1].
\end{equation}
The function $u$ is called \emph{the control}. We define the \emph{length} of an admissible curve:
\begin{equation}
    \ell(\gamma):=\int_0^1 \norm{\dot\gamma(t)}_{\dis} \de t\in[0,\infty).
\end{equation}
For every couple of points $q_0,q_1\in M$, define the \emph{\sF distance} between them as
\begin{equation*}
    \di (q_0,q_1) := \inf \left\{\ell(\gamma)\, :\, \gamma \text{ admissible, } \gamma(0)=q_0 \text{ and }\gamma(1)=q_1\right\}.
\end{equation*}

We recall that the Chow--Rashevskii Theorem ensures that the \sF distance on $\hei$ is finite, continuous on and the induced topology is the manifold one. 

\begin{remark}
    Since both the norm and the distribution are left-invariant, the left-translations defined by
\begin{equation}
\label{eq:left_translations}
    L_p:\hei\to\hei;\qquad L_p(q):=p\star q,
\end{equation}
are isometries for every $p\in \hei$.
\end{remark}

\begin{definition}
    Let $\hei$ be the Heisenberg group. We say that a Borel measure $\m$ on $\hei$ is \emph{smooth} if it is absolutely continuous with respect the Lebesgue measure $\Leb^3$ with a smooth and strictly positive density.    
\end{definition}

\subsection{Geodesics in the Heisenberg group}

In the \sF Heisenberg group, the geodesics were originally studied in \cite{Busemann} and \cite{Bereszynski} for the three-dimensional case and in \cite{lok2} for general left-invariant structures on higher-dimensional Heisenberg groups.
Now, we define the map $G_t$ which plays the role of a \emph{\sF exponential map} from the origin at time $t$, as justified by Propositions \ref{prop:respira} and \ref{prop:summary} below.

\begin{definition}\label{def:exponentialmap}
    Let $\hei$ be the \sF Heisenberg group, equipped with a norm $\normdot$, let $\Omega$ be the associated closed unit ball and $\Omega^\circ$ its polar. Let 
    \begin{equation}
        \mathscr{U}:=\R_{>0}\times \R/2\pi_{\Omega^\circ}\mathbb{Z}\times \{(-2\pi_{\Omega^\circ},2\pi_{\Omega^\circ})\setminus\{0\}\}.
    \end{equation} 
    For every $t\in \R$, we define the continuous mapping
    $G_t:\mathscr{U}\to\hei$ such that for any $(r,\phi,\omega)\in \mathscr{U}$, $G_t(r,\phi,\omega):=(x_t,y_t,z_t)$, where
\begin{equation}\label{eq:xyz}
    \begin{cases}
    \begin{aligned}
        x_t(r,\phi,\omega) &= \frac{r}{\omega}\left(\sinomp(\phi+\omega t) - \sinomp(\phi)\right),\\
        y_t(r,\phi,\omega) &= -\frac{r}{\omega}\left(\cosomp(\phi+\omega t) - \cosomp(\phi)\right),\\
        z_t(r,\phi,\omega) &= \frac{r^2}{2\omega^2}\left(\omega t + \cosomp(\phi+\omega t) \sinomp(\phi) - \sinomp(\phi+\omega t ) \cosomp(\phi)\right).
    \end{aligned}
    \end{cases}
\end{equation}
We stress that the domain $\mathscr{U}$ does not contain $\omega=0,\pm 2\pi_{\Omega^\circ}$. 
\end{definition}

 The curve $[0,1]\ni t\mapsto  G_t(r,\phi,\omega)\in \hei$ satisfies the Pontryagin maximum principle, cf. \cite[Thm.\ 4]{lok2}, thus, as a consequence of \cite[Thm.\ 1]{Bereszynski}, we deduce the following result.

\begin{prop}
    \label{prop:respira}
    The curve $\gamma_{(r,\phi,\omega)}:[0,1]\to \hei$, defined as $\gamma_{(r,\phi,\omega)}(t):=G_t(r,\phi,\omega)$ is a geodesic between its endpoints $\e=\gamma_{(r,\phi,\omega)}(0)$ and $\gamma_{(r,\phi,\omega)}(1)$.
\end{prop}

\begin{prop} 
\label{prop:summary}
Let $\hei$ be the \sF Heisenberg group, equipped with a norm $\normdot$ and let 
\begin{equation}
\mathscr{R}:=\{(r,\phi,\omega)\in \mathscr{U}:  -2\pi_{\Omega^\circ}+\delta_+(\phi)<\omega<2\pi_{\Omega^\circ}-\delta_-(\phi)\},
\end{equation}
where $\delta_\pm$ are defined as in \eqref{eq:interval}. For $(r,\phi,\omega)\in\mathscr{R}$, the curve $\gamma_{(r,\phi,\omega)}:[0,1]\to\hei$; $\gamma_{(r,\phi,\omega)}(t)=G_t(r,\phi,\omega)$ is the unique geodesic between its endpoints.
    
If the norm $\normdot$ is $C^1$, then $\mathscr{R}=\mathscr{U}$. While, if the norm $\normdot$ is strictly convex, then $G_t$ is continuously extended to $\omega = 0$ by 
\begin{equation}
\label{eq:straight_lines}
    \begin{cases}
    \begin{aligned}
        x_t(r,\phi,0) &= \left(r\cosom(\phi_\circ)\right) t,\\
        y_t(r,\phi,0) &= \left(r\sinom(\phi_\circ)\right) t,\\
        z_t(r,\phi,0) &= 0,
    \end{aligned}
    \end{cases}
\end{equation}
where we use the shorthand $\phi_\circ:=C_\circ(\phi)$. Finally, if the norm $\normdot$ is $C^1$ and strictly convex, then $G_1:\mathscr{U}\to\{(x,y,z)\in\hei: z\neq 0,(x,y)\neq(0,0)\}$ is a homeomorphism.
\end{prop}

% \todo[inline]{Sam: Add a remark about the ``if and only if'' part of the statement}

\begin{notation}
    We fix two important notations:
    \begin{itemize}
        \item[(i)] Denote by $ \Omega^\circ_{(r,\phi,\omega)}$ the following transformation of $\Omega^\circ=\overline B^{\|\cdot\|_*}_1(0)$:   
    \begin{equation}
          \Omega^\circ_{(r,\phi,\omega)}:= R_{-\pi/2}\left[\frac{r}{\omega}(\Omega^\circ-(\cosomp(\phi),\sinomp(\phi)))\right],
    \end{equation}
where $R_{-\pi/2}$ is counter-clockwise rotation in the plane of angle $-\pi/2$. 
    \item[(ii)] For a continuous curve $\gamma=(x,y,z):[0,1]\to \hei$ and for every $t\in [0,1]$, we denote by $A_t(\gamma)$ the oriented area that is swept by the vector joining $(0,0)$ with $(x(s),y(s))$, for $s\in [0,t]$.
    %\item[c)] \todo{maybe we can delete this} For $(r,\phi,\omega)\in\mathscr{R}$, define the positive number $t^*(r,\phi,\omega):=\frac{2\pi_{\Omega^\circ}}{|\omega|}$. Moreover, we write $t^*=+\infty$ when $\omega=0$. 
    \end{itemize}
\end{notation}

The proof of Proposition \ref{prop:summary} is based on the following results.
\begin{prop}[{\cite[Thm.\ 1]{Bereszynski}}]
\label{prop:z=area}
    Let $\hei$ be the \sF Heisenberg group, equipped with a norm $\normdot$. Consider the curve $\gamma_{(r,\phi,\omega)}$ for some $(r,\phi,\omega)\in\mathscr{U}$. Then, the curve $t\mapsto (x_t(r,\phi,\omega), y_t(r,\phi,\omega))$, which is the projection of $\gamma_{(r,\phi,\omega)}$ onto $\{z=0\}$, belongs to the boundary of $\Omega^\circ_{(r,\phi,\omega)}$. Moreover, for every $t\in\big[0,\frac{2\pi_{\Omega^\circ}}{|\omega|}\big]$, we have $z(t)=A_t(\gamma)$.
\end{prop}

\begin{prop}[{\cite[Thm.\ 6]{noskov2008}}]\label{prop:Noskov}
    There exists a continuous function $\mu:\R^2 \to \R_{\geq 0}$, which is $2$-homogeneous, i.e. $\mu(\lambda p)= \lambda^2 \mu(p)$ for every $\lambda \in \R$, such that if $|z| > \mu(x,y)$ then any geodesic connecting the origin $\e$ with the point $(x,y,z)$ projects to a subpath of a unique (up to translations) isoperimetric profile on the plane $\{z=0\}$.
\end{prop}
    \begin{remark}
        \label{rmk:mu}
        According to \cite[Thm.\ 6]{noskov2008}, for every $p\neq 0 \in \R^2$ the quantity $\mu(p)$ is the infimum of the positive areas swept by the subpaths of isoperimetric profiles joining $0$ to $p$. In particular, if the norm is strictly convex, then $\partial\Omega^\circ$ is $C^1$, and $\mu\equiv 0$.  Finally, we stress that the uniqueness of the isoperimetric profile has to be intended up to translations in $\R^2$.  
    \end{remark}

\begin{proof}[Proof of Proposition \ref{prop:summary}]

First of all, Proposition \ref{prop:respira} tells us that $\gamma_{(r,\phi,\omega)}$ is a geodesic. We are now going to prove uniqueness assuming $\omega>0$, without loss of generality. As $(r,\phi,\omega)\in\mathscr{R}$, we can fix $T>1$ such that $\pi_{\Omega^\circ}<\omega T< 2\pi_{\Omega^\circ}-\delta_-(\phi)$. Observe that the lower bound on $\omega T$ ensures that the area $A_T(\gamma_{(r,\phi,\omega)})$ is bigger than half the area of $\Omega^\circ_{(r,\phi,\omega)}$. Moreover, recalling the definition of $\delta_-(\phi)$, cf. \eqref{eq:interval}, we have that $(x_T(r,\phi,\omega),y_T(r,\phi,\omega))$ does not lie on the flat segment of $\partial \Omega^\circ_{(r,\phi,\omega)}$ containing $O:=(0,0)$. In particular, thanks to Proposition \ref{prop:z=area}, we deduce that $z_T(r,\phi,\omega)> \frac{r^2}{\omega^2}\pi_{\Omega^\circ}$. 

On the other hand, take $\phi'= \phi+\omega T - \pi_{\Omega^\circ}$ and $T'= \frac{2 \pi_{\Omega^\circ}}{\omega}-T$ and observe that, keeping in mind \eqref{eq:symmetry1}, we can explicitly compute that 
\begin{equation}\label{eq:iltruccone}
    \big(x_T(r,\phi,\omega),y_T(r,\phi,\omega)\big) = \big(x_{T'}(r,\phi',\omega),y_{T'}(r,\phi',\omega)\big).
\end{equation}
Moreover, since $\pi_{\Omega^\circ} < \omega T< 2\pi_{\Omega^\circ}-\delta_-(\phi)$, we have that $0 < \omega T'< \pi_{\Omega^\circ}$. Therefore, the area $A_{T'}(\gamma_{(r,\phi',\omega)})$ is smaller than half the area of $\Omega^\circ_{(r,\phi',\omega)}$. In particular, we deduce that $A_{T'}(\gamma_{(r,\phi',\omega)}) < \frac{r^2}{\omega^2}\pi_{\Omega^\circ}$. Then, keeping in mind Remark \ref{rmk:mu} and by Proposition \ref{prop:z=area} and \eqref{eq:iltruccone}, we can then conclude that 
    \begin{equation}
        z_T(r,\phi,\omega)> \frac{r^2}{\omega^2}\pi_{\Omega^\circ} > A_{T'}(\gamma_{(r,\phi',\omega)}) \geq \mu\big(x_T(r,\phi,\omega),y_T(r,\phi,\omega)\big).
    \end{equation}
    Thus, by Proposition \ref{prop:Noskov}, we conclude that any geodesic connecting $\e$ and $\gamma_{(r,\phi,\omega)}(T)$ projects to a subpath of an isoperimetric profile. 
    
    Assume by contradiction that there exist two distinct geodesics joining $\e$ and $\gamma_{(r,\phi,\omega)}(T)$. Both geodesics must project to a subpath of an isoperimetric profile, therefore we find two distinct subpaths  
    \begin{equation}
        \Gamma_1\subset \partial\Omega^\circ_{(r,\phi,\omega)},\qquad\Gamma_2\subset \partial\Omega^\circ_{(r',\phi',\omega')}
    \end{equation}
    of isoperimetric profiles, enclosing the same signed area (thus $\frac{r}{\omega}=\frac{r'}{\omega'}$), and joining $O=(0,0)$ and $P:=(x_T(r,\phi,\omega),y_T(r,\phi,\omega))$.
    From Proposition \ref{prop:Noskov}, cf. Remark \ref{rmk:mu}, there exists a translation $T_v:\R^2\to\R^2$ such that 
    \begin{equation}
    \label{eq:translation_v}
        T_v(x,y)=(x,y)+v,\qquad T_v(\partial\Omega^\circ_{(r,\phi,\omega)})=\partial\Omega^\circ_{(r',\phi',\omega')},
    \end{equation}
    for some $v\in\R^2$. Now, let $\ell$ be the straight line through $O$ and $P$ and let $R_i$ be the region bounded by $\Gamma_i$ and $\ell$, for $i=1,2$. Then, by our assumptions, $R_i$ has swept area $z_T(r,\phi,\omega)$, for $i=1,2$. Thus, as $T_v$ is a translation, the swept area of the regions $R_2$ and $T_v(R_1)$ is the same. But $T_v(R_1)$ is bounded by $T_v(\Gamma_1)\subset\partial\Omega^\circ_{(r',\phi',\omega')}$ and by the line $T_v(\ell)$, hence, to not change the signed area, $\ell$ must be fixed by $T_v$, meaning that $v$ is parallel to the vector $OP$. 
    % as $T_v$ is a translation, the subpath $T_v(\Gamma_i)$ encloses the same (signed) area as $\Gamma_i$, for $i=1,2$. Thus also the subpaths $T_v(\Gamma_1)$ and $\Gamma_2$ enclose the same signed area. This implies that $v$ is parallel to the vector $OP$, since translating $\Gamma_1$ by $v$ can not change the area. 
    Assume, without loss of generality, that $v=\lambda\, OP$ for $\lambda>0$. Then, we claim that 
    \begin{equation}
        \label{eq:claim_flat_segment_norm}
        [O,P]\subset\partial\Omega^\circ_{(r,\phi,\omega)},
    \end{equation}
    where $[O,P]:=\{s P : s\in [0,1]\}$. Indeed, according to \eqref{eq:translation_v}, the point $Q:= (T_v)^{-1}(O)= -v$ belongs to $\partial\Omega^\circ_{(r,\phi,\omega)}$ and, by construction, it is distinct from $O$ and $P$. Then we find three distinct extremal points $O,Q$ and $P$ for the convex set $\Omega^\circ_{(r,\phi,\omega)}$, belonging to $\ell$. Thus the claim \eqref{eq:claim_flat_segment_norm} is verified, by convexity.
    
    To conclude the proof, we recall that the assumption $\omega T<2\pi_{\Omega^\circ}-\delta_-(\phi)$ implies that the endpoint $P$ must not lie on the flat segment of $\partial\Omega^\circ_{(r,\phi,\omega)}$ containing $O$.
     This is in contradiction with \eqref{eq:claim_flat_segment_norm} and concludes the proof of the first part of the statement, since $T>1$ and any restriction of a geodesic is still a geodesic.

We now show the second part of the statement. If the norm $\normdot$ is $C^1$,
then $C^\circ$ is single valued and the whole interval $C_\circ(\phi)$ are mapped to a single point.
Therefore $\mathscr{R}=\mathscr U$.

If the norm $\normdot$ is strictly convex, then $C_\circ$ is single valued and continuous (cf.\ Proposition \ref{prop:regularityCo}),
and we can compute the limits
\begin{equation}
    \lim_{\omega\to 0}x_t(r,\phi,\omega)=(r\cosom(\phi_\circ))t\quad \text{and} \quad \lim_{\omega\to 0}y_t(r,\phi,\omega)=-(r\sinom(\phi_\circ))t,
\end{equation}
cf. Proposition \ref{prop:difftrig}. Moreover, since $\gamma_{(r,\phi,\omega)}$ is horizontal, then $\dot z_t=\frac{1}{2}(x_t\dot{y}_t-y_t\dot{x}_t)$, and we can obtain that
\begin{align*}
    z_t(r,\phi,\omega)=\frac{1}{2}\int_0^t&x_t\dot{y}_t-y_t\dot{x}_t\diff t\\
    =\frac{1}{2}\int_0^t&r\Big[\frac{\sinomp(\phi+\omega t) - \sinomp(\phi)}{\omega}\sinom((\phi+\omega t)_\circ)\\
    &+\frac{\cosomp(\phi+\omega t) - \cosomp(\phi)}{\omega}\cosom((\phi+\omega t)_\circ)\Big]\diff t 
\end{align*}
converges to $0$ as $\omega\to 0$, by the dominated convergence theorem.

Finally, if the norm $\normdot$ is $C^1$ and strictly convex, then $\mu\equiv 0$ (cf. Remark \ref{rmk:mu}) and $\mathscr{R}=\mathscr{U}$. Therefore, according to the first part of the statement, for every $p\in\{(x,y,z)\in\hei : z\neq 0,(x,y)\neq (0,0)\}$, there exists a unique geodesic of the form $\gamma_{(r,\phi,\omega)}$ for a unique $(r,\phi,\omega)\in \mathscr{U}$, joining $\e$ and $p$. This implies that $G_1$ is a bijection onto the subset $\{(x,y,z)\in\hei: z\neq 0,(x,y)\neq (0,0)\}$. In addition, $G_1$ is homeomorphism onto its image since it is continuous and proper on $\mathscr{U}$.
\end{proof}

\begin{remark}\label{rmk:chesbatti}
    From Definition \ref{def:exponentialmap}, we can observe that
    \begin{equation}
        G_t(r,\phi,\omega)=G_1(tr,\phi,t\omega).
    \end{equation}
    In particular, when the the norm $\normdot$ is $C^1$ and strictly convex, for every $t\in(0,1]$ Proposition \ref{prop:summary} guarantees that the map 
    \begin{equation}
        G_t: \R_{>0}\times \R/2\pi_{\Omega^\circ}\mathbb{Z} \times(-2\pi_{\Omega^\circ}/t,2\pi_{\Omega^\circ}/t)\setminus\{0\} \to \{(x,y,z)\in\hei: z\neq 0,(x,y)\neq (0,0)\}
    \end{equation}
    is a homeomorphism. Note that this identifies a sub-Finsler analogue to the so-called cotangent injectivity domain of sub-Riemannian geometry.
\end{remark}

\subsection{The Jacobian of the exponential map}\label{subsec:jacobian}

In this section we study the Jacobian of the exponential map \eqref{eq:xyz}, in particular when it is well-defined and its properties depending on the assumption we make on the reference norm $\normdot$. This object will play a fundamental role in this work, as it will allow us to formulate strategies to address the validity of the $\MCP(K,N)$ and $\cd(K,N)$ conditions (cf. Section \ref{sec:CDMCPH}).

In this first proposition we identify the set of differentiability points of the exponential map
and provide the explicit expression of the Jacobian at these points. We recall that $\mathsf D_0\subset \R/2\pi_{\Omega^\circ}\mathbb{Z}$ is the set of differentiability points of the functions $\sinomp$ and $\cosomp$, cf.\ Section \ref{sec:convex_trigtrig}.
% {\red Denote by $\mathscr{R}\subset \R$ the set of continuous point of $C_\circ$.}
\begin{prop}
\label{prop:NONO}
    Let $\hei$ be the \sF Heisenberg group, equipped with a norm $\normdot$. Given any $(r,\phi,\omega)\in \mathscr{U}$ with $\phi\in \mathsf{D}_0$, for every $t\in [0,1]$ such that $\phi+ \omega t\in \mathsf{D}_0$,
    the exponential map at time $t$
    \begin{equation}
        (r,\phi,\omega)\longmapsto G_t(r,\phi,\omega)
    \end{equation}
    is differentiable in $(r,\phi, \omega)$ with Jacobian 
    \begin{equation}\label{eq:jacobian}
    \begin{split}
         \J_t (r,\phi, \omega) = \frac{r^3 t}{ \omega^4} \bigg[2 &- \Big(\sinomp(\phi + \omega t) \sinom(\phi_\circ)+ \cosomp(\phi + \omega t) \cosom(\phi_\circ)\Big)  \\
         &- \Big(\sinom\big(( \phi + \omega t)_\circ\big) \sinomp( \phi)+ \cosom\big( ( \phi + \omega t)_\circ \big) \cosomp( \phi)\Big)\\
         &- \omega t \Big(\sinom\big((\phi +  \omega t)_\circ\big) \cosom(\phi_\circ) - \cosom\big((\phi + \omega t)_\circ\big) \sinom(\phi_\circ) \Big) \bigg]. 
    \end{split}
    \end{equation}
    
\end{prop}

\begin{proof} 
   Note that, since $ \phi, \phi+   \omega t\in \mathsf{D}_0$, the trigonometric functions $\cosomp$ and $\sinomp$ are differentiable at $\phi$ and $\phi+  \omega t$ and, in addition, the correspondence map $C_\circ$ is single-valued at $\phi$ and $ \phi+ \omega t$. It is then possible to apply Proposition \ref{prop:difftrig} to explicitly differentiate the quantities $x_t(r,\phi,\omega)$, $y_t(r,\phi,\omega)$ and  $z_t(r,\phi,\omega)$ (cf. \eqref{eq:xyz}). After routine computations, we end up with \eqref{eq:jacobian}.
\end{proof}

\begin{remark}
\label{rmk:Jdefinedae}
Since $\mathsf{D}_0$ has full $\Leb^1$-measure in $\R/2\pi_{\Omega^\circ}\mathbb{Z}$, Fubini's theorem implies that for $\Leb^{2}$-a.e. $(\phi,\psi)\in \R/2\pi_{\Omega^\circ}\mathbb{Z}\times (-2\pi_{\Omega^\circ},2\pi_{\Omega^\circ})$, we have $\phi, \phi+\psi\in\mathsf{D}_0$. In particular, given any $t$, the Jacobian $\J_t (r,\phi, \omega) $ is defined for $\Leb^3$-almost every $(r,\phi, \omega)\in \mathscr U$. In addition, if the reference norm $\normdot$ is strictly convex, then, according to Proposition \ref{prop:regularityCo}, we know that the map $C_\circ$ is everywhere single valued (and then $\mathsf{D}_0= \R/ 2\pi_{\Omega^\circ}\mathbb{Z}$) and continuous. Therefore, the map $G_t$ is $C^1$, as it is everywhere differentiable and its Jacobian \eqref{eq:jacobian} is continuous. 
\end{remark}

We define the \emph{reduced Jacobian} as the measurable function $\J_R:\R/2\pi_{\Omega^\circ}\mathbb Z\times (-2\pi_{\Omega^\circ},2\pi_{\Omega^\circ})\to\R$, such that for every $\phi$ and $\psi$ satisfying $\phi\in \mathsf{D}_0$ and $\phi + \psi \in \mathsf{D}_0$, we have 
    \begin{equation}\label{eq:reduced_Jac}
    \begin{split}
         \J_R (\phi, \psi) := 2 &- \Big(\sinomp( \phi +  \psi) \sinom( \phi_\circ)+ \cosomp( \phi +  \psi) \cosom( \phi_\circ)\Big)  \\
         &- \Big(\sinom\big(( \phi +  \psi)_\circ\big) \sinomp( \phi)+ \cosom\big( ( \phi +  \psi)_\circ \big) \cosomp( \phi)\Big)\\
         &- \psi \Big(\sinom\big(( \phi +  \psi)_\circ\big) \cosom( \phi_\circ) - \cosom\big(( \phi +  \psi)_\circ\big) \sinom( \phi_\circ) \Big). 
    \end{split}
    \end{equation}
According to Proposition \ref{prop:NONO}, the Jacobian in \eqref{eq:jacobian}, when defined, can be expressed as 
\begin{equation}\label{eq:J-redJ}
    \J_t (r,\phi, \omega) = \frac{r^3 t}{\omega^4} \J_R (\phi, \omega t ).
\end{equation}
%The regularity of the reduced Jacobian $\J_R$ is as follows.

\begin{remark}
   Along the same lines of Remark \ref{rmk:Jdefinedae}, we deduce that the reduced Jacobian $\J_R(\phi,\psi)$ is defined $\Leb^2$-almost everywhere. Then, Lebesgue differentiation theorem ensures that for $\Leb^2$-almost every $(\phi,\psi)$, we have that $\phi, \phi+\psi\in\mathsf{D}_0$ and $(\phi,\psi)$ is a Lebesgue point for $\J_R$.
    Applying Fubini's theorem, we deduce the existence of a $\Leb^1$-full measure set $\bar {\mathsf{D}}_0\subset\mathsf{D}_0$ such that, for every $\phi \in \bar {\mathsf{D}}_0$, we have that $\phi+ \psi \in \mathsf{D}_0$ and $(\phi,\psi)$ is a Lebesgue point for $\J_R$, for $\Leb^1$-almost every $\psi \in (-2\pi_{\Omega^\circ},2\pi_{\Omega^\circ})$.
\end{remark}

\begin{prop}\label{lem:reducedJ}
    Let $\hei$ be the \sF Heisenberg group, equipped with a norm $\normdot$. Given $\phi\in\mathsf{D}_0$, we have that $\J_R(\phi,\psi)>0$ for every $\psi\in(\pi_{\Omega^\circ},2\pi_{\Omega^\circ}-\delta_-(\phi))$ such that $\phi+\psi\in\mathsf{D}_0$.
        In particular, if the reference norm $\normdot$ is $C^1$, then $\J_R(\phi,\psi)>0$ for every $\phi\in \mathsf{D}_0$ and every $\psi\in(\pi_{\Omega^\circ},2\pi_{\Omega^\circ})$ such that $\phi+\psi\in\mathsf{D}_0$.
\end{prop}

\begin{proof}
    
First of all, according to \eqref{eq:pytagorean}, we have that 
    \begin{equation}\label{eq:in1}
        \sinomp( \phi +  \psi) \sinom( \phi_\circ)+ \cosomp( \phi +  \psi) \cosom( \phi_\circ) \leq 1 ,
    \end{equation}
    \begin{equation}\label{eq:in2}
        \sinom\big(( \phi +  \psi)_\circ\big) \sinomp( \phi)+ \cosom\big( ( \phi +  \psi)_\circ \big) \cosomp( \phi) \leq 1,
    \end{equation}
    and the equalities hold if and only if $\phi+\psi\in C^\circ\circ C_\circ(\phi)$. We claim that
    \begin{equation}\label{eq:in3}
       - \sinom\big(( \phi +  \psi)_\circ\big) \cosom( \phi_\circ) + \cosom\big(( \phi +  \psi)_\circ\big) \sinom( \phi_\circ) \geq0,
    \end{equation}
    with the equality holding under the same condition. On the one hand, the quantity at the left-hand side can be interpreted as the scalar product of the vectors 
    \begin{equation}\label{eq:scalar1}
        \big( \cosom( \phi_\circ), \sinom( \phi_\circ)\big) \quad \text{and} \quad \big( - \sinom\big(( \phi +  \psi)_\circ\big), \cosom\big(( \phi +  \psi)_\circ\big)\big). 
    \end{equation}
    On the other hand, observe that the (Euclidean) angle between the vectors 
    \begin{equation}\label{eq:scalar2}
        \big( \cosom( \phi_\circ), \sinom( \phi_\circ)\big) \quad \text{and} \quad \big( \cosom\big( (\phi+\psi)_\circ\big), \sinom( (\phi+\psi)_\circ)\big)
    \end{equation}
    is in $(\pi,2\pi)$ for every $\psi \in (\pi_{\Omega^\circ}, 2 \pi_{\Omega^\circ}-\delta_-(\phi))$. Therefore, inequality \eqref{eq:in3} follows by noticing that the second vector in \eqref{eq:scalar1} is the second vector in \eqref{eq:scalar2} rotated by an (Euclidean) angle of $\frac\pi 2$. Keeping in mind \eqref{eq:reduced_Jac}, the thesis follows by putting together \eqref{eq:in1}, \eqref{eq:in2} and \eqref{eq:in3}.

    The second part of the statement is trivial, because, if the reference norm $\normdot$ is $C^1$, we have that $\delta_-(\phi)=0$ for every $\phi \in \R/2\pi_{\Omega^\circ}\mathbb{Z}$.
\end{proof}

By requiring more regularity on the reference norm $\normdot$, we can prove a stronger version of Proposition \ref{lem:reducedJ}, taking advantage of the following lemma.

\begin{lemma}
    \label{cor:diff_points_positive_measure}
    Assume that the reference norm $\normdot$ is $C^1$ and strongly convex. Let $\sfD_1^+\subset \R/2\pi_{\Omega^\circ}\mathbb{Z}$ be the set of angles where $C_\circ$ is differentiable with positive derivative. Then, $\sfD_1^+ \cap I$ has positive $\Leb^1$-measure, for every open interval $I\subset \R/2\pi_{\Omega^\circ}\mathbb{Z}$.
\end{lemma}

\begin{proof}
    According to item $\textsl{(iii)}$ of Proposition \ref{prop:regularityCo}, the map $C_\circ$ is Lipschitz and therefore it is differentiable $\Leb^1$-almost everywhere and absolutely continuous. Moreover, item $\textsl{(i)}$ of Proposition \ref{prop:regularityCo} guarantees that $C_\circ$ is strictly increasing. In particular, given any open interval $I=(a,b)$, we deduce that 
    \begin{equation}
        \int_{a}^{b} C'_\circ(\psi) \de \psi = C_\circ(b) - C_\circ(a)>0.
    \end{equation}
    Therefore the set $\sfD_1^+ \cap I$ must have positive $\Leb^1$-measure. 
\end{proof}

\begin{prop}
    \label{prop:stronglyJacpositivity}
    Let $\hei$ be the sub-Finsler Heisenberg group, equipped with a $C^1$ and strongly convex norm $\normdot$. Then, the reduced Jacobian $\J_R$ is everywhere non-negative and we have $\J_R(\phi, \psi) = 0$ if and only if $\psi = 0$.
\end{prop}

\begin{proof}
    From the explicit formula \eqref{eq:reduced_Jac}, it is easy to check that if $\psi = 0$, then $\J_R(\phi, \psi) = 0$. We now focus on proving that if $\psi \neq 0$, then $\J_R(\phi, \psi) > 0$. By the periodicity of the trigonometric functions, we can assume $\phi\in \R/2\pi_{\Omega^\circ}\mathbb{Z}$ and, without loss of generality, fix $\psi\in (0,2 \pi_{\Omega^\circ})$, as the case for negative $\psi$ is completely analogous. 
    % it is sufficient to prove the positivity of $\J_R$ for all $\phi, \psi \in \rinterval{0}{2 \pi_{\Omega^\circ}}$ with $\psi \neq 0$. 
    Proposition \ref{lem:reducedJ} ensures that $\J_R(\phi, \psi) > 0$ for every $\phi \in \R/2\pi_{\Omega^\circ}\mathbb{Z}$ and $\psi \in \ointerval{\pi_{\Omega^\circ}}{2 \pi_{\Omega^\circ}}$. To prove positivity for the other values of $\psi$, we preliminary observe that, under the hypothesis of this proposition, the map $C_\circ$ is Lipschitz, cf. Proposition \ref{prop:regularityCo}. In particular, $C_\circ$ is absolutely continuous and differentiable $\Leb^1$-almost everywhere and consequently, given any $\phi \in \R/2\pi_{\Omega^\circ}\mathbb{Z}$, the function $\psi \mapsto \J_R(\phi, \psi)$ is itself absolutely continuous and differentiable $\Leb^1$-almost everywhere. For every differentiability point $\psi\in (0,\pi_{\Omega^\circ}]$ of $\J_R(\phi,\cdot)$, we have
    \begin{equation}
        \label{eq:dJK}
        \partial_\psi \J_R(\phi, \psi) = C'_\circ(\phi + \psi) \mathcal{K}(\phi, \psi),
    \end{equation} where
    \begin{multline}
    \label{def:functionK}
        \mathcal{K}(\phi, \psi) := \sin_{\Omega^\circ}(\phi + \psi)  \cos_{\Omega^\circ}(\phi) - \cos_{\Omega^\circ}(\phi + \psi) \sin_{\Omega^\circ}(\phi) \\
     - \psi \Big( \cos_{\Omega^\circ}(\phi + \psi) \cos_{\Omega}(\phi_\circ) + \sin_{\Omega^\circ}(\phi + \psi) \sin_{\Omega}(\phi_\circ) \Big).
    \end{multline}
    We claim that $\mathcal{K(\phi, \psi)} > 0$ for all $\psi \in \linterval{0}{\pi_{\Omega^\circ}}$. In order to prove this statement, we consider the function $h:\R \to \R$ defined as
    \begin{equation}
        t \mapsto \mathcal{K}(\phi + \psi - t, t),
    \end{equation}
    a direct computation shows that $h(0)=0$. Moreover, reasoning as before, we can deduce that $h$ is absolutely continuous and differentiable almost everywhere with derivative:
    \begin{equation}
        h'(t)= - t C_\circ'(\phi+\psi -t) \left( \cos_{\Omega^\circ}(\phi+\psi) \sin_{\Omega^\circ}(\phi+\psi -t) - \sin_{\Omega^\circ}(\phi+\psi) \cos_{\Omega^\circ}(\phi+\psi -t)\right).
    \end{equation}
    The term between parentheses can be interpreted as the inner product between the vectors
    \begin{equation}
        \big(\cos_{\Omega^\circ}(\phi+\psi ), \sin_{\Omega^\circ}(\phi + \psi)\big ) \quad \text{ and } \quad  \big(\sin_{\Omega^\circ}(\phi+\psi-t ), -\cos_{\Omega^\circ}(\phi + \psi-t)\big ),
    \end{equation}
     the latter being the rotation of $\big(\cos_{\Omega^\circ}(\phi+\psi-t ), \sin_{\Omega^\circ}(\phi + \psi-t)\big )$ by an (Euclidean) angle of $- \frac{\pi}{2}$. Then, the scalar product is negative for every $t \in \ointerval{0}{\pi_{\Omega^{\circ}}}$. In particular, recalling Lemma \ref{cor:diff_points_positive_measure} and that $h$ is absolutely continuous, we conclude that 
     \begin{equation}
         \mathcal{K}(\phi , \psi )= h(\psi) = \int_0^\psi h'(s) \de s > 0,
     \end{equation}
    proving our claim. Analogously, equation \eqref{eq:dJK} combined with Lemma \ref{cor:diff_points_positive_measure} allows us to deduce that 
    $\J_R(\phi, \psi) > 0$ for every $\phi \in \R/2\pi_{\Omega^\circ}\mathbb{Z}$ and $\psi \in \linterval{0}{\pi_{\Omega^\circ}}$, concluding the proof.
\end{proof}
% \todo[inline]{Sam: reminder to myself. Add a potential remark about the injectivity radius}
% \todo[inline]{Sam: I think that in fact the ``cotangent injectivity domain'' already appears in Remark 3.10.}

\subsection{The \texorpdfstring{$\MCP$}{MCP} condition in the Heisenberg group}\label{sec:CDMCPH}

In this section we discuss several results that will help us in addressing the validity of the $\MCP(K,N)$ condition in the \sF Heisenberg groups. In these results the Jacobian of the exponential map studied in the previous section plays a fundamental role, as it describes the infinitesimal volume contraction along geodesics.

First of all, we show that, since the Heisenberg group admits a one-parameter family of dilations, it is sufficient to study the validity of the $\MCP(K,N)$ (and the $\cd(K,N)$ condition) with curvature parameter $K=0$ and having $\Leb^3$ as reference measure.

\begin{prop}\label{prop:Knoncipiace}
    Let $\hei$ be the sub-Finsler Heisenberg group, equipped with a norm $\normdot$, and with a smooth measure $\m$. Assume that the metric measure space $(\hei, \di, \m)$ satisfies the $\MCP(K,N)$ (resp. $\cd(K, N)$) condition, for some $K\in \R$ and $N\in (1,\infty)$. Then, the metric measure $(\hei, \di, \Leb^3)$ satisfies the $\MCP(0,N)$ (resp. $\cd(0, N)$) condition.
\end{prop}

\begin{proof}
    Let $m: \hei\to \R_{>0}$ be the (smooth) density of the measure $\m$ with respect to the Lebesgue measure $\Leb^3$. Since the Heisenberg group admits a one-parameter family of dilations \cite{MR3283670}, it holds that
    \begin{equation}
        \big( \hei, n \cdot \di, n^4 \cdot \m, \e \big) \xlongrightarrow{\text{pmGH}} (\hei, \di, m(\e)\cdot \Leb^3,\e) \quad \text{as }n \to \infty.
    \end{equation}
    Moreover, by the scaling property of the $\MCP$ condition, we have that $\big( \hei, n \cdot \di, n^4 \cdot \m \big)$ is a $\MCP(K/n^2,N)$ space.
    Then, the $\mathrm{pmGH}$-stability of the $\MCP$ condition ensures that $(\hei, \di, m(\e)\cdot \Leb^3)$ is a $\MCP(0,N)$ space. Using once again the scaling property, we conclude that $(\hei, \di,\Leb^3)$ is a $\MCP(0,N)$ spaces as well. The same argument, proves the result for the $\cd$ condition.
\end{proof}

Second of all, we characterize the measure contraction property $\MCP(K,N)$ in terms of the reduced Jacobian of the exponential map.

\begin{prop}\label{prop:MCPgeneral}
    Let $\hei$ be the sub-Finsler Heisenberg group, equipped with a norm $\normdot$, and with the Lebesgue measure $\Leb^3$. If the metric measure space $(\hei, \di, \Leb^3)$ satisfies the measure contraction property $\mathsf{MCP}(0, N)$, then
    \begin{equation}\label{eq:equivalent_reducedJ}
        |\J_R (\phi, \omega t)|\geq t^{N-1} |\J_R (\phi, \omega)|,
    \end{equation}
    for every $(r, \phi, \omega)\in \mathscr R$ and $t\in [0,1]$ such that $\phi, \phi+ \omega t, \phi+  \omega  \in \mathsf{D}_0$ and $(\phi, \omega t)$, $(\phi, \omega)$ are Lebesgue points for $\J_R$.
\end{prop}

\begin{proof}
     Fix any $(r,\phi,\omega) \in \mathscr R$ and $t$ as in the hypothesis and assume that $\mathsf{MCP}(0, N)$ holds. For every $\epsilon>0$ sufficiently small, we consider the ball $B_\epsilon(r,\phi,\omega)$ of radius $\epsilon$ (with respect to the Euclidean distance) centred at the point $(r,\phi,\omega)$. Observe that Proposition \ref{prop:summary} ensures that for every $x\in G_1\big(B_\epsilon(r,\phi,\omega)\big)$ there exists a unique geodesic connecting $\e$ to $x$ and it is of the form $\gamma_{(r'
     ,\phi'
     ,\omega'
     )}$, for some $(r'
     ,\phi'
     ,\omega'
     )\in B_\epsilon(r,\phi,\omega)$. Therefore, the characterization \eqref{eq:tj_pantaloncini} yields that for every $\epsilon$ (sufficiently small) we have that
     \begin{equation}
         \Leb^3 \big(G_t\big(B_\epsilon(r,\phi,\omega)\big) \big) \geq t^N \Leb^3 \big(G_1\big(B_\epsilon(r,\phi,\omega)\big) \big).
     \end{equation}
     Then, thanks to our assumptions on the parameters and keeping in mind \eqref{eq:J-redJ}, we deduce that
    \begin{align*}
        |\J_t( r,\phi,\omega)| ={}& \lim\limits_{\varepsilon \to 0^+} \frac{1}{\Leb^3(B_\epsilon(r,\phi,\omega))} \int_{B_\epsilon(r,\phi,\omega)} |\J_t(r'
        ,\phi'
        ,\omega'
        )| \de r'
        \de \phi'
        \de \omega'
        \\
        ={}& \lim\limits_{\varepsilon \to 0^+} \frac{\Leb^3 \big(G_t\big(B_\epsilon( r,\phi,\omega)\big) \big)}{\Leb^3(B_\epsilon(r,\phi,\omega))} \geq \lim\limits_{\varepsilon \to 0^+} 
        t^N \frac{\Leb^3 \big(G_1\big(B_\epsilon(r,\phi,\omega)\big) \big)}{\Leb^3(B_\epsilon(r,\phi,\omega))} \\
        ={}& t^N \lim\limits_{\varepsilon \to 0^+} \frac{1}{\Leb^3(B_\epsilon(r,\phi,\omega))} \int_{B_\epsilon(r, \phi,\omega)} |\mathcal{J}_1(r'
        ,\phi'
        ,\omega'
        )| \de r'
        \de \phi'
        \de \omega'
        = t^N |\J_1(r,\phi,\omega)|,
    \end{align*}
    where in the second and third equality we were able to use the area formula because $G_t$ is locally Lipschitz for every $t$. The thesis immediately follows from \eqref{eq:J-redJ}.
\end{proof}

Combining Proposition \ref{prop:Knoncipiace} and Proposition \ref{prop:MCPgeneral}, we immediately deduce the following corollary, that provides an effective strategy to disprove the measure contraction property. 

\begin{corollary}\label{cor:NO-MCP}
    Let $\hei$ be the sub-Finsler Heisenberg group, equipped with a norm $\normdot$, and with a smooth measure $\m$. Suppose there exist $(  r,  \phi,   \omega) \in \mathscr R$ and $t\in [0,1]$, such that $  \phi,   \phi+   \omega t,   \phi+   \omega \in \mathsf{D}_0$, $(  \phi,  \omega t)$, $(  \phi,  \omega)$ are Lebesgue points for $\J_R$ and 
    \begin{equation}\label{eq:criterionNOMCP}
        |\J_R ( \phi,  \omega t)|< t^{N-1} |\J_R ( \phi,  \omega)|.
    \end{equation}
     Then, the metric measure space $(\hei, \di, \m)$ does not satisfy the measure contraction property $\mathsf{MCP}(K, N)$, for every $K\in \R$ and $N\in(1,\infty)$.
\end{corollary}

In the following statement we equivalently characterize the validity of the $\MCP(0,N)$ condition, when the reference norm $\normdot$ is $C^1$ and strictly convex. In the sequel, we denote by $\mathcal{U}$ the projection of $\mathscr{U}$ in the variables $(\phi,\omega)$, namely 

\begin{equation}
\mathcal{U}:=\R/2\pi_{\Omega^\circ}\mathbb{Z}\times(-2\pi_{\Omega^\circ},2\pi_{\Omega^\circ})\setminus\{0\}.
\end{equation}

\begin{prop}\label{prop:phoenix}
    Let $\hei$ be the sub-Finsler Heisenberg group, equipped with a $C^1$ and strictly convex norm $\normdot$, and with the Lebesgue measure $\Leb^3$. Then, the metric measure space $(\hei, \di, \Leb^3)$ satisfies the $\mathsf{MCP}(0, N)$ if and only if 
    \begin{equation}
    \label{eq:equivalent_reducedJ2}
        |\J_R (\phi, \omega t)|\geq t^{N-1} |\J_R (\phi, \omega)|,
    \end{equation}
    for every $(\phi,\omega)\in \mathcal{U}$ and every $t\in [0,1]$.
\end{prop}

\begin{remark}\label{rmk:sneaky}
    The left-translations \eqref{eq:left_translations} are isometries and $\Leb^3$ is a left-invariant measure. As a consequence, in order to test the validity of the measure contraction property $\MCP(0,N)$, it is sufficient to prove the condition (cf. Definition \ref{def:mcp} and Remark \ref{rmk:SIUUUUUUU}) for $x=\e$.
\end{remark}

\begin{proof}[Proof of Proposition \ref{prop:phoenix}]
    The ``only if'' part is a consequence of Proposition \ref{prop:MCPgeneral} and Proposition \ref{prop:stronglyJacpositivity}. In fact, in the case we are considering, we have that $\mathsf{D}_0= \R/2 \pi_{\Omega^\circ}\mathbb{Z}$, $\mathscr R = \mathscr{U}$ and the reduced Jacobian $\J_R(\cdot,\cdot)$ is defined and continuous on $\mathcal{U}$. 
    
    For the ``if'' part, we assume that \eqref{eq:equivalent_reducedJ2} holds. Given any Borel set $A \subset \hei$, we observe that $A\subset \{(x,y,z)\in\hei: z\neq 0,(x,y)\neq(0,0)\}$ up to a $\Leb^3$-null set. In particular, according to the last part of Proposition \ref{prop:summary}, there exists a Borel set $B\subset \mathscr{U}$ such that $A = G_1(B)$ up to a $\Leb^3$-null set and $G_1$ is an homeomorphism between $B$ and $G_1(B)$. Moreover, Remark \ref{rmk:chesbatti} guarantees that $G_t$ is injective for every $t\in(0,1]$. Then, using the area formula and \eqref{eq:equivalent_reducedJ2}, we deduce that
    \begin{multline*}
            \Leb^3(M_t(\e,A)) = \Leb^3(G_t(B)) = \int_{B} |\J_t(r,\phi,\omega)| \de r\de \phi \de \omega = \int_{B} \frac{r^3 t}{\omega^4}|\J_R(\phi,t\omega)| \de r\de \phi \de \omega \\
            \geq t^N \int_{B} \frac{r^3}{\omega^4}|\J_R(\phi,\omega)| \de r\de \phi \de \omega  = t^N\int_{B} |\J_1(r,\phi,\omega)| \de r\de \phi \de \omega = t^N\Leb^3(G_1(B)) = t^N \Leb^3(A).
    \end{multline*}
    Finally, we note that this inequality is sufficient to prove the $\MCP(0,N)$ condition, according to Remarks \ref{rmk:SIUUUUUUU} and \ref{rmk:sneaky}.
\end{proof}
\noindent We provide an alternative version of Proposition \ref{prop:phoenix}, under the further assumption that the norm $\normdot$ is strongly convex. In this case, Proposition \ref{prop:phoenix}, combined with Proposition \ref{prop:stronglyJacpositivity}, guarantees that the $\MCP(0,N)$ condition holds if and only if 
\begin{equation}\label{eq:equivalent_reducedJ3}
        \J_R (\phi, \omega t)\geq t^{N-1} \J_R (\phi, \omega),
    \end{equation}
    for every $(\phi,\omega)\in \mathcal{U}$ and every $t\in [0,1]$.
Moreover, when $\normdot$ is $C^1$ and strongly convex, Proposition \ref{prop:regularityCo} ensures the map $C_\circ$ is Lipschitz and differentiable $\Leb^1$-almost everywhere. We call $\mathsf{D}_1\subset \R/2\pi_{\Omega^\circ}\mathbb{Z}$ the (full-measure) set of differentiability points.

\begin{corollary}
    \label{prop:diffcharacMCP}
    Let $\hei$ be the sub-Finsler Heisenberg group, equipped with a $C^1$ and strongly convex norm $\normdot$, and with the Lebesgue measure $\Leb^3$. Then, the metric measure space $(\hei, \di, \Leb^3)$ satisfies the $\mathsf{MCP}(0, N)$ if and only if
    \begin{equation}
        \label{eq:nec&sufMCPdJacineq}
        N(\phi, \omega) := 1 + \frac{\omega \partial_\omega \mathcal{J}_R(\phi, \omega)}{\mathcal{J}_R(\phi, \omega)} \leq N,
    \end{equation}
    for all $(\phi,\omega)\in \mathcal{U}$ with $\phi+\omega\in \mathsf{D}_1$.
\end{corollary}

\begin{proof}
    It is sufficient to prove that \eqref{eq:nec&sufMCPdJacineq} is equivalent to \eqref{eq:equivalent_reducedJ3}. Assume the latter holds and let $(\phi,\omega)\in\mathcal{U}$ such that $\phi + \omega \in \mathsf{D}_1$. Then, $\mathcal{J}_R(\phi, \cdot)$ is differentiable at $\omega$ and we obtain that
    \begin{align*}
    \frac{\partial_\omega \mathcal{J}_R(\phi, \omega)}{\mathcal{J}_R(\phi, \omega)} ={}& \frac{\diff}{\diff \omega} \log(\mathcal{J}_R(\phi, \omega)) = \lim_{z \to \omega} \frac{\log(\mathcal{J}_R(\phi, z)) - \log(\mathcal{J}_R(\phi, \omega))}{z - \omega}\\
    ={}& \frac{1}{\omega}\lim_{t \to 1^-} \frac{\log(\mathcal{J}_R(\phi, \omega t)) - \log(\mathcal{J}_R(\phi, \omega))}{t - 1}\\
    \leq{}& \frac{1}{\omega}\lim_{t \to 1^-} \frac{(N - 1)\log t}{t - 1} = \frac{N - 1}{\omega},
\end{align*} 
where the inequality follows from \eqref{eq:equivalent_reducedJ3}.

For the other implication, observe that, by Proposition \ref{prop:regularityCo},
$C_\circ$ is Lipschitz and therefore map $\mathcal{J}_R(\phi, \cdot)$ is Lipschitz and thus absolutely continuous. Then, \eqref{eq:nec&sufMCPdJacineq} implies that
    \[
    \int_{\omega t}^\omega \frac{\diff}{\diff z} \log(\mathcal{J}_R(\phi, z)) \diff z \leq (N - 1) \int_{\omega t}^\omega \frac{\diff}{\diff z} \log z \diff z,
    \]
    which simplifies to \eqref{eq:equivalent_reducedJ3}.
\end{proof}

 We conclude the section with a useful criterion that ensures the validity of the $\MCP(0,N)$ condition for some $N \in (1,\infty)$. 

\begin{prop}
    \label{prop:MCPlimsup}
    Let $\hei$ be the sub-Finsler Heisenberg group, equipped with a $C^1$ and strongly convex norm $\normdot$, and with the Lebesgue measure $\Leb^3$. Then, the metric measure space $(\hei, \di, \Leb^3)$ satisfies the $\mathsf{MCP}(0, N)$ for some $N \in (1,\infty)$ if and only if for any $\phi_\infty\in\R/2\pi_{\Omega^\circ}\mathbb Z$,
    \begin{equation}
        \label{eq:MCPdiffineqBoundary}
        \limsup_{(\phi,\omega)\to(\phi_{\infty},0)}N(\phi,\omega)<+\infty,
    \end{equation}
    where the $\limsup$ is taken over all $(\phi,\omega)\in \mathcal{U}$ such that $\phi+\omega\in \mathsf{D}_1$.
\end{prop}

\begin{proof}
    The ``only if'' part is follows immediately from Corollary \ref{prop:diffcharacMCP}, as \eqref{eq:nec&sufMCPdJacineq} implies \eqref{eq:MCPdiffineqBoundary}. Let us investigate the ``if'' part of the statement. For $k \in \mathbb{N}$, denote by $\mathcal{U}_k$ the set
    \[
    \mathcal{U}_k := \{ (\phi, \omega) \in \mathcal{U} \, |\, \phi + \omega \in \mathsf{D}_1 \text{ and } N(\phi, \omega) \geq k \}.
    \]
    By Corollary \ref{prop:diffcharacMCP}, satisfying the $\mathsf{MCP}(0, N)$ condition for some $N\in (1,\infty)$ is equivalent to the existence of a $k \in \mathbb{N}$ such that $\mathcal{U}_k = \emptyset$. Assume by contradiction that for all $k \in \mathbb{N}$, the set $\mathcal{U}_k$ is not empty. Then, there exists a sequence $\{\phi_k, \omega_k\}_{k\in \N} \subset \mathcal{U}$ such that $\lim_{k \to +\infty} N(\phi_k, \omega_k) = +\infty$ and $\phi_k + \omega_k \in \mathsf{D}_1$ for all $k \in \mathbb{N}$. Up to extracting a converging subsequence, we can assume that $(\phi_k, \omega_k)$ converges to $(\phi_\infty, \omega_\infty)$.
    Firstly, we claim that $\omega_{\infty} =0, \pm 2 \pi_{\Omega^\circ}$. Indeed, if this were not the case, keeping in mind \eqref{eq:dJK}, \eqref{def:functionK} and that $C_\circ$ is Lipschitz, we would have that
    \begin{align*}
        N(\phi_k, \omega_k) &= 1 + C'_\circ(\phi_k + \omega_k) \frac{\omega_k \mathcal{K}(\phi_k, \omega_k)}{\J_R(\phi_k, \omega_k)} \\
        &\leq 1 + \mathsf{Lip}(C_\circ) \frac{\omega_k \mathcal{K}(\phi_k, \omega_k)}{\J_R(\phi_k, \omega_k)} \to 1 + \mathsf{Lip}(C_\circ) \frac{\omega_\infty \mathcal{K}(\phi_\infty, \omega_\infty)}{\J_R(\phi_\infty, \omega_\infty)} < \infty \text{  as } k \to +\infty,
    \end{align*}
    where we used continuity of $\omega \mathcal{K}(\phi, \omega)/\J_R(\phi, \omega)$ and that the last term is finite because Proposition \ref{prop:stronglyJacpositivity} guarantees that $\J_R(\phi_\infty, \omega_\infty) > 0$. Moreover, we can exclude that $\omega_{\infty} = \pm 2 \pi_{\Omega^\circ}$ because $\omega\mathcal{K}(\phi,\omega)<0$ near $\omega=\pm 2\pi_{\Omega^\circ}$,
    while $C'_\circ(\phi+\omega)\geq 0$ and $\J_R(\phi,\omega)>0$ by Proposition \ref{prop:stronglyJacpositivity}. Finally, $\omega_\infty=0$ is excluded by the assumption \eqref{eq:MCPdiffineqBoundary} and we obtain the desire contradiction.
\end{proof}

\section{Failure of the measure contraction property in the sub-Finsler Heisenberg group}\label{sec:nomcp}

In this section we study the validity of the measure contraction property $\MCP(K,N)$ in the sub-Finsler Heisenberg group $\hei$, equipped with the norm $\normdot$. In particular, we are going to prove Theorem \ref{thm:MAINnomcp}, as a result of the combination of Theorem \ref{thm:noCDnonC1} and Theorem \ref{thm:noMCPnonstrctly}, and Theorem \ref{thm:MAINmcp}.

\subsection{Failure for non-\texorpdfstring{$C^1$}{C1} norms}\label{sec:nomcp1}

In this first subsection we prove Theorem \ref{thm:MAINnomcp} for every non-$C^1$ reference norm $\normdot$. The proof we are going to present follows the same strategy developed in \cite[Thm.\ 5.26]{magnabosco2023failure}, properly adapted taking advantage of Proposition \ref{prop:summary}, in order to include the case of non-strictly convex norms. As in \cite[Thm.\ 5.26]{magnabosco2023failure}, the main idea is to exploit a branching behaviour of geodesics which is caused by the singularities of the reference norm $\normdot$.

\begin{theorem}\label{thm:noCDnonC1}
    Let $\hei$ be the \sF Heisenberg group, equipped with a norm $\normdot$ which is not $C^1$, and let $\m$ be a smooth measure on $\hei$. Then, the metric measure space $(\hei, \di, \m)$ does not satisfy the measure contraction property $\MCP(K,N)$ for every $K\in \R$ and $N\in (1,\infty)$.
\end{theorem}
\begin{proof}
According to Proposition \ref{prop:propunderduality}, since $\normdot$ is not $C^1$, its dual norm $\normdot_*$ is not strictly convex. In particular, there exists a straight segment contained in the sphere $S^{\norm{\cdot}_*}_1(0)=\partial\Omega^\circ$. Since the distribution generating the Heisenberg group is invariant under rotations around the $z$-axis, we can assume without losing generality that this segment is vertical in $\R^2\cap\{x>0\}$, i.e. there exists $\bar x \in \R_{>0}$ and an interval $I:=[y_0,y_1]\subset \R$ such that 
\begin{equation*}
    \{\bar x\}\times I \subset \partial \Omega^\circ.
\end{equation*}
Moreover, we can take the interval $I$ to be maximal, namely for every $y\not\in I$ we have $(\bar x,y)\not\in \Omega^\circ$. Let $\psi_0,\psi_1 \in \R/2\pi_{\Omega^\circ}\mathbb{Z}$ be such that $Q_{\psi_0}=(\bar x,y_0)$ and $Q_{\psi_1}=(\bar x,y_1)$ (see Figure \ref{fig:MCP1}). Observe that, for every $y\in I$, it holds that 
\begin{equation}\label{eq:sincosflat}
    (\bar x, y)= Q_{\psi_0 + (y- y_0)\bar x},
\end{equation}
as a consequence, we deduce that 
\begin{equation}\label{eq:120.5}
   \cosomp(\psi_0 + (y-y_0)\bar x)= \bar x \quad\text{and}\quad \sinomp(\psi_0 + (y-y_0)\bar x)= y, \qquad \text{for }y\in I.
\end{equation}
Consider $\bar \phi := \frac 12 (\psi_0 + \psi_1)$ and $y_2= \frac 12 (y_0+y_1)$, so that $(\bar x, y_2)= Q_{\bar \phi}$, according to \eqref{eq:sincosflat}. 

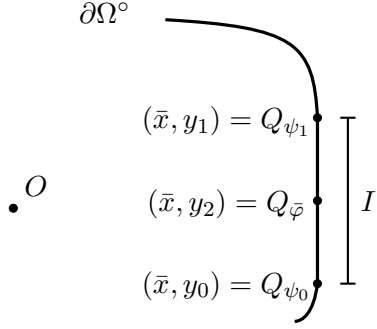
\begin{figure}
    \centering
    \begin{tikzpicture}
        \filldraw[black] (-2,-0.5) circle (1.5pt);
        \filldraw[black] (2,-1.5) circle (1.5pt);
        \filldraw[black] (2,-0.4) circle (1.5pt);
        \filldraw[black] (2,0.7) circle (1.5pt);
        %\filldraw[black] (0.6,1.95) circle (1.5pt);
        
        \draw[very thick] (1.7,-2) .. controls (1.8,-2) and (1.95,-1.9) ..(2,-1.5) -- (2,0.7).. controls (1.95,1.7) and (1.7,1.9) ..(0,2);
        \draw [thick](2.4,-1.5)--(2.4,0.7);
        \draw [thick](2.3,-1.5)--(2.5,-1.5);
        \draw [thick](2.3,0.7)--(2.5,0.7);
        \node at (-1.7,0.2)[label=south:$O$] {};
        \node at (-0.8,2.5)[label=south:$\partial\Omega^\circ$] {};
        \node at (0.8,1.12)[label=south:${(\bar x, y_1)= Q_{\psi_1}}$] {};
        \node at (0.8,0.02)[label=south:${(\bar x, y_2)= Q_{\bar \phi}}$] {};
        \node at (0.8,-1.06)[label=south:${(\bar x, y_0)= Q_{\psi_0}}$] {};
        \node at (2.3,-0.4)[label=east:$I$] {};
        
    \end{tikzpicture}
    \caption{The flat part of $\partial\Omega^\circ$.}
    \label{fig:MCP1}
\end{figure}

Fix $\bar r >0$ and $\bar \omega> \pi_{\Omega^\circ}$ such that $(\bar r ,\bar \phi,\bar \omega) \in \mathscr R$ and take a neighborhood $\mathscr A \subset \mathscr R$ of $(\bar r, \bar \phi,\bar \omega)$, such that $\omega > \pi_{\Omega^\circ}$ for every $(r,\phi,\omega)\in \mathscr A$. Proposition \ref{prop:summary} guarantees that, for every $(r,\phi,\omega)\in \mathscr A$, the curve $\gamma_{(r,\phi,\omega)}$ is the unique geodesic connecting $\e$ and $G_1(r,\phi,\omega)$.
Therefore, we deduce that 
\begin{equation}
     M_t\big(\{\e\},G_1 (\mathscr A)\big) = G_t (\mathscr A),\qquad\text{for every }t\in[0,1].
\end{equation}
%Since $\mathscr D$ has full $\Leb^1$-measure in $\R$, Fubini's theorem implies that for $\Leb^{3}$-a.e. $(\phi,\omega,r)\in\mathscr A$, we have $\phi, \phi+t\omega\in\mathscr D$. 
As observed in Remark \ref{rmk:Jdefinedae}, we have $\phi, \phi+t\omega, \phi+\omega\in\mathsf D_0$ for $\Leb^{3}$-a.e. $(r,\phi,\omega)\in\mathscr A$. 
Moreover, at those points, Proposition \ref{lem:reducedJ} ensures that
\begin{equation}
\label{eq:positivity_jacobian}
    \J_1 (r, \phi,  \omega) = \frac{ r^3 }{ \omega^4} \J_R ( \phi,  \omega ) > 0.
\end{equation}
Therefore, since $G_1(\cdot)$ is locally Lipschitz on $\mathscr A$, by the area formula and \eqref{eq:positivity_jacobian}, we deduce that 
\begin{equation}
\label{eq:positive_measure_set}
    \m \big(G_1(\mathscr A)\big) >0.
\end{equation}
We can now disprove the measure contraction property $\MCP(K,N)$, taking as marginals
\begin{equation}
\mu_0:=\delta_\e\qquad\text{and}\qquad\mu_1:= \frac{1}{\Leb^3(G_1(\mathscr A))}\, \Leb^3|_{G_1(\mathscr A)}.
\end{equation}
 According to Remark \ref{rmk:SIUUUUUUU}, it is enough to contradict \eqref{eq:tj_pantaloncini} with $ A'=A= G_1(\mathscr A)$. In particular, we are going to find $t_0\in(0,1)$ such that
\begin{equation}
\label{eq:very_thin_support}
    M_t(\{\e\},A) = G_t (\mathscr A) \subset \{ y=0,z=0\}, \qquad \forall\,t<t_0.
\end{equation}
 To this aim, fix any $(r,\phi,\omega)\in \mathscr A$ and note that, for every $t<\frac{\psi_1-\phi}{\omega}$, \eqref{eq:120.5} implies that
\begin{equation*}
    \cosomp(\phi+\omega  t ) = \bar x \qquad\text{and}\qquad \sinomp(\phi+\omega  t )= \sinomp(\phi) + \frac{\omega  t}{\bar x}.
\end{equation*}
From these relations and \eqref{eq:xyz}, it follows immediately that 
\begin{equation*}
    y_t(r,\phi,\omega)=0\qquad\text{and}\qquad z_t(r, \phi,\omega) = 0,
\end{equation*}
for every $s<\frac{\psi_1-\phi}{\omega}$. Then, according to our choice of $\mathscr A$, \eqref{eq:very_thin_support} holds with $t_0=\frac{\psi_1-\phi}{\pi_{\Omega^\circ}}$.
\end{proof}

\begin{remark}
     Observe that the previous theorem generalizes \cite[Thm.\ 5.26]{magnabosco2023failure}. While the strategy of its proof is similar, we highlight that now we are able to prove \eqref{eq:positive_measure_set} using Proposition \ref{lem:reducedJ} and working in a neighborhood $\mathscr A$ where $\omega > \pi_{\Omega^\circ}$. This is in contrast with the geometric construction of \cite[Thm.\ 5.26]{magnabosco2023failure} which was local around the flat part of $\partial\Omega^\circ$.  
\end{remark}

\begin{remark}
    As already observed in \cite[Rmk.\ 5.27]{magnabosco2023failure}, the construction presented in the last proof highlights the existence of a family of branching geodesics, originating from the presence of a flat part in $\partial\Omega^\circ$. In particular, when $\hei$ is equipped with a singular norm, geodesics can branch, although they are unique (in the sense of Proposition \ref{prop:summary}). 
\end{remark}

\subsection{Failure for \texorpdfstring{$C^1$}{C1} and non-strictly convex norms} \label{sec:nomcp2}

In this subsection we complete another step for the proof of Theorem \ref{thm:MAINnomcp}, considering the \sF Heisenberg group $\hei$, equipped with a non-strictly convex $C^1$ norm $\normdot$. %According to Proposition \ref{prop:propunderduality}, the dual norm $\normdot_*$ is $C^{1,1}$. 
The main idea behind our strategy is to exploit the discontinuities of the generalized trigonometric functions, caused by the ``flat parts'' of $\partial \Omega$ (cf. Lemma \ref{lem:limits}). According to Proposition \ref{prop:NONO}, these will result in discontinuities of the Jacobian of the exponential map, which, if properly utilized, will allow us to show the failure of the $\MCP(K,N)$ condition.

In particular, in this case there exists a straight segment $L$ contained in the sphere $S^{\norm{\cdot}}_1(0)=\partial\Omega$. As done in the proof of Theorem \ref{thm:noCDnonC1}, we can assume without losing generality that this segment is vertical in $\R^2\cap\{x>0\}$, i.e. there exists $\bar x \in \R_{>0}$ and a maximal interval $I:=[y_0,y_1]\subset \R$ such that $L= \{\bar x\} \times I \subset \partial \Omega$. 
%Moreover, we can take the interval $I$ to be maximal, namely for every $y\not\in I$ we have $(\bar x,y)\not\in \Omega$.

\begin{lemma}\label{lem:limits}
    We have that 
    \begin{equation}
        \lim_{\begin{smallmatrix} \psi \uparrow 0 & \\ \psi \in \mathsf{D}_0  \end{smallmatrix}} \cosom(\psi_\circ) = \lim_{\begin{smallmatrix} \psi \downarrow 0 & \\ \psi \in \mathsf{D}_0  \end{smallmatrix}} \cosom(\psi_\circ) = \bar x \quad \text{and}\quad \lim_{\begin{smallmatrix} \psi \uparrow 0 & \\ \psi \in \mathsf{D}_0  \end{smallmatrix}} \sinom(\psi_\circ)=y_0<y_1= \lim_{\begin{smallmatrix} \psi \downarrow 0 & \\ \psi \in \mathsf{D}_0  \end{smallmatrix}} \sinom(\psi_\circ) .
    \end{equation}
\end{lemma}

\begin{proof}
    Let $v\in \partial\Omega^\circ$ be the dual vector of every vector in $L$, i.e. 
    \begin{equation}
        v = \diff_{(\bar x,y)} \normdot, \qquad \text{for every }y \in I. 
    \end{equation}
    Notice that $v$ is an horizontal vector in $\R^2$, therefore $v=Q_{0}$. In particular, setting $\theta_0$ and $\theta_1$ to be such that $(\bar x, y_0)= P_{\theta_0}$ and $(\bar x, y_1)= P_{\theta_1}$, we have
    \begin{equation}
    \label{eq:image_Ccirc}
        C^\circ(\theta)=0\qquad\text{if and only if}\qquad \theta \in [\theta_0,\theta_1].
    \end{equation}
    Fix $\varepsilon>0$ and let $\tilde\theta_0 < \theta_0$ and $\tilde \theta_1 > \theta_1$ such that 
    \begin{equation}
    \begin{split}
     \cosom(\theta) &\in (\bar x- \varepsilon, \bar x],\!\!\quad\qquad \text{for every } \theta\in [\tilde\theta_0, \theta_0] \cup[\theta_1,\tilde \theta_1],\\
    \sinom(\theta) &\in (y_0- \varepsilon, y_0],\qquad \text{for every } \theta\in [\tilde\theta_0, \theta_0],\\
    \sinom(\theta) &\in [y_1, y_1 +\varepsilon),\qquad \text{for every } \theta\in [\theta_1,\tilde \theta_1].
    \end{split}
    \end{equation}
    Call $\psi_0 := C^\circ(\tilde\theta_0)< 0$ and $\psi_1 := C^\circ(\tilde\theta_1)> 0$, cf. \eqref{eq:image_Ccirc}. Now, consider $\psi \in (\psi_0, 0) \cap \mathsf{D}_0$ (so that $C_\circ$ is a single-valued map at $\psi$) and note that $C_\circ(\psi)=\psi_\circ\in [\tilde \theta_0, \theta_0]$ and therefore it holds that 
    \begin{equation}
    \label{eq:magical_inclusion1}
        \cosom (\psi_\circ) \in (\bar x- \varepsilon, \bar x],\qquad \sinom(\psi_\circ) \in (y_0- \varepsilon, y_0].
    \end{equation}
    With an analogous argument, we deduce that for every $\psi \in (0,\psi_1)\cap \mathsf{D}_0$, we have that $\psi_\circ\in[\theta_1,\tilde \theta_1]$ and
    \begin{equation}
    \label{eq:magical_inclusion2}
         \cosom (\psi_\circ) \in (\bar x- \varepsilon, \bar x],\qquad \sinom(\psi_\circ) \in [y_1, y_1 +\varepsilon).
    \end{equation}
    By the arbitrariness of $\varepsilon>0$, the combination of \eqref{eq:magical_inclusion1} and \eqref{eq:magical_inclusion2} yields the thesis.
\end{proof}

\begin{theorem}\label{thm:noMCPnonstrctly}
    Let $\hei$ be the \sF Heisenberg group, equipped with a norm $\normdot$ which is $C^1$ and not strictly convex, and let $\m$ be a smooth measure on $\hei$. Then, the metric measure space $(\hei, \di, \m)$ does not satisfy the measure contraction property $\MCP(K,N)$ for every $K\in \R$ and $N\in (1,\infty)$.
\end{theorem}

\noindent Before going through the proof, we recall from section \ref{subsec:jacobian} that there is a $\Leb^1$-full measure set $\bar {\mathsf{D}}_0\subset\mathsf{D}_0$ such that, for every $\phi \in \bar {\mathsf{D}}_0$, we have that $\phi+ \psi \in \mathsf{D}_0$ and $(\phi,\psi)$ is a Lebesgue point for $\J_R$, for $\Leb^1$-almost every $\psi \in \R$.

\begin{proof}[Proof of Theorem \ref{thm:noMCPnonstrctly}]
    According to Lemma \ref{lem:limits}, the following limits exist
    \begin{equation}
        y_0= \lim_{\begin{smallmatrix} \psi \uparrow 0 & \\ \psi \in \mathsf{D}_0  \end{smallmatrix}} \sinom(\psi_\circ) \qquad \text{and}\qquad y_1 =\lim_{\begin{smallmatrix} \psi \downarrow 0 & \\ \psi \in \mathsf{D}_0  \end{smallmatrix}} \sinom(\psi_\circ),
    \end{equation}
    and we know that $\delta:= y_1 - y_0>0$. Observe that, by the symmetry of the norm, $\cosomp(\pi_{\Omega^\circ})= - \cosomp(0)<0$ and $\sinomp( \pi_{\Omega^\circ})= \sinomp( 0)= 0$. Let $\rho>0$ be sufficiently small, then, by continuity of the trigonometric functions, there exists $\varepsilon>0$ such that, for every $\phi\in [\pi_{\Omega^\circ}-\varepsilon,\pi_{\Omega^\circ}]$,
    \begin{equation}
    \label{eq:continuity_trig_fun}
        \sinomp(\phi)\in \left[0,\frac{\rho}{M}\right]\qquad\text{and}\qquad -\cosomp(\phi)\in [\cosomp(0)-\rho,\cosomp(0)+\rho],
    \end{equation}
    where $M:=\max_\theta \sinom(\theta)$ is positive and finite. Thus, by the Pythagorean identity \eqref{eq:pytagorean} and the first relation in \eqref{eq:continuity_trig_fun}, we deduce that 
    \begin{equation}
        \cosom(\phi_\circ)\cosomp(\phi)\in [1-\rho,1+ \rho], \qquad\forall\,\phi\in [\pi_{\Omega^\circ}-\varepsilon,\pi_{\Omega^\circ}]\cap\mathsf{D}_0.
    \end{equation}
    Combining this with the second relation in \eqref{eq:continuity_trig_fun}, we deduce that 
    \begin{equation}
        -\cosom(\phi_\circ)\in \left[\frac{1-\rho}{\cosomp(0)+\rho},\frac{1+\rho}{\cosomp(0)-\rho}\right],\qquad\forall\,\phi\in [\pi_{\Omega^\circ}-\varepsilon,\pi_{\Omega^\circ}]\cap\mathsf{D}_0.
    \end{equation}
    In addition, up to restricting $\varepsilon>0$, by continuity, we may also assume that 
    \begin{equation}
        \frac{2}{\pi_{\Omega^\circ}}\sinomp(\phi)<\frac{1-\rho}{\cosomp(0)+\rho},\qquad\forall\,\phi\in [\pi_{\Omega^\circ}-\varepsilon,\pi_{\Omega^\circ}]\cap\mathsf{D}_0.
    \end{equation}
    In conclusion, we can find $\varepsilon>0$ such that   
    \begin{equation}\label{eq:METRO}
        - \cosom(\phi_\circ) > \frac{2}{\pi_{\Omega^\circ}} \sinomp( \phi) \qquad \text{for every }\phi \in [\pi_{\Omega^\circ}- \varepsilon, \pi_{\Omega^\circ}] \cap \mathsf{D}_0.
    \end{equation}
    
    Fix $\bar \phi \in [\pi_{\Omega^\circ}- \varepsilon, \pi_{\Omega^\circ}) \cap \bar {\mathsf{D}}_0$, call $\bar k (\psi):= -\sinomp( \bar\phi) - \psi \cosom( \bar \phi_\circ)$ and observe that \eqref{eq:METRO} ensures that
    \begin{equation}\label{eq:it'slit}
        \bar k (\psi) > -\sinomp( \bar\phi) - \frac{\pi_{\Omega^\circ}}{2} \cosom( \bar \phi_\circ) >0 \qquad \text{for every }\psi \geq \frac{\pi_{\Omega^\circ}}{2}.
    \end{equation}
    Note that the function $\psi \mapsto \J_R (\bar \phi, \psi)$ is not continuous at $2\pi_{\Omega^\circ} - \bar \phi$. In fact, looking at the explicit expression \eqref{eq:reduced_Jac} of the reduced Jacobian, thanks to Lemma \ref{lem:limits}, we have that every term of $\J_R(\phi,\psi)$ is continuous (in $\psi$, at $2\pi_{\Omega^\circ} - \bar \phi$) except for 
    \begin{equation}
        - \sinom\big(( \bar\phi +  \psi)_\circ\big) \sinomp( \bar\phi)
         - \psi \sinom\big(( \bar \phi +  \psi)_\circ\big) \cosom( \bar \phi_\circ) =  \bar k(\psi) \sinom\big(( \phi +  \psi)_\circ\big). 
    \end{equation}
    As a consequence, calling 
    \begin{equation}\label{eq:Ttime}
          J_1:=\lim_{\begin{smallmatrix} \psi \uparrow (2 \pi_{\Omega^\circ} - \bar \phi)& \\ \bar \phi + \psi \in \mathsf{D}_0 \end{smallmatrix}}\J_R (\bar\phi, \psi) \qquad \text{and} \qquad J_2:= \lim_{\begin{smallmatrix} \psi \downarrow (2 \pi_{\Omega^\circ} - \bar \phi)& \\ \bar \phi + \psi \in \mathsf{D}_0 \end{smallmatrix}}\J_R (\bar\phi, \psi), 
    \end{equation}
    we deduce that 
    \begin{equation}
        J_2 > J_2 - \bar k (2 \pi_{\Omega^\circ} - \bar \phi) \cdot  \delta = J_1 >0
    \end{equation}
    where the first inequality is a consequence \eqref{eq:it'slit}, while the second one follows from Proposition \ref{lem:reducedJ}. Recall that, since we chose $\bar\phi\in \bar{\mathsf{D}}_0$, we have that, for $\Leb^1$-almost every $\psi \in \R$, $\bar\phi + \psi \in \mathsf{D}_0$ and $(\phi,\psi)$ is a Lebesgue point for $\J_R$. Now, we consider the quantity
    \begin{equation}
        \bar t:=\frac{ J_2 - \frac 23 \bar k (2 \pi_{\Omega^\circ} - \bar \phi)}{ J_2 - \frac13\bar k (2 \pi_{\Omega^\circ} - \bar \phi)}< 1.
    \end{equation}
    Keeping in mind \eqref{eq:Ttime}, for any $N\in(1,\infty)$, we can find $\psi_1 < 2\pi_{\Omega^\circ} - \bar \phi$ and $\psi_2 > 2\pi_{\Omega^\circ} - \bar \phi$ such that $\bar\phi+ \psi_1, \bar\phi+ \psi_2 \in \mathsf{D}_0$, $(\bar \phi,\psi_1)$ and $(\bar \phi,\psi_2)$ are Lebesgue points for $\J_R$, $\frac{\psi_1}{\psi_2}=:t> \bar t^{\frac{1}{N-1}}$, 
    \begin{equation}
        \J_R (\bar\phi, \psi_1) < J_1 + \frac13\bar k (2 \pi_{\Omega^\circ} - \bar \phi) = J_2 - \frac23\bar k (2 \pi_{\Omega^\circ} - \bar \phi) \quad \text{and} \quad \J_R (\bar\phi, \psi_2) > J_2 - \frac13\bar k (2 \pi_{\Omega^\circ} - \bar \phi).
    \end{equation}
    Now, taking $\omega=\psi_2$, we conclude that 
    \begin{equation}
    \begin{split}
        \J_R (\bar\phi, \omega t) &= \J_R (\bar\phi, \psi_1) < J_2 - \frac23\bar k (2 \pi_{\Omega^\circ} - \bar \phi) = \bar t \left(J_2 - \frac13\bar k (2 \pi_{\Omega^\circ} - \bar \phi)\right) < \bar t \cdot \J_R (\bar\phi, \psi_2 ) \\ &< t^{N-1} \cdot \J_R (\bar\phi, \omega ).
    \end{split}
    \end{equation}
    Then, we can apply Corollary \ref{cor:NO-MCP} and deduce the thesis, thanks to the arbitrariness of $N$.
\end{proof}

\subsection{Failure for \texorpdfstring{$C^1$}{C1} and strictly but not strongly convex norms} \label{sec:nomcp3}

In this subsection, we consider the \sF Heisenberg group $\hei$, equipped with a norm $\normdot$ which is $C^1$ and strictly convex, but not strongly convex and prove the failure of the $\MCP(K,N)$ condition in these structures. In this way, we complete the proof of Theorem \ref{thm:MAINnomcp}. The strategy used in this section is similar to the one developed to prove Theorem \ref{thm:noMCPnonstrctly}. But, instead of looking at discontinuities of the Jacobian, which is continuous in this case, we exploit the existence of suitable parameters at which the Jacobian has big variations. This case includes the Heisenberg group equipped with the sub-Finsler $\ell^p$-norm when $p \in (2, +\infty)$, and recovers the result in \cite[Thm.\ A.2]{borzatashiro}.

    \begin{theorem}
    \label{thm:noMCPC1strictlynotstrongly}
    Let $\hei$ be the sub-Finsler Heisenberg group, equipped with a $C^1$ and strictly but not strongly convex norm $\normdot$, and with a smooth measure $\m$. Then, the metric measure space $(\hei, \di, \m)$ does not satisfy the $\mathsf{MCP}(K, N)$ for every $K\in \R$ and $N\in (1,\infty)$.
\end{theorem}

     \noindent Before going to the proof, we observe that the characterization \eqref{eq:equivalent_reducedJ2} of Proposition \ref{prop:phoenix} is equivalent to
    \[
    \frac{\log|\mathcal{J}_R(\varphi, \omega t)| - \log|\mathcal{J}_R(\varphi, \omega)|}{\omega t - \omega}
    \leq \frac{(N - 1)\log t}{\omega t - \omega},
    \]
    for every $(\varphi,\omega)\in \mathcal{U}$ and $t\in[0,1]$. In particular, if $(\hei, \di, \Leb^3)$ satisfies the $\mathsf{MCP}(0, N)$ for some $N \in (1,\infty)$, then

\begin{multline}\label{eq:MCPsuffC1strictly}
  \limsup\limits_{\substack{(\omega_1,\omega_2)\to(\omega, \omega) \\ \omega_1 \neq \omega_2}} \frac{\log|\mathcal{J}_R(\varphi, \omega_1)| - \log|\mathcal{J}_R(\varphi, \omega_2)|}{\omega_1  - \omega_2} \\
  \leq{} (N - 1) \limsup\limits_{\substack{(\omega_1,\omega_2)\to(\omega, \omega) \\ \omega_1 \neq \omega_2}} \frac{\log |\omega_1| - \log |\omega_2|}{\omega_1 - \omega_2} = \frac{N - 1}{\omega} <  +\infty,
\end{multline}
for every $(\varphi,\omega)\in\mathcal{U}$. The last equality is justified since
\[
\lim\limits_{\substack{(x,y)\to(a, a) \\ x \neq y}} \frac{f(x) - f(y)}{x - y} = f'(a),
\]
whenever $f$ is $C^1$ in a neighbourhood of $a$ (by the mean value theorem). In the proof, we will also use the property
\begin{equation}\label{eq:limsupworks}
\limsup\limits_{\substack{(x,y)\to(a, a) \\ x \neq y}} \frac{f(g(x)) - f(g(y))}{g(x) - g(y)} = \limsup\limits_{\substack{(x,y)\to(g(a), g(a)) \\ x \neq y}} \frac{f(x) - f(y)}{x - y}
\end{equation}
whenever $g$ is strictly monotone and continuous.

\begin{proof}[Proof of Theorem \ref{thm:noMCPC1strictlynotstrongly}]

    According to Proposition \ref{prop:regularityCo}, we know that the angle correspondence $C_\circ$ is strictly increasing and continuous but not Lipschitz.
   Then, we can fix $\psi \in \R/2\pi_{\Omega^\circ}\mathbb{Z}$ at which $C_\circ$ fails to be Lipschitz, i.e. there exists a sequence $\{\psi_n \}_{n\in \N} \to \psi$ such that
   \[
   \lim_{n \to \infty} \limsup_{\eta \to \psi_n} \frac{C_\circ(\psi_n) - C_\circ(\eta)}{\psi_n - \eta} = \infty, \text{ and thus } \limsup\limits_{\substack{(x,y)\to(\psi, \psi) \\ x \neq y}} \frac{C_\circ(x) - C_\circ(y)}{x - y} = +\infty.
   \]
 We set $\omega := \pi_{\Omega^\circ}$ and $\phi := \psi - \omega$ and we shall see that $\J_R(\phi, \cdot)$ fails to be Lipschitz at $\omega$.
   For every $\omega_1, \omega_2$ near $\omega$, we can explicitly compute that
   \begin{gather}
   \label{eq:lipschitzquotientJ}
       \begin{aligned}
       \frac{\J_R(\phi,\omega_1)-\J_R(\phi,\omega_2)}{\omega_1-\omega_2} ={}& - \left(\frac{\cos_{\Omega^\circ}(\phi +\omega_1)-\cos_{\Omega^\circ}(\phi +\omega_2)}{\omega_1-\omega_2}\right) \cos_{\Omega}(\phi_\circ) \\
       & - \left(\frac{\sin_{\Omega^\circ}(\phi+\omega_1)-\sin_{\Omega^\circ}(\phi+\omega_2)}{\omega_1-\omega_2}\right) \sin_{\Omega}(\phi_\circ) \\
       & -\cos_{\Omega}(\phi_\circ)\sin_{\Omega}((\phi+\omega_1)_\circ)+\sin_{\Omega}(\phi_\circ)\cos_{\Omega}((\phi+\omega_1)_\circ) \\
       & + \left(\frac{C_\circ(\phi+\omega_1)-C_\circ(\phi+\omega_2)}{\omega_1-\omega_2} \right) \times \left(\frac{F(\phi,\omega_1,\omega_2)}{C_\circ(\phi+\omega_1)-C_\circ(\phi+\omega_2)} \right),
   \end{aligned}
   \end{gather}
   where
   \begin{align*}
       F(\phi,\omega_1,\omega_2):=&  \big(\sin_{\Omega}(C_\circ(\phi+\omega_1))-\sin_{\Omega}(C_\circ(\phi+\omega_2))\big)(-\sin_{\Omega^\circ}(\phi)-\omega_2\cos_{\Omega}(\phi_\circ))\\
       &+ \big(\cos_{\Omega}(C_\circ(\phi+\omega_1))-\cos_{\Omega}(C_\circ(\phi+\omega_2))\big)(-\cos_{\Omega^\circ}(\phi)+\omega_2\sin_{\Omega}(\phi_\circ)).
   \end{align*}
Note that the continuity of $C_\circ$ yields that
   \[
   \lim\limits_{\substack{(\omega_1,\omega_2)\to(\omega, \omega) \\ \omega_1 \neq \omega_2}} \sin_{\Omega}((\phi+\omega_1)_\circ) = \sin_{\Omega}((\phi+\omega)_\circ)\text{ and } \lim\limits_{\substack{(\omega_1,\omega_2)\to(\omega, \omega) \\ \omega_1 \neq \omega_2}} \cos_{\Omega}((\phi+\omega_2)_\circ) = \cos_{\Omega}((\phi+\omega)_\circ).
   \]
    Moreover, since the generalised trigonometric functions $\cos_{\Omega^\circ}$ and $\sin_{\Omega^\circ}$ are $C^1$, we also have that
    \[
    \lim\limits_{\substack{(\omega_1,\omega_2)\to(\omega, \omega) \\ \omega_1 \neq \omega_2}} \frac{\cos_{\Omega^\circ}(\phi +\omega_1)-\cos_{\Omega^\circ}(\phi +\omega_2)}{\omega_1-\omega_2} = -\sin_{\Omega}((\phi + \omega)_\circ),
    \]
    and, similarly,
    \[
    \lim\limits_{\substack{(\omega_1,\omega_2)\to(\omega, \omega) \\ \omega_1 \neq \omega_2}} \frac{\sin_{\Omega^\circ}(\phi+\omega_1)-\sin_{\Omega^\circ}(\phi+\omega_2)}{\omega_1-\omega_2} = \cos_{\Omega}((\phi + \omega)_\circ).
    \]
   % It remains to study the superior of the last term of the last line in \eqref{eq:lipschitzquotientJ}. 
    Finally, since $C_\circ$ is strictly increasing and continuous, we find that
    \begin{multline*}
    \lim\limits_{\substack{(\omega_1,\omega_2)\to(\omega, \omega) \\ \omega_1 \neq \omega_2}} \frac{\sin_{\Omega}(C_\circ(\phi+\omega_1))-\sin_{\Omega}(C_\circ(\phi+\omega_2))}{C_\circ(\phi+\omega_1)-C_\circ(\phi+\omega_2)} \\
    = \lim\limits_{\substack{(\phi_1,\phi_2)\to((\phi + \omega)_\circ, (\phi + \omega)_\circ) \\ \phi_1 \neq \phi_2}} \frac{\sin_{\Omega}(\phi_1)-\sin_{\Omega}(\phi_2)}{\phi_1 - \phi_2}= \cos_{\Omega^{\circ}}(\phi + \omega)
    \end{multline*}
    where the first inequality follows from \eqref{eq:limsupworks}, while the second is a consequence of the differentiability of $\sin_\Omega$, cf. Proposition \ref{prop:difftrig}. In the same way, we also have that
    \[
    \lim\limits_{\substack{(\omega_1,\omega_2)\to(\omega, \omega) \\ \omega_1 \neq \omega_2}} \frac{\cos_{\Omega}(C_\circ(\phi+\omega_1))-\cos_{\Omega}(C_\circ(\phi+\omega_2))}{C_\circ(\phi+\omega_1)-C_\circ(\phi+\omega_2)} = -\sin_{\Omega^{\circ}}(\phi + \omega).
    \]
    As a consequence, keeping in mind \eqref{eq:symmetry1}, we obtain that
    \begin{align*}
        \lim\limits_{\substack{(\omega_1,\omega_2)\to(\omega, \omega) \\ \omega_1 \neq \omega_2}} \frac{F(\phi,\omega_1,\omega_2)}{C_\circ(\phi+\omega_1)-C_\circ(\phi+\omega_2)}  ={}& \cos_{\Omega^{\circ}}(\phi + \omega)(-\sin_{\Omega^\circ}(\phi)-\omega\cos_{\Omega}(\phi_\circ))\\
        &\ -\sin_{\Omega^{\circ}}(\phi + \omega) (-\cos_{\Omega^\circ}(\phi)+\omega\sin_{\Omega}(\phi_\circ))= \omega > 0.
    \end{align*}
    % \textbf{NEED TO MAKE SURE IT'S POSITIVE FOR ONE $\omega$} {\blue which is positive now because of the choice of $\omega$ above}.

    Putting everything together, since $C_\circ$ is not Lipschitz at $\psi = \phi +\omega$, we deduce that $\J_R(\phi,\cdot)$ fails to be Lipschitz at $\omega$. Furthermore, the logarithm function $\log$ is locally bi-Lipschitz and the value of $\J_R(\phi,\cdot)$ stays in a compact subset of $(0,+\infty)$ near $\omega$, therefore
    also $\log (\J_R(\phi,\cdot))$ fails to be Lipschitz at $\omega$.
    We conclude that the Heisenberg group is not $\mathsf{MCP}(0, N)$ for any $N \in (0,\infty)$, as \eqref{eq:MCPsuffC1strictly} is not satisfied. The thesis then follows from Proposition \ref{prop:Knoncipiace}.
\end{proof}

\section{Measure contraction property for \texorpdfstring{$C^1$}{C1} and strongly convex norms}\label{sec:afterstrongly}

In this section, we study the measure contraction property for the \sF Heisenberg group equipped with a norm $\normdot$ that is $C^1$ and strongly convex. Firstly, we provide further necessary (but not sufficient) conditions for the $\mathsf{MCP}$ to hold in this case.  Secondly, we identify sufficient (but not necessary) conditions for the validity of $\MCP$, providing estimates for the curvature exponent (cf. Definition \ref{def:curvatureexponent}). In particular, we show that $\MCP$ holds under $C^{1,1}$ regularity of the reference norm $\normdot$. Thirdly, we investigate possible necessary and sufficient conditions for $\MCP$ through several meaningful examples. This study relies on the behavior of (the derivative of) the angle correspondence $C_\circ$. We summarize the results of this section as follows.
% \todo[inline]{M,T: We propose the following modifications. Our motivation is that we don't want to be too long in this introductory paragraph. We think it is enough to refer to the examples instead. Keep in mind that we discuss this section also in the introduction\\
% S: I actually think that this list with the examples was helpful. It's not substantially longer than what remains after your modifications. I had the feeling that having a summary of all the scatered results + examples in one places would be helpful to the reader.}
\begin{itemize}
    \item[(i)] If $C_\circ^\prime$ has discontinuous zero points or non-negligible zero points, then $(\hei,\di,\Leb^3)$ does not satisfy $\MCP$ (Theorems \ref{thm:discontinuouszeropoint} and \ref{thm:manyzeropoints}). {This happens, for example, when $C_\circ^\prime(\phi)=1_{\mathsf{Z}^c}(\phi)$ near one of its zero points,
    where $\mathsf{Z}$ is a fat Cantor set.}
    
    \item[(ii)]If $C_\circ^\prime$ has infinite order at one of its zero points, then $(\hei,\di,\Leb^3)$ does not satisfy $\MCP$ (Theorem \ref{thm:veryflatzeropoint}). {This happens, for example, when $C_\circ^\prime(\phi)=e^{-1/|\phi|}$ near $\phi=0$.}
    {\item[(iii)] If $C_\circ^\prime$ is uniformly asymptotic to a fractional polynomial near its zero points, then $(\hei,\di,\Leb^3)$ satisfies $\MCP(0,N)$ for some $N \in (1, +\infty)$ (Theorem \ref{thm:MCPinstrongly} and \ref{thm:MCPinstronglysmallo}). This happens, for example, for the sub-Finsler Heisenberg group equipped with the $\ell^p$-norm with $p \in (1, 2)$.
    % \item[(iii')]In particular, if $C_\circ^\prime$ is positive, that is to say when $\normdot$ is $C^{1,1}$, then
    % $(\hei,\di,\Leb^3)$ satisfies $\MCP(0,N)$ for some $N \in (1 +\infty)$ (Corollary \ref{cor:C11andstrongly}).}
    % {\blue \item[(iii)] If $C_\circ^\prime$ is uniformly asymptotic to a fractional polynomial near its zero points, then $(\hei,\di,\Leb^3)$ satisfies $\MCP(0,N)$ for some $N \in (1, +\infty)$ (Theorem \ref{thm:MCPinstrongly} and \ref{thm:MCPinstronglysmallo}). 
    %This includes in particular the case when $\normdot$ is $C^{1,1}$ and strongly convex (Corollary \ref{cor:C11andstrongly}).
    }
\item[(iv)] If $C_\circ^\prime$ is not a fractional polynomial but is ``monotone'' near its zero points, then there are examples of sub-Finsler structures on the Heisenberg group that satisfy the $\MCP(0,N)$ for any $N \in (1, +\infty)$ 
% {\red. For instance, this happens when $C_\circ^\prime(\phi)=|\phi\log|\phi||$ near $\phi=0$} 
, see Example \ref{example:monotone}.

    \item[(v)] If $C_\circ^\prime$ ``oscillates with large variation'' near one of its zero points, then there are examples of sub-Finsler structures on the Heisenberg group that do not satisfy the $\mathsf{MCP}(0, N)$ for any $N \in (1, +\infty)$
    % {\red. For instance, this happens when $C_\circ^\prime$ oscillates between $|\phi|$ and $|\phi\log|\phi||$ near $\phi=0$ }
    , see Example \ref{example:oscilliation}.
\end{itemize}
% \todo[inline]{M,T: It is difficult to formulate a precise conjecture at this point. Since we have an in-depth discussion at the end of the section, we would remove this sentence\\
% S: To me removing that sentence is ok, although I liked having a rough conjecture}
% {\red We think that for a $C^1$ and strongly convex norm, the corresponding sub-Finsler Heisenberg group will satisfy the $\mathsf{MCP}(0, N)$ for some $N \in (1, +\infty)$ if and only if $C_\circ'$ is ``monotone'' and does not ``oscillate with large variation'' near its zero points. In the end of this section, we will reformulate this discussion in more formal terms.}
% \todo[inline]{M,T: We suggest to remove figure \ref{fig:C'o}}
% \begin{figure}
%     \centering
%     \includegraphics[width=0.8\textwidth]{Co.png}
%     \caption{Graph of $C_\circ^\prime(\phi)$ near a zero point}
%     \label{fig:C'o}
% \end{figure}

\subsection{Zero points of \texorpdfstring{$C_\circ^\prime$}{Ccirc} and non \texorpdfstring{$\MCP$}{MCP}}

Under the assumptions of this section, Proposition \ref{prop:regularityCo} ensures that the angle $C_\circ$ strictly increasing and Lipschitz continuous (thus differentiable almost everywhere). Let $\mathsf{Z}=\mathsf{D}_1\setminus\mathsf{D}_1^+\subset \R/2\pi_{\Omega^\circ}\mathbb{Z}$ be the set of zero points of $C_\circ^\prime$.
Under the assumption on the norm $\normdot$ to be $C^1$ and strongly convex,
we do not have information on the size of $\mathsf{Z}$ and the behavior of $C_\circ^\prime$ near its zero points. In this section, we investigate how the zero points in $\mathsf{Z}$ affect the validity (or failure) of the measure contraction property. We start with a lemma which gives a lower bound of the crucial ratio $\omega\partial_\omega\J_R/\J_R$ solely using $C_\circ$.

\begin{lemma}\label{prop:necessary4}
Let $\hei$ be the sub-Finsler Heisenberg group, equipped with a $C^1$ and strongly convex norm $\normdot$. For a fixed angle $\phi^\star\in\R/2\pi_{\Omega^\circ}\mathbb{Z}$,
we have 
\begin{equation}\label{ineq:sufficient3}
\limsup_{(\phi,\omega)\to (\phi^\star,0)}\frac{\omega\partial_\omega\J_R(\phi,\omega)}{\J_R(\phi,\omega)}\geq \limsup_{(\phi,\omega)\to (\phi^\star,0)}\frac{\omega C_\circ^{\prime}(\phi+\omega)}{C_\circ(\phi+\omega)-C_\circ(\phi)},
\end{equation}
where the $\limsup$ is taken over all $(\phi,\omega)\in \mathcal{U}$ such that $\phi+\omega\in \mathsf{D}_1$. In particular,
if there is $\phi^\star\in\R/2\pi_{\Omega^\circ}\mathbb{Z}$ such that
\[\limsup_{(\phi,\omega)\to(\phi^\star, 0)}\frac{\omega C_\circ^\prime(\phi+\omega)}{C_\circ(\phi+\omega)-C_\circ(\phi)}=+\infty,\]
    then the metric measure space $(\hei, \di, \Leb^3)$ does not satisfy the $\mathsf{MCP}(0, N)$ for any $N \in (1,\infty)$.
\end{lemma}

\begin{proof}
By using the function $\mathcal{K}(\phi,\omega)$ defined in \eqref{def:functionK},
we can write
\begin{equation}\label{eq:ratiofixedphi}
\frac{\omega\partial_{\omega}\J_R(\phi,\omega)}{\J_R(\phi,\omega)}=\frac{\omega C_\circ^\prime(\phi+\omega)}{C_\circ(\phi+\omega)-C_\circ(\phi)}\frac{(C_\circ(\phi+\omega)-C_\circ(\phi))\mathcal{K}(\phi,\omega)}{\J_R(\phi,\omega)}.
\end{equation}
We claim that
\begin{equation}\label{eq:greaterthan1}
\frac{(C_\circ(\phi+\omega)-C_\circ(\phi))\mathcal{K}(\phi,\omega)}{\J_R(\phi,\omega)}\geq 1,\end{equation}
for sufficiently small $\omega\in\mathsf{D}_1$. We prove \eqref{eq:greaterthan1} for $\omega>0$, as the proof for $\omega<0$ is completely analogous. At every differentiable point $\phi+\omega\in \mathsf{D}_1$,
we have
\begin{equation}\label{eq:tydolla}
\begin{split}
    \partial_\omega \Big{[} (C_\circ(\phi&+\omega)-C_\circ(\phi)) \mathcal{K}(\phi,\omega)-\J_R(\phi,\omega)\Big{]} =(C_\circ(\phi+\omega)-C_\circ(\phi))\partial_\omega \mathcal{K}(\phi,\omega)\\
    ={}&(C_\circ(\phi+\omega)-C_\circ(\phi))\Big{[}\sinom((\phi+\omega)_\circ)\big(\sinomp(\phi)+\omega\cosom(\phi_\circ)\big)\\
&\hspace{110pt}+\cosom((\phi+\omega)_\circ)\big(\cosomp(\phi)-\omega\sinom(\phi_\circ)\big)\\
    &\hspace{110pt}-\cosomp(\phi+\omega)\cosom(\phi_\circ)-\sinomp(\phi+\omega)\sinom(\phi_\circ)\Big{]}.
    \end{split}
\end{equation}
We are going to show that $\partial_\omega\mathcal{K}(\phi,\omega)$ is positive for $\omega>0$ sufficiently small.
By the fundamental theorem of calculus,
we have
\begin{align*}
    \partial_\omega\mathcal{K}(\phi,\omega)={}&\sinom((\phi+\omega)_\circ)\Big[\sinomp(\phi+\omega)-\int_0^\omega \Big(\cosom((\phi+t)_\circ)-\cosom(\phi_\circ)\Big)\diff t\Big]\\
    &+\cosom((\phi+\omega)_\circ)\Big[\cosomp(\phi+\omega)+\int_0^\omega \Big(\sinom((\phi+t)_\circ)-\sinom(\phi_\circ)\Big)\diff t\Big]\\
    &-\cosomp(\phi+\omega)\cosom(\phi_\circ)-\sinomp(\phi+\omega)\sinom(\phi_\circ)\\
    ={}&1-\cosomp(\phi+\omega)\cosom(\phi_\circ)-\sinomp(\phi+\omega)\sinom(\phi_\circ) + \int_0^\omega f(\phi, \omega, t)\diff t \\
    \geq{}& \int_0^\omega f(\phi, \omega, t)\diff t,
\end{align*}
where we have used the Pythagorean identity \eqref{eq:pytagorean} and where $f(t) := f(\phi, \omega, t)$ is given by
\[
f(t) := \cosom((\phi+\omega)_\circ)\big[\sinom((\phi+t)_\circ)-\sinom(\phi_\circ)\big]-\sinom((\phi+\omega)_\circ)\big[\cosom((\phi+t)_\circ)-\cosom(\phi_\circ)\big].
\]
The integral $\int_0^\omega f(t) \diff t$ is positive,
since again by \eqref{eq:pytagorean},
\[
\partial_tf(t)= C_\circ'(\phi + t) \underbrace{\Big[\sinom((\phi+\omega)_\circ)\sinomp(\phi+t)+\cosom((\phi+\omega)_\circ)\cosomp(\phi+t)\Big]}_{= 1 \text{ when } \omega = 0\text{ and }t=0},
\]
and thus by continuity $\partial_t f(t) \geq 0$
for all $\omega>0$ sufficiently small and all $t\in(0,\omega)$. Since $f(0)=0$, we deduce that $f(t)\geq 0$, and $\partial_\omega\mathcal{K}\geq 0$ for sufficiently small $\omega>0$. 
Then, since $C_\circ$ is non-increasing, the partial derivative in \eqref{eq:tydolla} is non-negative and therefore \eqref{eq:greaterthan1} holds. The second part of the statement follows from Propositions \ref{prop:MCPlimsup}.
\end{proof}

By using Lemma \ref{prop:necessary4},
we are able to give several necessary conditions for $\mathsf{MCP}$ to hold.

\begin{theorem}
\label{thm:discontinuouszeropoint}
Let $\hei$ be the sub-Finsler Heisenberg group, equipped with a $C^1$ and strongly convex norm $\normdot$, and with the Lebesgue measure $\Leb^3$. If $\mathsf{MCP}(0, N)$ holds for some $N\in (1,\infty)$, then the map $C_\circ' : \mathsf{D}_1 \to \R$ is continuous at every $\phi \in \mathsf{Z}$.

\end{theorem}

\begin{proof}
The claim is that if $\phi \in \mathsf{Z}$ and if $\mathsf{MCP}(0,N)$ holds, then $C_\circ'$ must be continuous at $\phi$. Assume by contradiction this is not the case, then there exists a sequence $\{\phi_n\}_{n \in \N} \subseteq \mathsf{D}_1$ such that $\phi_n \to \phi$ as $n \to +\infty$ and $\lim_{n \to + \infty} C_\circ'(\phi_n) = C > 0$. Now, setting $\omega_n := \phi_n - \phi$ and recalling that $\phi \in \mathsf{Z}$ by assumption, we compute that
\[
\frac{\omega_n C_\circ^{\prime}(\phi+\omega_n)}{C_\circ(\phi+\omega_n)-C_\circ(\phi)} = \frac{C_\circ^{\prime}(\phi+\omega_n)}{\frac{C_\circ(\phi+\omega_n)-C_\circ(\phi)}{\omega_n}} \to  +\infty,\qquad \text{ as } n\to+\infty.
\]
By Lemma \ref{prop:necessary4}, the metric measure space $(\hei,\di,\Leb^3)$ does not satisfy $\mathsf{MCP}(0,N)$ for any $N\in (1,\infty)$.
\end{proof}

\begin{theorem}\label{thm:manyzeropoints}
Let $\hei$ be the sub-Finsler Heisenberg group, equipped with a $C^1$ and strongly convex norm $\normdot$, and with the Lebesgue measure $\Leb^3$.
Assume that $\Leb^1(\mathsf{Z})>0$.
Then, the metric measure space $(\hei, \di, \Leb^3)$ does not satisfy the $\mathsf{MCP}(0, N)$ for any $N\in (1,\infty)$.
\end{theorem}

\begin{proof}
Let $1_{\mathsf{Z}^c}$ be the indicator function of the complementary subset $\mathsf{Z}^c\subset \R/2\pi_{\Omega^\circ}\mathbb{Z}$.
By the Lebesgue differentiation theorem and the assumption $\Leb^1(\mathsf{Z})>0$,
there is a Lebesgue point $\phi^\star\in \mathsf{Z}$ of $1_{\mathsf{Z}^c}$.
Denote by $\mathsf{D}_{\phi^\star}$ the set of points $\omega$ such that $\phi^\star+\omega\in\mathsf{D}_1$. Note that, since the norm $\normdot$ is strongly convex, $C_\circ$ is Lipschitz (cf. Proposition \ref{prop:regularityCo}) and therefore for all $\delta > 0$,
    \[
    \sup_{t\in(0,\delta]\cap\mathsf
{D}_{\phi^\star}}C_\circ^\prime(\phi^\star+t)<+\infty.
    \]
    Observe also that Lemma \ref{cor:diff_points_positive_measure} ensures that this supremum is strictly positive. Then, for all $\delta > 0$, there exists $\omega \in (0,\delta]\cap\mathsf{D}_{\phi^\star}$ such that
    \[
    C_\circ^\prime(\phi^\star+\omega) > \frac{1}{2} \sup_{t\in(0,\delta]\cap\mathsf{D}_{\phi^\star}}C_\circ^\prime(\phi^\star+t) \geq \frac{1}{2} \sup_{t\in(0,\omega]\cap\mathsf{D}_{\phi^\star}}C_\circ^\prime(\phi^\star+t),
    \]
    where the last inequality holds because $(0,\omega]\cap\mathsf{D}_{\phi^\star} \subseteq (0,\delta]\cap\mathsf{D}_{\phi^\star}$. In particular, there exists a sequence $\{\omega_n\}_{n \in \N} \subseteq \mathsf{D}_{\phi^\star}$ such that $\omega_n \to 0$ as $n \to +\infty$, and
    \begin{equation}
        \label{eq:omegan1}C_\circ^\prime(\phi^\star+\omega_n)>\frac{1}{2}\sup_{\omega\in(0,\omega_n]\cap \mathsf{D}_{\phi^\star}}C_\circ^{\prime}(\phi^\star+\omega), \qquad\forall\, n \in \N.
    \end{equation}
    We make also the following simple observation:
    \begin{align*}
        \frac{1}{\omega_n}\int_{0}^{\omega_n}C_\circ^{\prime}(\phi^\star+t)\diff t ={}& \frac{1}{\omega_n}\int_{0}^{\omega_n}C_\circ^{\prime}(\phi^\star+t) 1_{\mathsf{Z}^c}(\phi^\star+t) \diff t \\
        \leq{}& \Big(\frac{1}{\omega_n}\int_{0}^{\omega_n} 1_{\mathsf{Z}^c}(\phi^\star+t) \diff t \Big) \sup_{\omega\in(0,\omega_n]\cap \mathsf{D}_{\phi^\star}}C_\circ^{\prime}(\phi^\star+\omega)
    \end{align*}
    Consequently, we have the estimate
    \begin{equation}
    \label{eq:estimate_below}
    \begin{split}
        \frac{\omega_n C_\circ^\prime(\phi^\star+\omega_n)}{C_\circ(\phi^\star+\omega_n)-C_\circ(\phi^\star)}&=\frac{\omega_n C_\circ^\prime(\phi+\omega_n)}{\int_{0}^{\omega_n}C_\circ^{\prime}(\phi^\star+\omega)\diff \omega}\\
        &\geq \frac{ C_\circ^\prime(\phi^\star+\omega_n)}{\Big(\frac{1}{\omega_n}\int_{0}^{\omega_n} 1_{\mathsf{Z}^c}(\phi^\star+t) \diff t \Big) \sup_{\omega\in(0,\omega_n]\cap \mathsf{D}_{\phi^\star}}C_\circ^{\prime}(\phi^\star+\omega)}\\
        &>\frac{1}{\frac{2}{\omega_n}\int^{\omega_n}_{0}1_{\mathsf{Z}^c}(\phi^\star+t)\de t},
    \end{split}
    \end{equation}
    where we have used \eqref{eq:omegan1} for the last inequality. Since $\phi^\star\in\mathsf{Z}$ is a Lebesgue point of $1_{\mathsf{Z}^c}$,
    we have
    \[\lim_{n\to \infty}\frac{1}{\omega_n}\int^{\omega_n}_{0}1_{\mathsf{Z}^c}(\phi^\star+t)\de t=1_{\mathsf{Z}^c}(\phi^\star)=0.\]
    Thus, taking the limit as $n \to +\infty$ in \eqref{eq:estimate_below}, we obtain
    \[
    \lim_{n \to +\infty} \frac{\omega_n C_\circ^\prime(\phi^\star+\omega_n)}{C_\circ(\phi^\star+\omega_n)-C_\circ(\phi^\star)} = +\infty.
    \]
    This concludes the proof, as a consequence of Lemma \ref{prop:necessary4}.
\end{proof}

\noindent By Theorems \ref{thm:discontinuouszeropoint} and \ref{thm:manyzeropoints},
the following conditions can always be assumed when investigating a sufficient condition for the measure contraction property:
\begin{equation}\label{condition:ast}
\Leb(\mathsf{Z})=0~\text{ and }~C_\circ^\prime\text{ is continuous at every }\phi\in\mathsf{Z}\tag{$\ast$}    
\end{equation}

However, as shown in the following theorem,
the condition \eqref{condition:ast} does not give a sufficient condition for the measure contraction property. In particular, the $\mathsf{MCP}$ may fail due to a \emph{single} zero point of $C_\circ^\prime$.

\begin{theorem}\label{thm:veryflatzeropoint}
Let $\hei$ be the sub-Finsler Heisenberg group, equipped with a $C^1$ and strongly convex norm $\normdot$, 
% satisfying the condition \eqref{condition:ast}
and equipped with the Lebesgue measure $\Leb^3$. Assume that there is $\varphi^\star\in\mathsf{Z}$ such that for any $\alpha\geq 2$,
    \begin{equation}\label{eq:rapidconvergence}
        C_\circ^\prime(\varphi)=o(|\varphi-\varphi^\star|^{\alpha-2})\qquad \text{as }\varphi\to\varphi^\star \text{ in } \mathsf{D}_1.
    \end{equation}
    Then, the metric measure space $(\hei, \di, \Leb^3)$ does not satisfy the $\mathsf{MCP}(0, N)$ for any $N\in (1,\infty)$.

\end{theorem}

\begin{proof}
    We define the function $\beta:\R_{>0}\to \R\cup\{+\infty\}$ as follows: 
    \begin{equation}
        \beta(\omega):=
        \begin{cases}
                \frac{\log(C_\circ^{\prime}(\phi^\star+\omega))}{\log\omega} & \text{if }\omega\in {\sf D}_1\cap {\sf Z}^c \cap \R_{>0},\\
                +\infty & \text{otherwise.}
        \end{cases}
    \end{equation}
%     For $\omega>0$, let $\beta(\omega):=\frac{\log(C_\circ^{\prime}(\phi^\star+\omega))}{\log\omega}$,
%     with the convention that $\beta
% (\omega)=+\infty$ for $\phi^\star+\omega\in\mathsf{Z}\cup \mathsf{D}_1^c$.
    In other words, $\beta$ is such that $C_\circ^\prime(\phi^\star+\omega)=\omega^{\beta(\omega)}$, whenever is finite. In addition, since the norm is strongly convex and $\Leb^1(\mathsf{Z})=0$, $\beta$ is finite almost everywhere.
    The assumption \eqref{eq:rapidconvergence} implies that for any $\alpha\geq 2$,
    there is $\delta>0$ such that if $0<\omega<\delta$,
    \[C_\circ^\prime(\phi^\star+\omega)\leq \omega^{\alpha-2}.\]
    Therefore, we have $\beta(\omega)\geq \alpha-2$ for sufficiently small $\omega$. Since $\alpha\in[2,+\infty)$ is arbitrary,
    we deduce
    \begin{equation}\label{eq:betalower}
        \liminf_{\omega\to 0}\beta(\omega)=+\infty.
    \end{equation}
    We claim that there exists a sequence of positive numbers $\{\omega_n\}_{n\in \N}$ with $\omega_n\to0$ such that \begin{equation}\label{eq:choiceomegan}
        \beta(\omega_n)<\inf_{\omega\in(0,\omega_n]}\left\{\beta(\omega)+\left|\frac{\log 2}{\log\omega}\right| \right\}.
    \end{equation}
    Indeed, by \eqref{eq:betalower}, for any positive number $\psi_1\in \mathsf{D}_1$,
    there is $\delta_1=\delta_1(\psi_1)>0$ with the property that
    \begin{equation}\label{eq:truccone1}
        \beta(\psi)<\inf_{\omega\in(0,\psi_1)}\beta(\omega)+1\quad\Rightarrow\quad|\psi|>\delta_1.
    \end{equation}
    Without loss of generality,
    we can choose $\delta_1<2$.
    On the other hand,
    for any $\epsilon_1>0$,
    there is $\omega_1\in(0,\psi_1)$ such that
    \[\beta(\omega_1)<\inf_{\omega\in(0,\psi_1)}\beta(\omega)+\epsilon_1.\]
    Choose $\epsilon_1>0$ sufficiently small so that $\epsilon_1<\left|\frac{\log 2}{\log\delta_1}\right|$.
    Then, $\omega_1$ satisfies 
     \begin{multline}
    \beta(\omega_1)<\inf_{\omega\in(0,\psi_1)}\beta(\omega) +\epsilon_1<\inf_{\omega\in(0,\psi_1)}\beta(\omega) +\left|\frac{\log 2}{\log\delta_1}\right| \\
    \leq \inf_{\omega\in(0,\psi_1)}\left\{\beta(\omega)+\left|\frac{\log 2}{\log\omega}\right| \right\}
    \leq \inf_{\omega\in(0,\omega_1]}\left\{\beta(\omega)+\left|\frac{\log 2}{\log\omega}\right| \right\},
    \end{multline}
    where the second-to-last inequality follows from \eqref{eq:truccone1}.
    We can now repeat the same procedure starting from any $\psi_2<\omega_1$ to get a desired number $\omega_2$,
    and recursively get the desired sequence $\{\omega_n\}_{n\in \N}$ which satisfies \eqref{eq:choiceomegan}.
    
    Let $\beta_n:=\beta(\omega_n)$.
    By \eqref{eq:choiceomegan},
    for every $\omega\in(0,\omega_n)$,
    \[C_\circ^\prime(\phi^\star+\omega)\leq 2\omega^{\beta_n}\qquad\text{and}\qquad C_\circ^\prime(\phi^\star+\omega_n)=\omega_n^{\beta_n}.\]
    %  \begin{equation}
    %     \beta(\omega)=\beta(\omega)+|\frac{\log 2}{\log \omega}|-|\frac{\log 2}{\log \omega}|\geq  \inf_{(0,\omega_n)} \beta(\omega')+|\frac{\log 2}{\log \omega'}|-|\frac{\log 2}{\log \omega}|\geq \beta_n+\frac{\log 2}{\log \omega}
    % \end{equation}
    Then, as $n\to \infty$, we have:
    \begin{equation*}
        \frac{\omega_n C_\circ^\prime(\phi^\star+\omega_n)}{C_\circ(\phi^\star+\omega_n)-C_\circ(\phi^\star)}=\
        \frac{\omega_n^{\beta_n+1}}{\int_{0}^{\omega_n}C_\circ^\prime(\phi^\star+\omega)\diff \omega}\geq \ \frac{\omega_n^{\beta_n+1}}{2\int_{0}^{\omega_n}\omega^{\beta_n}\diff \omega}
        =\frac{\omega_n^{\beta_n+1}}{\frac{2}{\beta_n+1}\omega_n^{\beta_n+1}}=\frac{\beta_n+1}{2}\to+\infty.
    \end{equation*}
    By Lemma \ref{prop:necessary4}, the conclusion holds.
\end{proof}

\begin{remark}
    The condition \eqref{eq:rapidconvergence} in Theorem \ref{thm:veryflatzeropoint} implies that the norm $\normdot$ is not $C^{1,\alpha}$ for any $\alpha\in(0,1]$.
    Indeed,
    by \cite[Prop.\ 1.6]{MR1079061},
    the norm $\normdot$ is not $C^{1,\alpha}$ if and only if there is a sequence $\{(\phi_n,\omega_n)\}_{n\in \N}$ such that
    \begin{equation}\label{eq:coHolder}\frac{C_\circ(\phi_n+\omega_n)-C_\circ(\phi_n)}{\omega_n^{1/\alpha}}\to 0\qquad \text{as }n\to\infty.\end{equation}
    However, the converse is not necessarily true: if the norm $\normdot$ not being $C^{1, \alpha}$ for any $\alpha \in (0, 1]$ does not automatically imply the condition \eqref{eq:rapidconvergence}.
     \end{remark}

\subsection{Sufficient condition by using \texorpdfstring{$C_\circ^{\prime}$}{Ccircprime}}

In this section, we shall give a sufficient condition for the measure contraction property.
We start from rewriting the reduced Jacobian function $\mathcal{J}_R$ in terms the function $C_\circ$. 
Recall that the differential equality $\sin_{\Omega^\circ}''(x) = -C_\circ'(x) \sin_{\Omega^\circ}(x)$ holds for $x \in \mathsf{D}_1$.
Since $\mathsf{D}_1$ is a full measure subset,
we are allowed to use the integration by parts to write
\begin{equation}
\label{eq:firstsinexpansion}
    \begin{split}
    \sin_{\Omega^\circ}(x) ={}& \sin_{\Omega^\circ}(y) + \int_{y}^x \cos_{\Omega} (t_\circ) \de t  \\
    ={}& \sin_{\Omega^\circ}(y) + \cos_{\Omega}(y_\circ) (x - y) - \int_{y}^x \sin_{\Omega^\circ} (t) C_\circ'(t) (x - t) \de t, 
 \end{split}
\end{equation}
which is the Taylor's formula with integral remainder of order 2. Note that the integrand of the last term of \eqref{eq:firstsinexpansion} contains $\sin_{\Omega^\circ}(t)$, and it can be replaced recursively by the same expression \eqref{eq:firstsinexpansion}, i.e.
\begin{equation}
    \begin{split}
    \sin_{\Omega^\circ}(x) ={}& \sin_{\Omega^\circ}(y) + \cos_{\Omega}(y_\circ) (x - y) \\
    &- \int_{y}^x \left(\sin_{\Omega^\circ}(y) + \cos_{\Omega}(y_\circ) (t - y) - \int_{y}^t \sin_{\Omega^\circ} (s) C_\circ'(s) (t - s) \diff t\right) C_\circ'(t) (x - t) \diff t \\
    ={}& \sin_{\Omega^\circ}(y) + \cos_{\Omega}(y_\circ) (x - y) - \sin_{\Omega^\circ}(y) \int_{y}^x C_\circ'(t) (x - t) \diff t \\
    &- \cos_{\Omega}(y_\circ) \int_{y}^x C_\circ'(t) (x - t) (t - y) \diff t + \int_{y}^x \int_{y}^t \sin_{\Omega^\circ} (s)  C_\circ'(t)C_\circ'(s) (x - t)(t - s) \diff s \diff t.
 \end{split}
\end{equation}
We repeat this process a last time, that is to say, we replace the term $\sin_{\Omega^\circ}(s)$ in the last integral term above by its expression \eqref{eq:firstsinexpansion}:
\begin{equation}
\label{eq:sinexpansion}
    \begin{split}
    \sin_{\Omega^\circ}(x) ={}& \sin_{\Omega^\circ}(y) + \cos_{\Omega}(y_\circ) (x - y) - \sin_{\Omega^\circ}(y) \int_{y}^x C_\circ'(t) (x - t) \diff t  \\
    & - \cos_{\Omega}(y_\circ) \int_{y}^x C_\circ'(t) (x - t) (t - y) \diff t + \sin_{\Omega^\circ}(y) \int_{y}^x \int_{y}^t C_\circ'(s) C_\circ'(t) (x - t) (t - s) \diff s \diff t \\
    & + \cos_{\Omega}(y_\circ) \int_{y}^x \int_{y}^t C_\circ'(s) C_\circ'(t) (x - t) (t - s) (s - y) \diff s \diff t \\
    & - \int_y^x \int_y^t \int_y^s \sin_{\Omega^\circ}(u) C_\circ'(t) C_\circ'(s) C_\circ'(u) (x - t) (t - s) (s - u) \diff u \diff s \diff t. 
 \end{split}
\end{equation}
The same reasoning can be done for $\cos_{\Omega^\circ}(x)$, starting from $\cos_{\Omega^\circ}''(x) = -C_\circ'(x) \cos_{\Omega^\circ}(x)$ with similar computations and we get
\begin{equation}
\label{eq:cosexpansion}
    \begin{split}
    \cos_{\Omega^\circ}(x)
    ={}& \cos_{\Omega^\circ}(y) - \sin_{\Omega}(y_\circ) (x - y) - \cos_{\Omega^\circ}(y) \int_{y}^x C_\circ'(t) (x - t) \diff t \\
    &+ \sin_{\Omega}(y_\circ) \int_{y}^x C_\circ'(t) (x - t) (t - y) \diff t + \cos_{\Omega^\circ}(y) \int_{y}^x \int_{y}^t C_\circ'(s)  C_\circ'(t) (x - t) (t - s) \diff s \diff t \\
    &- \sin_{\Omega}(y_\circ) \int_{y}^x \int_{y}^t C_\circ'(s)  C_\circ'(t) (x - t) (t - s) (s - y) \diff s \diff t \\
    & - \int_y^x \int_y^t \int_y^s \cos_{\Omega^\circ}(u) C_\circ'(t) C_\circ'(s) C_\circ'(u) (x - t) (t - s) (s - u) \diff u \diff s \diff t. 
 \end{split}
\end{equation}
% where $R(x, y)$ is a remainder term satisfying \eqref{eq:remainder}. 
Since $\cos_{\Omega}(x_\circ) = \frac{\diff}{\diff x} \sin_{\Omega^\circ}(x)$ and $\sin_{\Omega}(x_\circ) = - \frac{\diff}{\diff x} \cos_{\Omega^\circ}(x)$, we also obtain
\begin{equation}\label{eq:coscexpansion}
    \begin{split}
    \cos_{\Omega}(x_\circ) ={}& \cos_{\Omega}(y_\circ) - \sin_{\Omega^\circ}(y) \int_{y}^x C_\circ'(t) \diff t  - \cos_{\Omega}(y_\circ) \int_{y}^x C_\circ'(t) (t - y) \diff t \\
    & + \sin_{\Omega^\circ}(y) \int_{y}^x \int_{y}^t C_\circ'(s) C_\circ'(t) (t - s) \diff s \diff t \\
    & + \cos_{\Omega}(y_\circ) \int_{y}^x \int_{y}^t C_\circ'(s) C_\circ'(t) (t - s) (s - y) \diff s \diff t \\
    & - \int_y^x \int_y^t \int_y^s \sin_{\Omega^\circ}(u) C_\circ'(t) C_\circ'(s) C_\circ'(u) (t - s) (s - u) \diff u \diff s \diff t, 
 \end{split}
\end{equation}
and
\begin{equation}\label{eq:sincexpansion}
    \begin{split}
    \sin_{\Omega}(x_\circ)={}& \sin_{\Omega}(y_\circ) + \cos_{\Omega^\circ}(y) \int_{y}^x C_\circ'(t) \diff t - \sin_{\Omega}(y_\circ) \int_{y}^x C_\circ'(t) (t - y) \diff t \\
    &- \cos_{\Omega^\circ}(y) \int_{y}^x \int_{y}^t C_\circ'(s)  C_\circ'(t) (t - s) \diff s \diff t \\
    &+ \sin_{\Omega}(y_\circ) \int_{y}^x \int_{y}^t C_\circ'(s)  C_\circ'(t) (t - s) (s - y) \diff s \diff t \\
    &+ \int_y^x \int_y^t \int_y^s \cos_{\Omega^\circ}(u) C_\circ'(t) C_\circ'(s) C_\circ'(u) (t - s) (s - u) \diff u \diff s \diff t.
 \end{split}
\end{equation}
The combination of these formulas leads to the following result.

\begin{prop}
    \label{prop:formulaJRdJRint}
    Let $\hei$ be the sub-Finsler Heisenberg group, equipped with a $C^1$ and strongly convex norm $\normdot$. Then, its reduced Jacobian $\J_R$ can be expressed in the following way. For all $(\phi,\omega)\in\mathcal{U}$, it holds
\begin{equation}\label{eq:JRaround0}
\J_R (\phi, \omega) = \frac{1}{2} \int_{\phi}^{\phi + \omega} \left( \int_{\phi}^{\phi + \omega} (t - s)^2 C_\circ'(t) C_\circ'(s) \diff s \right) \diff t + R(\phi, \omega),
\end{equation}
 and, if in addition $\varphi+\omega\in {\sf D}_1$, it holds: 
\begin{equation}\label{eq:JRderivaround0}
\partial_\omega \J_R (\phi, \omega) = C'_\circ(\phi + \omega) \int_{\phi}^{\phi + \omega} (\phi + \omega - t)^{2} C_\circ'(t) \diff t + \partial_{\omega} R(\phi, \omega),
\end{equation}
where the remainder term $R(\phi, \omega)$ is given by
\begin{equation}
    \label{eq:remainderterm}
    \begin{aligned}
    R(\phi, \omega) =&\int_{\phi}^{\phi + \omega} \int_\phi^t \int_\phi^s (t - s)(s - u) \Big[ \sin_{\Omega^\circ}(u) \cos_{\Omega^\circ}(\phi) - \sin_{\Omega^\circ}(\phi) \cos_{\Omega^\circ}(u) \\
    & - (t - \phi) (\cos_{\Omega}(\phi_\circ)\cos_{\Omega^\circ}(u) + \sin_{\Omega}(\phi_\circ)\sin_{\Omega^\circ}(u)) \Big] C'_\circ(t) C'_\circ(s) C'_\circ(u) \diff u \diff s \diff t,
\end{aligned}
\end{equation}
and its derivative by
\begin{equation}
    \label{eq:dremainderterm}
    \begin{aligned}
    \partial_\omega R(\phi, \omega) ={} C'_\circ(\phi + \omega) & \int_\phi^{\phi + \omega} \int_\phi^s (\phi + \omega - s)(s - u) \Big[ \sin_{\Omega^\circ}(u) \cos_{\Omega^\circ}(\phi) - \sin_{\Omega^\circ}(\phi) \cos_{\Omega^\circ}(u) \\
    & - \omega (\cos_{\Omega}(\phi_\circ)\cos_{\Omega^\circ}(u) + \sin_{\Omega}(\phi_\circ)\sin_{\Omega^\circ}(u)) \Big] C'_\circ(s) C'_\circ(u) \diff u \diff s. 
\end{aligned}
\end{equation}
\end{prop}

\begin{remark}
    This formula is a synthetic generalisation of the expansion of the Jacobian that appeared in \cite[Lem.\ 29]{borzatashiro} and \cite[Eq.\ (30)]{borzatashiro} for the $\ell^p$-Heisenberg group.
\end{remark}

\begin{proof}[Proof of Proposition \ref{prop:formulaJRdJRint}]
    The equality \eqref{eq:JRaround0} is shown by substituting \eqref{eq:sinexpansion}, \eqref{eq:cosexpansion}, \eqref{eq:coscexpansion}, and \eqref{eq:sincexpansion}, with $y = \phi$ and $x = \phi + \omega$, into the expression of the reduced Jacobian \eqref{eq:reduced_Jac}. 
    The term $\sinomp(\phi +  \omega)$, for example, is replaced using \eqref{eq:sinexpansion} by
\begin{align}
    \sin_{\Omega^\circ}&(\phi + \omega) = \sin_{\Omega^\circ}(\phi) + \omega \cos_{\Omega}(\phi_\circ) \label{eq:0orderterm}\\
    &- \sin_{\Omega^\circ}(\phi) \int_{\phi}^{\phi + \omega} C_\circ'(t) (\phi + \omega - t) \diff t  - \cos_{\Omega}(\phi_\circ) \int_{\phi}^{\phi + \omega} C_\circ'(t) (\phi + \omega - t) (t - \phi) \diff t \label{eq:1orderterm1}\\
    &+ \sin_{\Omega^\circ}(\phi) \int_{\phi}^{\phi + \omega} \int_{\phi}^t C_\circ'(s) C_\circ'(t) (\phi + \omega - t) (t - s) \diff s \diff t \label{eq:2orderterm1}\\
    & + \cos_{\Omega}(\phi_\circ) \int_{\phi}^{\phi + \omega} \int_{\phi}^t C_\circ'(s) C_\circ'(t) (\phi + \omega - t) (t - s) (s - \phi) \diff s \diff t \label{eq:2orderterm2}\\
    & - \int_\phi^{\phi + \omega} \int_\phi^t \int_\phi^s \sin_{\Omega^\circ}(r) C_\circ'(t) C_\circ'(s) C_\circ'(r) (\phi + \omega - t) (t - s) (s - r) \diff r \diff s \diff t. \label{eq:higherorderterms}
\end{align}
This is also done for the terms $\cosomp( \phi +  \omega)$, $\sinom(( \phi +  \omega)_\circ)$, and $\cosom( ( \phi +  \omega)_\circ)$ in \eqref{eq:reduced_Jac}. After a computation, one can find that all the terms of ``order zero'' (those with no integral and no $C'_\circ$ term) such as \eqref{eq:0orderterm} vanish. All the terms of ``order one'' (those with one integral and one $C'_\circ$ term) such as \eqref{eq:1orderterm1} cancel out too. The terms of ``second order'' (those with a double integral and two $C'_\circ$ terms) such as \eqref{eq:2orderterm1}--\eqref{eq:2orderterm2}, however, do not cancel and they simplify to the first term in \eqref{eq:JRaround0}. The remainder \eqref{eq:remainderterm} is obtained by gathering all the terms of ``higher order'' (those with three integrals and three $C'_\circ$ terms) such as \eqref{eq:higherorderterms}. Finally, differentiating with respect to $\omega$ yields \eqref{eq:JRderivaround0} and \eqref{eq:dremainderterm}.
\end{proof}

\begin{remark}
    Since the generalised trigonometric functions are bounded and $C_\circ$ is Lipschitz, there is a constant $C > 0$ such that for all $(\phi, \omega) \in \mathcal{U}$,
    we have
\begin{equation}
	\label{eq:boundremainder}
	\begin{aligned}
		|R(\phi, \omega)| \leq{}& C \int_{\phi}^{\phi + \omega} \int_\phi^t \int_\phi^s (t - s)(s - u) C'_\circ(t) C'_\circ(s) C'_\circ(u) \diff u \diff s \diff t \\
		\leq{}& C \int_{\phi}^{\phi + \omega} \int_\phi^{\phi + \omega} \int_\phi^{\phi + \omega} (t - u)^2 C'_\circ(t) C'_\circ(s) C'_\circ(u) \diff u \diff s \diff t \\
		\leq{}& 2C \|C_\circ'\|_{\infty} \ \omega \left(\frac{1}{2} \int_\phi^{\phi + \omega} \int_\phi^{\phi + \omega} (t - s)^2 C'_\circ(t) C'_\circ(s) \diff s \diff t\right)
	\end{aligned}
\end{equation}
    and, similarly, when defined, 
    \begin{equation}
        \label{eq:boundDremainder}
        |\partial_\omega R(\phi, \omega)| \leq 2C \|C_\circ'\|_{\infty} \ \omega \left( C'_\circ(\phi + \omega) \int_{\phi}^{\phi + \omega} (\phi + \omega - t)^{2} C_\circ'(t) \diff t \right)
    \end{equation}
    with an appropriate constant $C>0$.
\end{remark}

To simplify the notation in the proof of the next theorem, we introduce the following functions. For $\alpha \geq 2$, we define
\begin{equation}\label{def:Pphiomega}
\begin{aligned}
    P_\alpha(\phi, \omega) :={} & \frac{1}{2} \int_{\phi}^{\phi + \omega} \left( \int_{\phi}^{\phi + \omega} (t - s)^2 |t|^{\alpha - 2} |s|^{\alpha - 2} \diff s \right) \diff t \\
    ={}& \frac{1}{\alpha^2 (\alpha^2 - 1)} \big(\abs{\omega + \phi}^{2 \alpha} + \abs{\phi}^{2 \alpha}) + \frac{2}{\alpha^2} \abs{\omega + \phi}^\alpha \abs{\phi}^\alpha \\
    &\qquad\qquad\qquad\qquad\qquad - \frac{1}{\alpha^2 - 1} \left(\abs{\omega + \phi}^{\alpha}\abs{\phi}^{\alpha - 2} + |\omega+\phi|^{\alpha-2}|\phi|^\alpha \right)(\omega + \phi)\phi,
\end{aligned}
\end{equation}
and $P_\alpha(s) := P_\alpha(1, s)$. Note that
\begin{align*}
    \partial_\omega P_\alpha(\phi, \omega) ={}& \frac{2}{\alpha (\alpha^2 - 1)} \abs{\omega + \phi}^{2 (\alpha - 1)} (w + \phi) + \frac{2}{\alpha} \abs{\omega + \phi}^{\alpha-2} (\omega + \phi) \abs{\phi}^\alpha \\
    & \qquad\qquad\qquad - \frac{1}{\alpha^2 - 1} \left((\alpha+1)|\omega+\phi|^2+(\alpha-1)|\phi|^2\right)|\omega+\phi|^{\alpha-2}|\phi|^{\alpha-2} \phi,
    % ={}& \abs{\omega + \phi}^{\alpha - 2} \Big[ \frac{2}{\alpha (\alpha^2 - 1)} \abs{\omega + \phi}^{\alpha} (w + \phi) + \frac{2}{\alpha} (\omega + \phi) \abs{\phi}^\alpha - \frac{1}{\alpha^2 - 1} \left((\alpha+1)|\omega+\phi|^2+(\alpha-1)|\phi|^2\right)|\phi|^{\alpha-2} \phi \Big]
\end{align*}
and that $\partial_s P_\alpha(s) = \partial_\omega P_\alpha(1, s)$. More explicitly, the functions $P_\alpha(s)$ and $\partial_s P_\alpha(s)$ are given by
\begin{equation}
    \label{eq:P(s)}
    P_\alpha(s) := \frac{1}{\alpha^2 (\alpha^2 - 1)} \Big[\abs{1+s}^{2\alpha} - \alpha^2 \abs{1 + s}^{\alpha}(1+s) + 2 (\alpha^2 - 1) \abs{1+  s}^\alpha - \alpha^2 \abs{1 + s}^{\alpha - 2}(1 + s) + 1 \Big],
\end{equation}
and
\begin{equation}
    \label{eq:dP(s)}
    \partial_s P_\alpha(s) = \frac{2}{\alpha (\alpha^2 - 1)} \abs{1 + s}^{\alpha - 2} \Big[\abs{1 + s}^\alpha (1 + s) - 1 - (\alpha + 1) s -\frac{\alpha (\alpha + 1)}{2} s^2 \Big].
\end{equation}

The function $P_\alpha(s)$ vanishes if and only if $s = 0$. Indeed, if $1 + s < 0$, then it is easy to see that all the terms of $P_\alpha(s)$ are strictly positive. If $1 + s > 0$, the derivative of $P_\alpha(s)$ is
\[
\partial_s P_\alpha(s) = 2 \alpha (1 + s)^{\alpha - 2} \Big[(1 + s)^{\alpha + 1} - 1 - (\alpha + 1) s -\frac{\alpha (\alpha + 1)}{2} s^2 \Big].
\]
Consider the function $f(s):=(1 + s)^{\alpha + 1} - 1 - (\alpha + 1) s -\frac{\alpha (\alpha + 1)}{2} s^2$. Then $f''(s)=\alpha(\alpha+1)\big[|1+s|^{\alpha-1}-1\big]$, thus $f'(s)$ is convex on $(-1,+\infty)$ with a unique minimum $f'(0)=0$ and therefore $f$ is increasing on $(-1,+\infty)$.
Since $f(-1)=\frac{\alpha(1-\alpha)}{2}<0$ and $f(0)=0$, $\partial_sP_\alpha(s)$ is increasing with a unique zero point $\partial_sP_\alpha(0)=0$.
Therefore, $P_\alpha(s)$ attains a unique minimum at $0$. Since $P_\alpha(0)=0$ we conclude that $s=0$ is the unique zero point of $P_\alpha$.

    The following limits are easily obtained:
    \begin{equation}
    \label{eq:discussiondP/P}
        \lim_{s\to 0}\frac{s\partial_sP_\alpha(s)}{P_\alpha(s)}=4, \ \lim_{s\to+\infty}\frac{s\partial_sP_\alpha(s)}{P_\alpha(s)}=\lim_{s\to-\infty}\frac{s\partial_sP_\alpha(s)}{P_\alpha(s)}=2\alpha.
    \end{equation}
    The map $s \mapsto 1 + \frac{s \partial_s P_\alpha(s)}{P_\alpha(s)}$ is thus continuous for all $s \in \R$ and bounded.

 \begin{theorem}\label{thm:MCPinstrongly}
    Let $\hei$ be the sub-Finsler Heisenberg group, equipped with a $C^1$ and strongly convex norm $\normdot$.
    % {\red with the condition \eqref{condition:ast}.}
    % \todo[inline]{Sam: Do we use that condition \eqref{condition:ast} anywhere in the proof? I don't think so}
    Assume that for all $\phi^\star \in\R/2 \pi_{\Omega_\circ}\mathbb{Z}$, there exists $\alpha(\phi^\star) \in \rinterval{2}{+\infty}$ and $A(\phi^\star),B(\phi^\star)>0$ such that
    \begin{equation}
        \label{eq:fractionalcondition}
        A(\phi^\star)|\phi-\phi^\star|^{\alpha(\phi^\star)-2}\leq C_\circ'(\phi) \leq B(\phi^\star)|\phi - \phi^\star|^{\alpha(\phi^\star) - 2}\qquad\text{for a.e. }\phi \text{ near }\phi^\star.
    \end{equation}
    Assume furthermore that
    \begin{equation}\label{eq:uniformABalpha}
         \sup_{\phi^\star \in \R/2 \pi_{\Omega_\circ}\mathbb{Z}}\frac{B(\phi^\star)}{A(\phi^\star)},\ \sup_{\phi^\star \in \R/2 \pi_{\Omega_\circ}\mathbb{Z}}\alpha(\phi^\star)<+\infty.
    \end{equation}
    Then, the metric measure space $(\hei, \di, \Leb^3)$ satisfies the $\mathsf{MCP}(0, N)$ condition for some $N\in (1,\infty)$. 
\end{theorem}

\begin{proof}
Fix $\phi^\star\in\R/2\pi_{\Omega^\circ}\mathbb{Z}$.
 If $(\phi, \omega) \in \mathcal{U}$ is close enough to $(\phi^\star,0)$,
 then we have, by \eqref{eq:JRaround0} and the assumption, that
\begin{align*}
    \mathcal{J}_R(\phi, \omega) ={}& \frac{1}{2} \int_{\phi}^{\phi + \omega} \left( \int_{\phi}^{\phi + \omega} (t - s)^2 C_\circ'(t) C_\circ'(s) \diff s \right) \diff t + R(\phi, \omega) \\
    \geq{}& \frac{A(\phi^\star)^2}{2} \int_{\phi}^{\phi + \omega} \left( \int_{\phi}^{\phi + \omega} (t - s)^2 |t - \phi^\star|^{\alpha(\phi^\star) - 2} |s - \phi^\star|^{\alpha(\phi^\star) - 2} \diff s \right) \diff t + R(\phi, \omega) \\
    ={}& A(\phi^\star)^2 P_{\alpha(\phi^\star)}(\phi - \phi^\star, \omega) + R(\phi, \omega)\\
    \geq {}&\Big[A(\phi^\star)^2 - \bar{C} B(\phi^\star)^2\omega\Big]P_{\alpha(\phi^\star)}(\phi - \phi^\star, \omega),
\end{align*}
where $P_{\alpha(\phi^\star)}(\phi, \omega)$ is the function defined in \eqref{def:Pphiomega}
and a lower bound on the remainder term $R(\phi, \omega)$ can be obtained from \eqref{eq:boundremainder} and the lower bound of \eqref{eq:fractionalcondition}, with $\bar{C}:=2C\|C_\circ^\prime\|_\infty$.
Similarly, we can deduce an upper bound for $\partial_\omega \J_R$  from \eqref{eq:JRderivaround0} and the upper bound of \eqref{eq:fractionalcondition}: 
\begin{align*}
    \partial_\omega \J_R (\phi, \omega) ={}& C'_\circ(\phi + \omega) \int_{\phi}^{\phi + \omega} (\phi + \omega - t)^{2} C_\circ'(t) \diff t + \partial_\omega R(\phi, \omega)\\
    \leq{}& B(\phi^\star)^2 \partial_\omega P_{\alpha(\phi^\star)}(\phi - \phi^\star, \omega) + \partial_\omega R(\phi, \omega)\\
    \leq{}&\Big[B(\phi^\star)^2 +\bar{C} B(\phi^\star)^2\omega\Big]\partial_\omega P_{\alpha(\phi^\star)}(\phi - \phi^\star, \omega),
\end{align*}
where we also have a bound for the remainder term $\partial_\omega R(\phi, \omega)$ using \eqref{eq:boundDremainder}.
The goal now is to show that $\limsup_{(\phi, \omega) \to (\phi^\star,0)} N(\phi, \omega)$ has a uniform upper bound independent of $\phi^\star$, since this would prove that the $\mathsf{MCP}(0, N)$ condition is satisfied for some $N\in (1,\infty)$, according to Proposition \ref{prop:MCPlimsup}. A bound for $\limsup_{(\phi, \omega) \to (\phi^\star,0)} N(\phi, \omega)$ is found following the blueprint of the proof of \cite[Thm.\ 39]{borzatashiro}. 

Consider a converging sequence $\{(\phi_n, \omega_n)\}_{n \in \N} \subset \mathcal{U}$ such that $\phi_n+\omega_n\in\mathsf{D}_1^{+}$ and $(\phi_n, \omega_n) \to (\phi^\star, 0)$ as $n \to +\infty$. Assume also, without loss of generality, that that the ratio $s_n := \omega_n/(\phi_n - \phi^\star)$ converges to $s \in \interval{-\infty}{+\infty}$ as $n \to +\infty$.  Since $\omega_n \to 0$,
we obtain a positive lower bound
\begin{align*}
\mathcal{J}_R(\phi_n, \omega_n)\geq {}&\Big[A(\phi^\star)^2 - \bar{C} B(\phi^\star)^2\omega_n\Big]P_{\alpha(\phi^\star)}(\phi_n - \phi^\star, \omega_n)\\
={}&\Big[A(\phi^\star)^2-\bar{C}B(\phi^\star)^2\omega_n\Big]|\phi_n - \phi^\star|^{2 \alpha(\phi^\star)} P_{\alpha(\phi^\star)}(s_n).
\end{align*}
We estimate $\omega_n \mathcal{J}_R(\phi_n, \omega_n)$ is a similar fashion:
\begin{align*}
\omega_n \partial_\omega \mathcal{J}_R(\phi_n, \omega_n )\leq{}&\Big[B(\phi^\star)^2 +\bar{C} B(\phi^\star)^2\omega_n\Big]\partial_\omega P_{\alpha(\phi^\star)}(\phi_n - \phi^\star, \omega_n)\\
={}&  \Big[ B(\phi^\star)^2+ \bar{C}B(\phi^\star)^2 \omega_n\Big]|\phi_n - \phi^\star|^{2 \alpha(\phi^\star)} s_n \partial_s P_{\alpha(\phi^\star)}(s_n).
\end{align*}
With these bounds in hand, we can write
\begin{align*}
    \limsup_{n\to\infty}N(\phi_n, \omega_n) \leq{}& \lim_{n\to\infty}1 + \frac{B(\phi^\star)^2+\bar{C}B(\phi^\star)^2\omega_n}{A(\phi^\star)^2-\bar{C}B(\phi^\star)^2\omega_n}\cdot \frac{|\phi_n - \phi^\star|^{2 \alpha(\phi^\star)}}{|\phi_n - \phi^\star|^{2 \alpha(\phi^\star)}} \cdot \frac{s_n \partial_s P_{\alpha(\phi^\star)}(s_n)}{P_{\alpha(\phi^\star)}(s_n)}\\
    ={}&1+\frac{B(\phi^\star)^2}{A(\phi^\star)^2}\frac{s\partial_sP_{\alpha(\phi^\star)}(s)}{P_{\alpha(\phi^\star)}(s)}\\
    \leq{}& 1+\sup_{\phi^\star\in\R/2\pi_{\Omega^\circ}\mathbb{Z}}\frac{B(\phi^\star)^2}{A(\phi^\star)^2}\,\sup_{s\in[-\infty,+\infty]}\,\sup_{\phi^\star\in\R/2\pi_{\Omega^\circ}\mathbb{Z}}\frac{s\partial_sP_{\alpha(\phi^\star)}(s)}{P_{\alpha(\phi^\star)}(s)}.
\end{align*}
Taking into account the first assumption in \eqref{eq:uniformABalpha}, it is enough to show that the second supremum is finite. Since the function 
\begin{equation}
    (\alpha,s) \mapsto \frac{s\partial_sP_{\alpha}(s)}{P_{\alpha}(s)}
\end{equation}
is continuous and bounded as $s \to \pm \infty$ (cf. \eqref{eq:discussiondP/P}) and thus uniformly bounded on $[2, M] \times [-\infty,+\infty]$ for every $M\in (2,+\infty)$, the second assumption in \eqref{eq:uniformABalpha} allows to conclude.

\end{proof}

The next theorem shows that, we can get a lower bound on the curvature exponent, cf. Definition \ref{def:curvatureexponent}, by assuming $C_\circ'$ is differentiable in a fractional way. 

% {\blue 
% \begin{definition}[Curvature exponent] \label{def:finedelparto}
%     Let $\hei$ be the \sF Heisenberg group, equipped with a norm $\normdot$, and with the Lebesgue measure $\Leb^3$. Assume that $(\hei,\di,\Leb^3)$ satisfies $\MCP(0,N)$ for some $N\in (1,+\infty)$. The \emph{curvature exponent} of $\hei$ is defined as 
%     \begin{equation}
%         N_{\mathrm{curv}}:=\inf\{N\in  (1,+\infty):(\hei,\di,\Leb^3) \text{ is } \MCP(0,N)\}.
%     \end{equation}
% \end{definition}
% }
% \todo[inline]{SK: The definition of curvature exponent appears here a little out of nowhere. Suddenly we are back to metric measure spaces, and only for three lines. One option would be to put this definition in Section 3? But let's discuss about that.\\
% M,T: Proposal to put the def into context a bit more}

\begin{theorem}\label{thm:MCPinstronglysmallo}
    Let $\hei$ be the sub-Finsler Heisenberg group, equipped with a $C^1$ and strongly convex norm $\normdot$. Assume that for all $\phi^\star \in \R/2 \pi_{\Omega_\circ}\mathbb{Z}$, there exists $\alpha(\phi^\star) \in \rinterval{2}{+\infty}$ and $A(\phi^\star) > 0$ such that
    \begin{equation}
        \label{eq:smallocondition}
        C_\circ'(\phi) = A(\phi^\star) |\phi - \phi^\star|^{\alpha(\phi^\star) - 2} + o(|\phi - \phi^\star|^{\alpha(\phi^\star) - 2}), \ \text{ as } \phi \to \phi^\star \text{ in } \mathsf{D}_1.
    \end{equation}
    Assume furthermore that $q:=\sup\{\alpha(\phi^\star) \ | \ \phi^\star \in \R/2 \pi_{\Omega_\circ}\mathbb{Z}\} < +\infty$. Then, the metric measure space $(\hei, \di, \Leb^3)$ satisfies the $\mathsf{MCP}(0, N)$ condition for some $N\in (1,+\infty)$. In addition, it holds that
    \begin{equation}\label{eq:curexpbound}
    N_{\mathrm{curv}} \geq 2 q + 1.
    \end{equation}     
\end{theorem}

\begin{proof}
     Observe that the assumptions of Theorem \ref{thm:MCPinstrongly} are verified, therefore $(\hei, \di, \Leb^3)$ satisfies the $\mathsf{MCP}(0, N)$ condition for some $N\in (1,+\infty)$. In the reminder of the proof we find an estimate for its curvature exponent $N_{\mathrm{curv}}$.
    The argument follows the same lines as the proof of Theorem \ref{thm:MCPinstrongly}, but with finer asymptotics. Let $\phi^\star \in \R/2 \pi_{2 \Omega^\circ}\mathbb{Z}$ and, for simplicity, we write $A = A(\phi^\star)$. We start by showing that
    \[
    \mathcal{J}_R(\phi,  \omega) = \underbrace{\frac{A^2}{2} \int_{\phi}^{\phi + \omega} \left( \int_{\phi}^{\phi + \omega} (t - s)^2 |t - \phi^\star|^{\alpha(\phi^\star) - 2} |s - \phi^\star|^{\alpha(\phi^\star) - 2} \diff s \right) \diff t}_{= A^2 P_{\alpha(\phi^\star)}(\phi - \phi^\star, \omega)} + o(P_{\alpha(\phi^\star)}(\phi - \phi^\star, \omega)),
    \]
    as $(\phi, \omega) \to (\phi^\star, 0)$. In order to prove that, consider the leading term of \eqref{eq:JRaround0}, and observe that
    \begin{align}
        \frac{1}{2} &\int_{\phi}^{\phi + \omega}  \int_{\phi}^{\phi + \omega} (t - s)^2 C_\circ'(t) C_\circ'(s) \diff s  \diff t \\
        &= \frac{A^2}{2} \int_{\phi}^{\phi + \omega} \int_{\phi}^{\phi + \omega} (t - s)^2 |t - \phi^\star|^{\alpha(\phi^\star) - 2} |s - \phi^\star|^{\alpha(\phi^\star) - 2} \diff s \diff t \nonumber \\
        &\quad  + \frac{1}{2} \int_{\phi}^{\phi + \omega} \int_{\phi}^{\phi + \omega} (t - s)^2 A |t - \phi^\star|^{\alpha(\phi^\star) - 2} (C_\circ'(s) - A |s - \phi^\star|^{\alpha(\phi^\star) - 2}) \diff s  \diff t \label{eq:smallopart1}\\
        &\quad + \frac{1}{2} \int_{\phi}^{\phi + \omega} \int_{\phi}^{\phi + \omega} (t - s)^2 (C_\circ'(t) - A |t - \phi^\star|^{\alpha(\phi^\star) - 2}) A |s - \phi^\star|^{\alpha(\phi^\star) - 2} \diff s \diff t \label{eq:smallopart2}\\
        &\quad + \frac{1}{2} \int_{\phi}^{\phi + \omega} \int_{\phi}^{\phi + \omega} (t - s)^2 (C_\circ'(t) - A |t - \phi^\star|^{\alpha - 2}) (C_\circ'(s) - A |s - \phi^\star|^{\alpha(\phi^\star) - 2}) \diff s \diff t. \label{eq:smallopart3}
    \end{align}
    The term \eqref{eq:smallopart1} is $o(P_{\alpha(\phi^\star)}(\phi - \phi^\star, \omega))$ as $(\phi, \omega) \to (\phi^\star, 0)$. Indeed,
    by \eqref{eq:smallocondition},
    for all $\epsilon > 0$,
    we have that
    \begin{align*}
        \Bigg|\frac{1}{2} \int_{\phi}^{\phi + \omega}& \int_{\phi}^{\phi + \omega} (t - s)^2 (C_\circ'(t) - A |t - \phi^\star|^{\alpha(\phi^\star) - 2}) (C_\circ'(s) - A |s - \phi^\star|^{\alpha(\phi^\star) - 2}) \diff s \diff t\Bigg| \\
        & \leq \epsilon^2 \frac{A^2}{2} \int_{\phi}^{\phi + \omega} \int_{\phi}^{\phi + \omega} (t - s)^2 |t - \phi^\star|^{\alpha(\phi^\star) - 2} |s - \phi^\star|^{\alpha(\phi^\star) - 2} \diff s \diff t,
    \end{align*}
    for all $(\phi,\omega)\in\mathcal{U}$ sufficiently close to $(\phi^\star,0)$ so that
    \[
    \Big|C_\circ'(t) - A |t - \phi^\star|^{\alpha(\phi^\star) - 2}\Big| \leq \epsilon |t - \phi^\star|^{\alpha(\phi^\star) - 2}, \text{ for all } t \in \interval{\phi}{\phi + \omega}.
    \]
    Analogous computations show that the terms \eqref{eq:smallopart2}, \eqref{eq:smallopart3} are $o(P_{\alpha(\phi^\star)}(\phi - \phi^\star, \omega))$ as $(\phi, \omega) \to (\phi^\star, 0)$. The same thing can be done for $\partial_\omega \J_R(\phi, \omega)$ with the help of \eqref{eq:JRderivaround0} and \eqref{eq:boundDremainder}. More precisely, under the condition \eqref{eq:smallocondition}, we have that
    \begin{equation*}
        \partial_\omega \J_R (\phi, \omega) = \underbrace{A^2 |\phi + \omega - \phi^\star|^{\alpha(\phi^\star) - 2} \int_{\phi}^{\phi + \omega} (\phi + \omega - t)^{2} |t - \phi^\star|^{\alpha(\phi^\star) - 2} \diff t}_{= A^2 \partial_\omega P_{\alpha(\phi^\star)}(\phi - \phi^\star, \omega)} + o(\partial_\omega P_{\alpha(\phi^\star)}(\phi - \phi^\star, \omega)),
    \end{equation*}
    as $(\phi, \omega) \to (\phi^\star, 0)$. 

     Now, given any $s \in [-\infty,+\infty]$, consider a sequence $\left\{(\phi_n, \omega_n)\right\}_{n \in \N} \subset \mathcal{U}$ such that $\phi_n,\phi_n+\omega_n\in \mathsf{D}_1$, $(\phi_n, \omega_n) \to (\phi^\star, 0)$ as $n \to +\infty$ and the ratio $s_n := \omega_n/(\phi_n - \phi^\star)$ converges to $s \in [-\infty,+\infty]$ as $n \to +\infty$.
     Then, we obtain that, as $n\to \infty$,
    \begin{equation}
        \label{eq:upperboundNsmallo}
        \begin{aligned}
        N(\phi_n, \omega_n) ={}& 1 + \frac{A^2 \omega_n \partial_\omega P_{\alpha(\phi^\star)}(\phi_n - \phi^\star, \omega_n) +  o(\omega_n\partial_\omega P_{\alpha(\phi^\star)}(\phi_n - \phi^\star, \omega_n))}{A^2 P_{\alpha(\phi^\star)}(\phi_n - \phi^\star, \omega_n) + o(P_{\alpha(\phi^\star)}(\phi_n - \phi^\star, \omega_n))} \\
        ={}& 1 + \frac{|\phi_n - \phi^\star|^{2 \alpha(\phi^\star)} \Big[A^2 s_n \partial_\omega P_{\alpha(\phi^\star)}(s_n) + s_n o(\partial_s P_{\alpha(\phi^\star)}(s_n))\Big]}{|\phi_n - \phi^\star|^{2 \alpha(\phi^\star)} \Big[ A^2 P_{\alpha(\phi^\star)}(s_n) + o(P_{\alpha(\phi^\star)}(s_n)) \Big]}\\
        \to{}&1+\frac{s\partial_sP_{\alpha(\phi^\star)}(s)}{P_{\alpha(\phi^\star)}(s)},
    \end{aligned}
    \end{equation}
    where $P_{\alpha(\phi^\star)}(s)$ is the function defined in \eqref{eq:P(s)}. %{\red In the same way of Theorem \ref{thm:MCPinstrongly},
    %we can deduce that $\mathsf{MCP}(0, N)$ is satisfied for some $N\in (1,\infty)$.
    %Moreover,
    %note that \eqref{eq:upperboundNsmallo} {\red also} implies that} 
    Keeping in mind \eqref{eq:upperboundNsmallo} and since $s\in [-\infty,+\infty]$ is arbitrary, we finally deduce that 
    \[\sup_{\phi^\star\in\R/2\pi_{\Omega^\circ}\mathbb{Z}}\limsup_{(\phi,\omega)\to(\phi^\star,0)}N(\phi,\omega)=\sup_{\phi^\star\in\R/2\pi_{\Omega^\circ}\mathbb{Z}}\sup_{s\in[-\infty,+\infty]}1+ \frac{s\partial_sP_{\alpha(\phi^\star)}(s)}{P_{\alpha(\phi^\star)}(s)}\geq 2q+1,\]
    where the last inequality is a consequence of the estimate that $\lim_{s\to\pm\infty}\frac{s\partial_sP_q(s)}{P_q(s)}=2q$ cf. \eqref{eq:discussiondP/P}. The conclusion follows from Corollary \ref{prop:diffcharacMCP}.%, as $N_{\mathrm{curv}}\geq N(\varphi,\omega)$, for all $(\varphi,\omega) \in U$ with $\varphi +\omega \in {\sf D}_1$.}
\end{proof}

\begin{remark}
    \begin{itemize}
       \item[(i)] If $q>2$ and there is $\phi^\star$ such that $\alpha(\phi^\star)=q$,
    then one can prove that the curvature exponent satisfies $N_{\mathrm{curv}} > 2q+1$. Indeed, it is possible to observe that, in this case, the map \begin{equation}
        s \mapsto 1 + \frac{s \partial_s P_q(s)}{P_q(s)}
    \end{equation} 
    converges from above to its limits at $- \infty$.
    %{\red $\pm \infty$}
    \item[(ii)] For $p \in \ointerval{1}{2}$, consider the $\ell^p$-Heisenberg group, that is to say the sub-Finsler Heisenberg group equipped with an $\ell^p$-norm. The $\ell^p$-norm is $C^1$ and strongly convex, and its angle correspondence $C_\circ$ satisfies
    \[
    C_\circ'(\phi) =
    \begin{cases}
        (q - 1) |\phi - \phi^\star|^{q - 2} + o(|\phi - \phi^\star|^{q - 2}) & \text{if } \phi^\star \in\frac{\pi_q}{2}\mathbb{Z}\\ 
   (q-1)|\sin_{\ell^q}(\phi^\star)\cos_{\ell^q}(\phi^\star)|^{q-2} + o(1) & \text{if } \phi^\star \notin\frac{\pi_q}{2}\mathbb{Z}
   \end{cases}, \text{ as } \phi \to \phi^\star.
    \]
    Therefore, Theorem \ref{thm:MCPinstronglysmallo}, together with item (i), recovers the estimate in \cite{borzatashiro}.
    \end{itemize}
\end{remark}

As a consequence of the previous results, we also obtain the following important corollary.

\begin{corollary}\label{cor:C11andstrongly}
    Let $\hei$ be the sub-Finsler Heisenberg group, equipped with a $C^{1,1}$ and strongly convex norm $\normdot$. Then, the metric measure space $(\hei, \di, \Leb^3)$ satisfies $\mathsf{MCP}(0, N)$ for some $N\in (1,\infty)$. Furthermore, if the norm $\normdot$ is $C^2$, then $N_{\mathrm{curv}}\geq 5$.
\end{corollary}

% \todo[inline]{M: I think we have in mind different things about this estimates with $5$, we have to discuss the last part of this subsection accordingly}

\begin{proof}
    As the norm $\| \cdot \|$ is $C^{1, 1}$ and strongly convex, then $C^\circ$ and $C_\circ$ are Lipschitz continuous, cf. Proposition \ref{prop:regularityCo}. In particular, for every $\varphi\in {\sf D}_1$, $1/A\leq C'_\circ(\phi) \leq A$ for some constant $A > 0$ since $C_\circ \circ C^\circ = \mathrm{Id}$. Therefore, the hypothesis of Theorem \ref{thm:MCPinstrongly} are satisfied with $\alpha(\phi^\star) = 2$ for every $\phi^\star \in \R/2 \pi_{\Omega_\circ}\mathbb{Z}$, and the first claim follows.

    If the norm is $C^2$,
    then the (inverse) angle correspondence map $C^\circ$ is of class $C^1$.
    By the inverse function theorem, $C_\circ$ is also of class $C^1$,
    and $C_\circ^\prime$ is continuous.
    Then, Theorem \ref{thm:MCPinstronglysmallo}
 holds with $q=2$. This implies the last part of the claim.
\end{proof}

\begin{remark}
    \begin{itemize}
        \item[(i)] For the estimate $N\geq 5$,
    we only need the continuity of $C_\circ^\prime$ at one non-zero point. However, if the norm is only $C^{1,1}$, the existence of such a point is not guaranteed.
    \item[(ii)]If the norm is $C^2$,
    then what we have proven is that
    \[\limsup_{(\phi,\omega)\to(\phi^\star,0)}N(\phi,\omega)=5\]
    for every $\phi^\star\in\R/2\pi_{\Omega^\circ}$ (and thus $N_{\mathrm{curv}}\geq 5$).
    To prove $\MCP(0,5)$ (so that $N_{\mathrm{curv}}=5$),
    we would need to show the inequality $N(\phi,\omega)\leq 5$ for $\omega\neq 0$.
    \end{itemize}
\end{remark}

The following example shows that there is a $C^2$ and strongly convex norm such that the associated sub-Finsler Heisenberg group has $N_{\mathrm{curv}}>5$.
\begin{example}\label{ex:greaterthan5}

    For $K,L>0$ and $\epsilon\in(0,K/L)$,
let $f$ be the function defined on $(-\epsilon,\epsilon)$ by
\[
f(\phi)=K+L\phi.
\]
Since $f$ is bounded, positive and smooth, there is a $C^2$ and strongly convex norm on $\R^2$ such that its angle correspondence satisfies $C_\circ(\phi)=\int_0^\phi f(t)\diff t$ in a sufficiently small neighbourhood of $\phi = 0$,
and such that $\cosom(0_\circ)=\cosomp(0)=1,$ $\sinom(0_\circ)=\sinomp(0)=0$ (for the construction of such a norm, see Remark \ref{remark:construction}).
For $\omega\in(-\epsilon,\epsilon)$,
we have that
\begin{align*}
    \frac{1}{2} \int_0^\omega \left( \int_0^\omega (t - s)^2 C_\circ'(t) C_\circ'(s) \diff s \right) \diff t ={}& \frac{1}{2} \int_0^\omega \int_0^\omega (t - s)^2 (K+Ls)(K+Lt)\diff s \diff t\\
    ={}& \frac{\omega^4}{12}\left(K^2+KL\omega+\frac{L^2}{6}\omega^2\right).
\end{align*}
By \eqref{eq:remainderterm} and the assumption on the initial values,
the remainder term $R(0,\omega)$ satisfies
\begin{align*}
    R(0, \omega) =&\int_{0}^{\omega} \int_0^t \int_0^s (t - s)(s - u) (u  - t +O(u^2)) C'_\circ(t) C'_\circ(s) C'_\circ(u) \diff u \diff s \diff t=o(\omega^6)~~\text{ as } \omega\to 0.
\end{align*}
Therefore, by \eqref{eq:JRaround0}, we have that
\[\J_R(0,\omega)=\frac{\omega^4}{12}\left(K^2+KL\omega+\frac{L^2}{6}\omega^2\right)+o(\omega^6)\qquad\text{ as }\omega\to 0.\]
Since $\omega \in \mathsf{D}_1$ by construction, we also have that
\begin{align*}
    \omega C'_\circ(\omega) \int_0^\omega (\omega - t)^{2} C_\circ'(t) \diff t ={}& \omega(K+L\omega) \int_0^\omega (\omega - t)^{2}(K+Lt) \diff t \\
    ={}& \frac{1}{3}\omega^4\left(K^2+\frac{5KL}{4}\omega+\frac{L^2\omega^2}{4}\right).
\end{align*}
In addition, by \eqref{eq:dremainderterm}, it holds that $\omega\partial_\omega R(0, \omega) = o(\omega^6)$, as $\omega \to 0$. Thus, recalling \eqref{eq:JRderivaround0}, we have:
\[
\omega \partial_\omega \J_R(0, \omega) = \frac{1}{3}\omega^4\left(K^2+\frac{5KL}{4}\omega+\frac{L^2\omega^2}{4}\right) + o(\omega^6).
\]
Consequently, for sufficiently small $\omega>0$,
we obtain that
\[
4\J_R(0,\omega)-\omega\partial_\omega\J_R(0,\omega)=\frac{\omega^4}{3}\left(-\frac{KL\omega}{4}-\frac{L^2\omega^2}{12}\right)+o(\omega^6)<0.
\]
This shows that $1+\omega\partial_\omega\J_R(0,\omega)/\J_R(0,\omega)>5$ and $N_{\mathrm{curv}}>5$ by Proposition \ref{prop:diffcharacMCP}.

\end{example}

\begin{remark}\label{remark:construction}
    For a given integrable, bounded and almost everywhere positive function $f$, we construct a norm $\normdot$ whose angle correspondence $C_\circ$ satisfies the following conditions:
\begin{equation}\label{condition:norm}
    \begin{cases}
        \text{$C_\circ$ is Lipschitz continuous and strictly increasing,}\\
        \text{$C_\circ^\prime(\varphi)=f(\varphi)$ near $\varphi=0$.}
    \end{cases}
\end{equation}
Recall that the generalized trigonometric functions satisfy the differential equations
\[\begin{cases}
        \frac{\diff}{\diff t}\cosomp(t)=-\sinom(t_\circ),\\
        \frac{\diff}{\diff t}\sinom(t_\circ)=C_\circ^{\prime}(t)\cosomp(t),
    \end{cases}~~~~\begin{cases}
        \frac{\diff}{\diff t}\cosom(t_\circ)=-C_\circ^\prime(t)\sinomp(t),\\
        \frac{\diff}{\diff t}\sinomp(t)=\cosom(t_\circ),
    \end{cases}\]
    provided that the norm $\normdot$ is $C^1$ and strongly convex.
    Conversely,
 the following differential equations locally recovers a generalized trigonometric function:
    \begin{equation}\label{Cosystem}\begin{cases}
        \dot{x}(t)=-Y(t),\\
        \dot{Y}(t)=f(t)x(t),
    \end{cases}~~~~\begin{cases}
        \dot{X}(t)=-f(t)y(t),\\
        \dot{y}(t)=X(t).
    \end{cases}\end{equation}
    Indeed, by Carath\'eodory's Theorem,
    there is an absolutely continuous solution to the differential equation \eqref{Cosystem} for a given initial value $(X_0,Y_0,x_0,y_0)$ at time $t=0$.

    Let $(X,Y,x,y)$ be a solution to the initial value $(1,0,1,0)$.
    Note that, since $(X,Y)$ is absolutely continuous and its differential is bounded,
    the curve $(x,y)$ is $C^{1,1}$.
    Then, we have obtained $C^{1,1}$ curves $\gamma_{\pm}:(-\epsilon,\epsilon)\to \R^2$,
    $\gamma_+(t):=(x(t),y(t))$ and $\gamma_-(t):=(-x(t),-y(t))$.
    Without loss of generality,
    we can assume that $\gamma_{\pm}$ is twice differentiable at endpoints.
    Join $\gamma_+(\epsilon)$ and $\gamma_-(-\epsilon)$ by an Euclidean arc smoothly.
    The central inversion of this arc joins $\gamma_+(-\epsilon)$ and $\gamma_-(\epsilon)$.
    This procedure yields a centrally symmetric, strictly convex and $C^{1,1}$ domain $B$,
    and define $\normdot_\ast$ as the norm whose unit ball is $B$. Let $\normdot$ be the dual norm of $\normdot_\ast$.
    From the construction,
    the angle correspondence $C_\circ$ of the norm $\normdot$ satisfies the condition \eqref{condition:norm}.
    
\end{remark}

\subsection{Towards a necessary and sufficient condition}\label{section:oscillates}
In this section, we provide examples of norms that satisfy and do not satisfy the measure contraction property. We begin by presenting an example of a norm that satisfies the $\MCP$ condition, even though it is not included by the assumptions of Theorem \ref{thm:MCPinstrongly}.

\begin{example}\label{example:monotone}

Let $f$ be defined by $f(\phi)=|\phi\log|\phi||$ in a neighbourhood of $\phi=0$.
Since $f$ is bounded, strictly positive and locally integrable, following the construction laid out in Remark \ref{remark:construction},  there is a $C^{1}$ and strongly convex norm on $\R^2$ such that its angle correspondence satisfies $C_\circ(\phi)= \int_0^\phi f(t)\diff t$ in a sufficiently small neighbourhood of $\phi = 0$. Note that $C'_\circ$ is everywhere positive except at $\phi = 0$ or $\phi = \pi_{\Omega^\circ}$. We shall see that the \sF Heisenberg group associated with this norm satisfies the measure contraction property.

For sufficiently small and positive $\phi$ and $\omega$, we have that
\begin{align*}
    \frac{1}{2}\int_\phi^{\phi+\omega}&\int_\phi^{\phi+\omega}(s-t)^2C_\circ^\prime(s)C_\circ^\prime(t)\diff s\diff t = \frac{1}{2}\int_\phi^{\phi+\omega}\int_\phi^{\phi+\omega}(s-t)^2 |s \log |s||\cdot |t \log|t|| \diff s\diff t \\
    ={}& \Big[ \tfrac{1}{16} x |x|^3 (1 - 4 \log |x|) \Big]_{\phi}^{\phi + \omega} \times \Big[ \tfrac{1}{4} x |x| (1 - 2 \log |x|) \Big]_{\phi}^{\phi + \omega} - \Big( \Big[ \tfrac{1}{9} |x|^3 (1 - 3 \log |x|) \Big]_{\phi}^{\phi + \omega} \Big)^2 ,
\end{align*}
as well as
\begin{align*}
    C'_\circ(\phi + \omega) \int_{\phi}^{\phi + \omega}& (\phi + \omega - t)^{2} C_\circ'(t) \diff t = |(\phi + \omega) \log|\phi + \omega|| \int_{\phi}^{\phi + \omega} (\phi + \omega - t)^{2} |t \log |t|| \diff t \\
    ={}& |(\phi + \omega) \log|\phi + \omega|| \times \Big( (\phi + \omega)^2 \Big[ \tfrac{1}{4} x |x| (1 - 2 \log |x|) \Big]_{\phi}^{\phi + \omega} \\
    & \qquad \qquad - 2 (\phi + \omega) \Big[ \tfrac{1}{9} |x|^3 (1 - 3 \log |x|) \Big]_{\phi}^{\phi + \omega} + \Big[ \tfrac{1}{16} x |x|^3 (1 - 4 \log |x|) \Big]_{\phi}^{\phi + \omega} \Big).
\end{align*}
In the computations above, we have use the notation $[F(x)]_a^b := F(b) - F(a)$ as well as
\[
\frac{\diff}{\diff x} \Big( \frac{1}{4} x |x|(1 - 2 \log|x|)\Big) = |x \log|x||, \qquad \frac{\diff}{\diff x} \Big( \frac{1}{9} |x|^3(1 - 3 \log|x|)\Big) = x |x \log|x||,
\]
and
\[
\frac{\diff}{\diff x} \Big( \frac{1}{16} x |x|^3 (1 - 4 \log|x|)\Big) = x^2 |x \log|x||, \qquad \text{ for all } x \in (-1, 1)\setminus\{0\}.
\]
Since the only zero points of $C_\circ^\prime$ are $\phi=0$ and $\phi = \pi_{\Omega^\circ}$ and $\limsup_{(\phi, \omega) \to (\phi^\star,0)} N(\phi, \omega) = 5$ for the other values of $\phi^\star$, proving that $\limsup_{(\phi, \omega) \to (0,0)} N(\phi, \omega) < +\infty$ implies that the corresponding sub-Finsler Heisenberg group satisfies $\mathsf{MCP}$, according to Proposition \ref{prop:MCPlimsup} (it follows by symmetry that $\limsup_{(\phi, \omega) \to (\pi_{\Omega^\circ},0)} N(\phi, \omega) < +\infty$).

Consider a sequence $\{(\phi_n, \omega_n)\}_{n\in \N} \subset \mathcal{U}$ of positive numbers such that $\phi_n,\phi_n+\omega_n\in \mathsf{D}_1$ and $(\phi_n, \omega_n) \to (0, 0)$ as $n \to +\infty$. Assume also, without loss of generality, that the ratio $s_n := \omega_n/\phi_n$ converges to $s \in [-\infty,+\infty]$ as $n \to +\infty$. If $s \neq -1, \pm \infty$, then 
\begin{equation}\label{eq:logratio}
    \lim_{n \to +\infty} \frac{\log|\phi_n +\omega_n|}{\log|\phi_n|}=\lim_{n\to\infty}\frac{\log\left|1+\frac{\omega_n}{\phi_n}\right|}{\log|\phi_n|}+1 = 1.
\end{equation}
In this case, we find that
 \begin{align*}
    \lim_{n \to \infty}& \frac{\frac{1}{2}\int_{\phi_n}^{\phi_n+\omega_n}\int_{\phi_n}^{\phi_n+\omega_n} (s-t)^2C_\circ^\prime(s)C_\circ^\prime(t)\diff s\diff t}{\phi_n^6   (\log|\phi_n|)^2} \\
    &  {} =\lim_{n\to\infty}\frac{1}{16}\left[\left(1+\frac{\omega_n}{\phi_n}\right)\left|1+\frac{\omega_n}{\phi_n}\right|^3\left(\frac{1}{\log|\phi_n|}-\frac{4\log|\phi_n+\omega_n|}{\log|\phi_n|}\right)-\left(\frac{1}{\log|\phi_n|}-4\right)\right]\\
    &\qquad \qquad\times\frac{1}{4}\left[\left(1+\frac{\omega_n}{\phi_n}\right)\left|1+\frac{\omega_n}{\phi_n}\right|\left(\frac{1}{\log|\phi_n|}-\frac{2\log|\phi_n+\omega_n|}{\log|\phi_n|}\right)-\left(\frac{1}{\log|\phi_n|}-2\right)\right]\\
    &\qquad\qquad -\frac{1}{81}\left[\left|1+\frac{\omega_n}{\phi_n}\right|^3\left(\frac{1}{\log|\phi_n|}-\frac{3\log|\phi_n+\omega_n|}{\log|\phi_n|}\right)-\left(\frac{1}{\log|\phi_n|}-3\right)\right]^2\\
    &{} = \tfrac{1}{8} \Big[(1 + s) |1 + s|^3 - 1 \Big] \times \Big[(1 + s) |1 + s| - 1 \Big] - \tfrac{1}{9} \Big[ |1 + s|^3 - 1 \Big]^2 =: P(s),
\end{align*}
and similarly
\begin{align*}
    \lim_{n \to \infty} & \frac{\omega_n C'_\circ(\phi_n + \omega_n) \int_{\phi_n}^{\phi_n + \omega_n} (\phi_n + \omega_n - t)^{2} C_\circ'(t) \diff t}{\phi_n^6 (\log|\phi_n|)^2} \\
    & \qquad \qquad = s |1 + s| \times \Big[\tfrac{1}{12} (1+s) |1+s|^3 -\tfrac{1}{2}(1 + s)^2 + \tfrac{2}{3} (1 + s) - \tfrac{1}{4}\Big] = s \partial_s P(s).
\end{align*}
Note that the function $P(s)$ attains $0$ only at $s=0$, the map $s \mapsto s \partial_s P(s)/P(s)$ is continuous, and the following limits can be easily established:
\[
\lim_{s\to 0}\frac{s\partial_sP(s)}{P(s)}=4,\qquad\lim_{s\to -1}\frac{s\partial_sP(s)}{P(s)}=0,\qquad\lim_{s\to+\infty}\frac{s\partial_sP(s)}{P(s)}=\lim_{s\to -\infty}\frac{s\partial_sP(s)}{P(s)}=6.
\]
The remainder terms $R(\phi_n, \omega_n)$ and $\partial_\omega R(\phi_n, \omega_n)$ (cf.\ \eqref{eq:JRaround0} and \eqref{eq:JRderivaround0}) can be dealt with in the same way as in the proofs of Theorems \ref{thm:MCPinstrongly} and \ref{thm:MCPinstronglysmallo} and they are of higher orders. Therefore, we have that
\[
\lim_{n \to +\infty} \frac{\omega_n\partial_\omega\J_R(\phi_n,\omega_n)}{\J_R(\phi_n,\omega_n)}=\frac{s\partial_s P(s)}{P(s)} \leq \sup_{s\in\R\setminus\{-1\}}\frac{s\partial_s P(s)}{P(s)}<+\infty.
\]
Assume now that $\omega_n/\phi_n \to -1$ as $n \to +\infty$. In that case, in a similar way to \eqref{eq:logratio}, $|1+\frac{\omega_n}{\phi_n}|^{a}\log|\phi_n+\omega_n|/\log|\phi_n|\to 0$ for all $a > 0$ and thus
\[
\lim_{n \to +\infty} \frac{\omega_n\partial_\omega\J_R(\phi_n,\omega_n)}{\J_R(\phi_n,\omega_n)} =\lim_{s\to -1}\frac{s\partial_sP(s)}{P(s)}=0.
\]
Similarly, if $\omega_n/\phi_n \to \pm \infty$, then for any $a > 0$, it holds that $|1+\frac{\omega_n}{\phi_n}|^{a}\log|\phi_n+\omega_n|/\log|\phi_n| \to+\infty$ which implies that
\[
\lim_{n \to +\infty} \frac{\omega_n\partial_\omega\J_R(\phi_n,\omega_n)}{\J_R(\phi_n,\omega_n)} =\lim_{s\to\pm\infty}\frac{s\partial_sP(s)}{P(s)}=6.
\]
In any case, the limit superior of the map $(\phi, \omega) \mapsto \omega\partial_\omega\J_R(\phi,\omega)/\J_R(\phi,\omega)$ as $(\phi,\omega) \to (0,0)$ remains bounded. Therefore, the Heisenberg group $(\hei,\di,\Leb^3)$ satisfies $\mathsf{MCP}(0,N)$ for some $N \in (1, +\infty)$ according to Proposition \ref{prop:MCPlimsup}.
\end{example}

In Example \ref{example:monotone}, we have seen that if the angle correspondence satisfies $C_\circ^\prime\sim |\phi\log|\phi||$ near its zero points, then it is possible to verify the measure contraction property. This means that the sufficient condition of Theorem \ref{thm:MCPinstrongly} is not a necessary condition. Indeed, the upper bound of $C_\circ'$ in \eqref{eq:fractionalcondition} is satisfied with $\alpha \leq 3$ while the lower bound is only satisfied with $\alpha > 3$. In some sense, this indicates that the principal term of $C_\circ^\prime$ %{\red does not have to be a fractional polynomial}
need not have fractional order at any zero point for the $\mathsf{MCP}$ to hold.

Next, we present an example of sub-Finsler structure on the Heisenberg group for which the the differential of the angle correspondence $C_\circ'$ oscillates between $|\phi|$ and $|\phi \log|\phi||$. This example is going to tell us that if the differential of the angle correspondence oscillates, then the measure contraction property may fail.

\begin{example}\label{example:oscilliation}
%{\red We are going to construct an example of {\red $C^{1,1}$}{\blue $C^1$} and strongly convex norm such that the map $C_\circ'$ oscillates between $\phi$ and $\phi \log(\phi)$, and as a result, the $\mathsf{MCP}$ condition will not hold.}
Let $\{a_n\}_{n \in \N}$, $\{b_n\}_{n \in \N}$ and $\{c_n\}_{n \in \N}$ be three sequences of positive real numbers converging to zero such that $a_n < b_n < c_n$, $c_{n + 1} = a_n$, and $a_n/b_n \to 0$ as $n \to +\infty$. In particular, $b_n/(b_n - a_n) \to 1$ and $a_n/(b_n - a_n) \to 0$ as $n \to +\infty$. Let $f$ be the function defined almost everywhere by
\[
f(\phi)=\begin{cases}
    |\phi| & \text{ if } |\phi| \in (a_n, b_n)\\
    |\phi \log |\phi|| & \text{ if } |\phi| \in (b_n, c_n).
\end{cases}
\]
Since $f$ is bounded, positive and locally integrable, there is a $C^1$ and strongly convex norm on $\R^2$ such that its angle correspondence satisfies $C_\circ(\phi)= \int_0^\phi f(t)\diff t$ in a sufficiently small neighbourhood of $\phi = 0$ as in Remark \ref{remark:construction}. Set $\phi_n := a_n$ and $\omega_n := b_n + \epsilon_n-a_n$, where $\epsilon_n > 0$ is chosen so that $b_n + \epsilon_n \in (b_n, c_n)$ and $\epsilon_n \leq (b_n - a_n)^3$. We have that
\begin{align}
    \frac{1}{2} \int_{\phi_n}^{\phi_n + \omega_n}& \int_{\phi_n}^{\phi_n + \omega_n} (t - s)^2 C_\circ'(t) C_\circ'(s) \diff s \diff t = \frac{1}{2} \int_{a_n}^{b_n} \int_{a_n}^{b_n} (t - s)^2 s t \diff s \diff t \nonumber \\
    &- \frac{1}{2} \int_{a_n}^{b_n} \int_{b_n}^{b_n + \epsilon_n} (t - s)^2 t s \log s \diff s \diff t - \frac{1}{2} \int_{b_n}^{b_n + \epsilon_n} \int_{a_n}^{b_n} (t - s)^2 s t \log t \diff s \diff t \label{eq:examplelogterm23} \\
    &+ \frac{1}{2} \int_{b_n}^{b_n + \epsilon_n} \int_{b_n}^{b_n + \epsilon_n} (t - s)^2 t s \log t \log s \diff s \diff t \label{eq:examplelogterm4} \\
    ={}& \frac{1}{72} (b_n - a_n)^6 + o((b_n - a_n)^6) =\frac{1}{72}b_n^6+o(b_n^6), \text{   as } n \to +\infty. \nonumber
\end{align}
The asymptotic above is justified since
\begin{align*}
    \frac{1}{(b_n - a_n)^6}\Big|\int_{a_n}^{b_n} \int_{b_n}^{b_n + \epsilon_n} (t - s)^2 t s \log s \diff s \diff t\Big| \leq \frac{(b_n - a_n) \epsilon_n (\epsilon_n + b_n - a_n)^2 b_n (b_n + \epsilon_n)|\log b_n|}{(b_n - a_n)^6},
\end{align*}
and the right-hand side tends to zero as $n \to +\infty$.
A similar computation can be done for the other terms in \eqref{eq:examplelogterm23} and \eqref{eq:examplelogterm4}. Since $\phi_n + \omega_n \in \mathsf{D}_1$ by construction, we also have that
\begin{align*}
    \omega_n C'_\circ& (\phi_n + \omega_n) \int_{\phi_n}^{\phi_n + \omega_n} (\phi_n + \omega_n - t)^{2} C_\circ'(t) \diff t \\
    ={}& - (b_n - a_n +\epsilon_n) (b_n + \epsilon_n) \log(b_n + \epsilon_n) \Big(\! \int_{a_n}^{b_n} \!(b_n + \epsilon_n - t)^{2} t \diff t - \!\int_{b_n}^{b_n + \epsilon_n}\! (b_n + \epsilon_n - t)^{2} t \log t \diff t \!\Big) \\
    ={}&- (b_n - a_n + \epsilon_n)(b_n + \epsilon_n) \log(b_n + \epsilon_n) \Big[ \frac{1}{12}(b_n - a_n)^4 + o((b_n - a_n)^4) \Big]\\
    ={}& -\frac{1}{12}b_n^6\log(b_n)+o(-b_n^6\log(b_n)),\qquad \text{as } n \to +\infty.
\end{align*}
From \eqref{eq:boundremainder} and \eqref{eq:boundDremainder}, it is also not difficult to see that $R(\phi_n, \omega_n) =o(b_n^6)$, as well as $\partial_\omega R(\phi_n, \omega_n) =  o(-b_n^5\log(b_n))$, as $n \to +\infty$. Consequently, we obtain that
\[
\lim_{n \to +\infty} N(\phi_n, \omega_n) =\lim_{n \to +\infty} 1 + \frac{-\frac{1}{12}b_n^6\log(b_n)+o(-b_n^6\log(b_n))}{\frac{1}{72}b_n^6+o(b_n^6)} = +\infty.
\]
This shows that the metric measure space $(\hei,\di,\Leb^3)$ does not satisfy the $\mathsf{MCP}(0, N)$ for any $N\in (1,\infty)$, according to Proposition \ref{prop:MCPlimsup}.

\end{example}

Let us summarise the findings of this section. For $\phi^\star \in \R/2 \pi_{\Omega_\circ} \mathbb{Z} $, we introduce the following constants
\begin{equation}
    \label{eq:betaC0'}
    \overline{\alpha}(\phi^\star) :=\inf\left\{\alpha\geq 2:A(\phi^\star):= \sup_{\delta>0}\underset{|\phi-\phi^\star|<\delta}{\mathrm{ess\,inf}}\frac{C_\circ^\prime(\phi)}{|\phi-\phi^\star|^{\alpha-2}}>0\right\}. 
\end{equation}
and
\begin{equation}
    \label{eq:alphaC0'}
    \overline{\beta}(\phi^\star) := \sup\left\{\beta\geq 2: B(\phi^\star):= \inf_{\delta>0}\underset{|\phi-\phi^\star|<\delta}{\mathrm{ess\,sup}}\frac{C_\circ^\prime(\phi)}{|\phi-\phi^\star|^{\beta-2}}<+\infty\right\}
\end{equation}
Note that, by construction, we have that $\overline{\alpha}(\phi^\star) \geq \overline{\beta}(\phi^\star)$. Intuitively, they are the optimal constants for which following inequality holds asymptotically
\[
|\phi-\phi^\star|^{\overline{\alpha}(\phi^\star)-2}\lesssim C_\circ^\prime(\phi)\lesssim |\phi-\phi^\star|^{\overline{\beta}(\phi^\star)-2}\qquad\text{as }\phi\to\phi^\star.
\]
Observe that Theorem \ref{thm:veryflatzeropoint} states that if $\mathsf{MCP}$ holds, then $\overline{\beta}(\phi^\star)$ must be finite for every $\phi^\star$. The essence of the proof of Theorem \ref{thm:MCPinstrongly} is that if the supremum in \eqref{eq:alphaC0'} and the infimum in \eqref{eq:betaC0'} are equal, finite and attained, then $\limsup_{(\phi, \omega) \to (\phi^\star, 0)} N(\phi, \omega)$ is finite. The refinement in Theorem \ref{thm:MCPinstronglysmallo} shows that if furthermore %{\red the superior limit in \eqref{eq:alphaC0'} and the inferior limit in \eqref{eq:betaC0'} coincide}
$A(\phi^\star)= B(\phi^\star)$ for every $\phi^\star$, then $\limsup_{(\phi, \omega) \to (\phi^\star, 0)} N(\phi, \omega)  \geq  2 \alpha(\phi^\star) + 1$ and thus $N_{\mathrm{curv}} \geq  2 \alpha(\phi^\star) + 1$ for every $\phi^\star$.

One may wonder if these sufficient conditions for $\mathsf{MCP}$ are also necessary, and that is what the examples we provided are investigating. Example \ref{example:monotone} has $ \overline{\alpha}(\phi^\star)=\overline{\beta}(\phi^\star)$, while the supremum and infimum in \eqref{eq:alphaC0'} and \eqref{eq:betaC0'} are not attained, and yet the $\limsup_{(\phi, \omega) \to (\phi^\star, 0)} N(\phi, \omega)$ is finite. At this point, it is reasonable to ask whether the necessary and sufficient condition for $\mathsf{MCP}$ to hold is simply that $\overline{\alpha}(\phi^\star)=\overline{\beta}(\phi^\star) < +\infty$. Example \ref{example:oscilliation} shows that this is not the case. We constructed an example that has $\overline{\alpha}(\phi^\star)=\overline{\beta}(\phi^\star) < +\infty$ with $\limsup_{(\phi, \omega) \to (\phi^\star, 0)} N(\phi, \omega) = +\infty$. %{\red The reason for this behavior was because $C'_\circ$ oscillates too much.}

In conclusion, $\overline{\alpha}(\phi^\star)=\overline{\beta}(\phi^\star) < +\infty$ should be a necessary condition for $\mathsf{MCP}$ to hold, and the necessary and sufficient condition for the measure contraction property probably lies in the ``oscillation'' of the angle correspondence $C_\circ^\prime$ near its zero points.
There may be two obstacles for achieving this goal. Firstly, one would need to define rigorously a good notion of oscillation. Secondly, this oscillating function would need to be integrated and quantitatively estimated. %{\red We leave this problem as an open question.}

\section{Failure of the \texorpdfstring{$\cd(K,N)$}{CD(K,N)} condition in the \sF Heisenberg group}\label{sec:noCD}
In this section we prove Theorem \ref{thm:MAINnocd}. Our argument consists in a suitable refinement of the strategy developed in \cite[Sec.\ 5.3]{magnabosco2023failure}. For sake of clarity, we will explicitly adapt the proof of \cite[Thm.\ 5.24]{magnabosco2023failure} to prove Theorem \ref{thm:noCDstrongly}, as this adaptation consists of several non-trivial improvements. On the contrary, for the fundamental preliminary result (Proposition \ref{prop:dopopocafatica}) we will simply refer to \cite{magnabosco2023failure}, explaining the arguments can be easily adapted to the setting of this section. 

We consider the \sF Heisenberg group $\hei$, equipped with a $C^1$ strongly convex norm $\normdot$, recalling that in this case the correspondence map $C_\circ$ is single-valued and Lipschitz, thus differentiable almost everywhere. Moreover, according to Proposition \ref{prop:summary}, we know that if $p, q \in \hei$ are such that $p\star q^{-1}\notin\{x=y=0\}$, then there exists unique geodesic joining $p$ and $q$. Recall the definition of \emph{midpoint map}: 
\begin{equation}\label{eq:midpoint_hei}
    \M(p,q):=e_{\frac12}\left(\gamma_{pq}\right),\qquad \text{if } p\star q^{-1}\notin\{x=y=0\},
\end{equation}
where $\gamma_{pq}:[0,1]\to\hei$ is the unique geodesic joining $p$ and $q$. Similarly, we define the \emph{inverse geodesic map} $I_m$ (with respect to $m\in\hei$) as:
\begin{equation}\label{eq:inverse_geodesic_hei}
    I_m(q)=p, \qquad \text{if }\ \M(p,q)=m.
\end{equation}

In Lemma \ref{cor:diff_points_positive_measure}, we have identified a $\Leb^1$-positive set $ \mathsf D_1^+$ where $C_\circ$ is differentiable with positive derivative. In  \cite{magnabosco2023failure}, this property played a fundamental role in deducing Jacobian estimates for the exponential map. In particular, with the same strategy we can prove the following proposition, which is the analogous of \cite[Prop.\ 5.23]{magnabosco2023failure}.

\begin{remark}
    Recall Definition \ref{def:exponentialmap} and observe that 
    \begin{equation}
        G_t(r,\phi, \omega)= G_{-t}(r,\phi + \pi_{\Omega^\circ}, -\omega).
    \end{equation}
    In particular, the map $G_t$ for $t\in[-1,0]$ has the same properties of $G_{-t}$, cf. Proposition \ref{prop:summary} and Proposition \ref{prop:NONO}. 
\end{remark}

\begin{prop}\label{prop:dopopocafatica} 
     Let $\hei$ be the \sF Heisenberg group, equipped with a $C^1$ strongly convex norm $\normdot$. There is a full-measure set $\Dreg\subset  \mathsf D_1^+$ so that, for every $(r,\phi,\omega)\in \mathscr R= \mathscr U$ with $\phi\in \Dreg$, there exists a positive constant $\rho=\rho(r,\phi,\omega)$ such that for every $t\in[-\rho,\rho]\setminus\{0\}$:
    \begin{enumerate}
        \item[(i)] the inverse geodesic map $I_\e$ is well-defined and $C^1$ in a neighborhood of $G_t(r,\phi,\omega)$;
        \item[(ii)] the midpoint map $\M$ is well-defined and $C^1$ in a neighborhood of $(\e,G_t(r,\phi,\omega))$, moreover
        \begin{equation}
        \label{eq:suitable_jacobian_estimate}
           \big|\det \diff_{G_t(r,\phi,\omega)} \M(\e,\cdot) \big| \leq \frac 1{2^4}.
        \end{equation}
    \end{enumerate}
\end{prop}

\begin{proof}[Sketch of the proof]
    We follow the blueprint of \cite{magnabosco2023failure}. Therein, the map $C_\circ$ is differentiable with positive derivative $\Leb^1$-a.e.\ in $\R/2\pi_{\Omega^\circ}\mathbb{Z}$ and this allows to deduce the needed estimates for the Jacobian $\J_t(r,\phi,\omega)$, for every $r>0$, $\omega\in (-2 \pi_{\Omega^\circ}, 2 \pi_{\Omega^\circ}) \setminus\{0\}$ and for $\Leb^1$-a.e.\ $\phi\in \R/2\pi_{\Omega^\circ}\mathbb{Z}$. In this case, according to Lemma \ref{cor:diff_points_positive_measure}, the set $ \mathsf D_1^+\subset \R/2\pi_{\Omega^\circ}\mathbb{Z}$ where $C_\circ$ is differentiable with positive derivative, has only positive $\Leb^1$-measure. Nonetheless, the latter is sufficient to prove the estimates of \cite[Lem.\ 5.16--5.19]{magnabosco2023failure} for every $r>0$, $\omega\in (-2 \pi_{\Omega^\circ}, 2 \pi_{\Omega^\circ}) \setminus\{0\}$ and for $\Leb^1$-a.e.\ $\phi\in  \mathsf D_1^+$. Thus, \cite[Cor.\ 5.20]{magnabosco2023failure} holds also in this setting for $\Leb^1$-a.e.\ angle in $ \mathsf D_1^+$. Finally, once the latter result is proven, the proof of the current proposition can be carried out repeating verbatim the one of \cite[Prop.\ 5.23]{magnabosco2023failure}, with $ \mathsf D_1^+$ in place of $\R/2\pi_{\Omega^\circ}\mathbb{Z}$.
\end{proof}

\begin{theorem}\label{thm:noCDstrongly}
    Let $\hei$ be the \sF Heisenberg group, equipped with a $C^1$ strongly convex norm $\normdot$ and with a smooth measure $\m$. Then, the metric measure space $(\hei,\di,\m)$ does not satisfy the Brunn--Minkowski inequality $\bm(K,N)$, for every $K\in\R$ and $N\in (1,\infty)$.
\end{theorem}

\begin{proof}
    Take an angle $\phi \in \Dreg$ (cf.\ Proposition \ref{prop:dopopocafatica}) which is a density point for $\Dreg$. Fix $r>0$, $\omega \in (-2 \pi_{\Omega^\circ}, 2 \pi_{\Omega^\circ}) \setminus\{0\}$ and, for simplicity, call $\gamma$ the curve $\gamma_{(r,\phi,\omega)}$, i.e.
    \begin{equation}
        [0,1] \ni s \mapsto \gamma(t):= G_t(r,\phi,\omega).
    \end{equation}
   By Remark \ref{rmk:Jdefinedae}, we know that for every $t\in[0,1]$ the map $G_t$ is $C^1$, moreover, Proposition \ref{prop:stronglyJacpositivity} ensures that $\J_t(r,\phi, \omega) = \frac{r^3t}{\omega^4} \J_R(\phi, \omega t) > 0$. Therefore, $G_t$ is invertible, with $C^1$ inverse, in a neighborhood $B_t\subset \hei$ of $G_t(r,\phi,\omega)= \gamma(t)$. Consider the curve 
    \begin{equation}
        s \mapsto \alpha(s):=L_{\gamma(s)^{-1}}( \gamma(t+s)) =\gamma(s)^{-1}\star\gamma(t+s),
    \end{equation}
    and observe that, for $s$ sufficiently small, we have $\alpha(s) \in B_t$. Therefore, the function 
    \begin{equation}
        s \mapsto \Phi(s) :=  \p_2 \big[ G_t^{-1} \big(\alpha(s)\big)\big]=\p_2 \big[ G_t^{-1} \big(L_{\gamma(s)^{-1}}( \gamma(t+s))\big)\big],
    \end{equation}
    where $\p_i:T_\e^{\ast}\hei\to\R$ denotes the projection onto the $i$-th coordinate, is well-defined and $C^1$ (being composition of $C^1$ functions) in an open interval $I\subset\R$ containing $0$. In particular, note that $\Phi(s)$ is the ``initial angle'' for the geodesic joining $\e$ and $\alpha(s)$. 
    
    We are now going to prove that
    \begin{equation}\label{eq:megaclaim}\tag{$\mathsf C$}
        \Phi'(0)\neq 0.
    \end{equation}
    Consider the map $F:=\p_2\circ G_t^{-1}:\hei\to\R$. Since $F$ is $C^1$ in $B_t$ with non-zero differential at $\gamma(t)$, the set $\mathcal O:= F^{-1}(\phi)$ locally defines a $C^1$-surface and its tangent space is $T_p\mathcal O=\text{Ker } \diff_pF$, for every $p\in\mathcal O\cap B_t$. On the one hand, the tangent space to $\mathcal O$ is spanned by $\big\{\frac{\partial G_t}{\partial \omega},\frac{\partial G_t}{\partial r}\big\}$ and, on the other hand, $\Phi'(0)=\diff_{\alpha(0)} F(\dot\alpha(0))$. Therefore, \eqref{eq:megaclaim} is equivalent to showing that 
    \begin{equation}
    \label{eq:vector_in_span}
        \dot\alpha(0)\notin \text{Ker } \diff_{\alpha(0)}F=\text{span}\bigg\{ \frac{\partial G_t}{\partial \omega}(r,\phi,\omega),\frac{\partial G_t}{\partial r}(r,\phi,\omega)\bigg\}.
    \end{equation}
    In order to prove \eqref{eq:vector_in_span}, we observe that, since the left-translations are isometries and $\gamma$ is a geodesic, 
    \begin{equation}
        \di(\e,\alpha(s))=\di(\gamma(s),\gamma(t+s))=tr,\qquad\forall\,s\in I.
    \end{equation}
    This implies that $\p_1\big[ G_t^{-1} \big(\alpha(s)\big)\big]\equiv r$ for $s\in I$ and, as a consequence, reasoning as above with $\tilde F:=\p_1\circ G_t^{-1}$, we deduce that
    \begin{equation}
    \label{dumaron}
        \dot{\alpha}(0) \in \text{span} \bigg\{ \frac{\partial G_t}{\partial \phi}(r,\phi,\omega), \frac{\partial G_t}{\partial \omega}(r,\phi,\omega) \bigg\}.
    \end{equation}
    In conclusion, to prove \eqref{eq:megaclaim} it is enough to show that
    \begin{equation}\label{eq:claim}
        \dot{\alpha}(0) \not\in \text{span} \bigg\{ \frac{\partial G_t}{\partial \omega}(r,\phi,\omega)\bigg\}. 
    \end{equation}
    To this aim, we explicitly compute that 
    \begin{equation}
        \dot{\alpha}(0)= r \big( \cosom((\phi+\omega t)_\circ) - \cosom(\phi_\circ), \,  \sinom((\phi+\omega t)_\circ) - \sinom(\phi_\circ) , \,\triangle \,  \big),
    \end{equation}
    \begin{equation}
    \begin{split}
         \frac{\partial G_t}{\partial \omega}= \frac r {\omega^2} \big( t \omega \cosom ((\phi + \omega t)_\circ) &-  \left(\sinomp(\phi+\omega t ) - \sinomp(\phi)\right), \\
         &t\omega  \sinom ((\phi + \omega t)_\circ) +  (\cosomp(\phi+\omega t ) - \cosomp(\phi)), \, \diamondsuit\, \big),
    \end{split}
    \end{equation}
    where $\triangle$ and $\diamondsuit$ denote quantities that we do not need to make explicit. In particular, call $\tilde M$ the minor of the $(3\times 2)$-matrix
    \begin{equation}
       \begin{pNiceArray}{c|c}
     \dot \alpha (0) & \displaystyle\frac{\partial G_t}{\partial \omega}
    \end{pNiceArray},
    \end{equation}
    obtained by erasing the last row and observe that, in order to prove \eqref{eq:claim}, it is sufficient to verify that $\det \tilde M \neq 0$. The explicit computation shows that 
    \begin{equation}
        \begin{split}
            \det \tilde M = \frac{r^2}{\omega^2}\mathcal{J}_R (\phi,\omega t) >0,
        \end{split}
    \end{equation}
     cf. Proposition \ref{prop:stronglyJacpositivity}. This proves \eqref{eq:megaclaim}.

    Now, we want to find an interval $\tilde I \subset I$ (containing $0$) such that
    \begin{equation}\label{eq:7/8}
        \Leb^1 \big( \{ s\in  \tilde I \, :\, \Phi(s)\in \Dreg \}\big) > \frac{7}{8} \, \Leb^1(\tilde I).
    \end{equation}
    Consider the set 
    \begin{equation}
        J:=\{\psi\in \Phi(I): \psi \in \Dreg\}\subset\Phi( I)
    \end{equation}
    and observe that $J$ has density $1$ in $\Phi(0)=\phi$. 
    In fact, $\phi$ is a density point for $\Dreg$ and, as a consequence of \eqref{eq:megaclaim}, $\Phi(I)$ contains a neighborhood of $\phi$. Then, using once again claim \eqref{eq:megaclaim}, we can show that the set $\{ s\in   I \, :\, \Phi(s)\in \Dreg \}=\Phi^{-1}(J)\subset  I$ has density $1$ in $0$. In particular, up to taking a smaller interval $\tilde I \subset I$, we can realize \eqref{eq:7/8}.
    
    At this point, let $\bar s\in \tilde I$ such that $\Phi(\bar s) \in \Dreg$ and consider 
    \begin{equation}
        \bar\rho:=\rho\big(G_t^{-1}(\alpha(\bar s))\big)>0,
    \end{equation} 
    where $\rho(\cdot)$ is the positive constant found in Proposition \ref{prop:dopopocafatica}.
    For every $s \in [-\bar\rho,\bar\rho] \setminus\{0\}$, from Proposition \ref{prop:dopopocafatica}, we deduce that the inverse geodesic map $I_\e$ and the midpoint map $\M$ are well-defined and $C^1$ in a neighborhood of $G_s\big(G_t^{-1} (\alpha(\bar s))\big)$ and $\big(\e, G_s\big(G_t^{-1} (\alpha(\bar s))\big)\big)$, respectively. Moreover, we have that 
    \begin{equation}
           \big|\det \diff_{G_s(G_t^{-1} (\alpha(\bar s)))} \M(\e,\cdot) \big| \leq \frac 1{2^4}.
    \end{equation}
    Observe that, since the left-translations are smooth isometries, the inverse geodesic map $I_{\gamma(\bar s)}$ is well-defined and $C^1$ in a neighborhood of $\gamma(\bar s + s)$, in fact it can be written as
    \begin{equation}
        I_{\gamma(\bar s)} (p) = L_{\gamma(\bar s)} \big[I_\e \big(L_{\gamma(\bar s)^{-1}} (p)\big)\big],
    \end{equation}
    and $L_{\gamma(\bar s)^{-1}} \big(\gamma(\bar s + s)\big)=G_s\big(G_t^{-1} (\alpha(\bar s))\big)$.
    Similarly, we can prove that the midpoint map is well-defined and $C^1$ in a neighborhood of $(\gamma(\bar s),\gamma(\bar s + s))$, with
    \begin{equation}
        \big|\det \diff_{\gamma(\bar s + s)} \M(\gamma(\bar s),\cdot) \big| \leq \frac 1{2^4}.
    \end{equation}
% \todo[inline]{Ken: Can we choose measurable $\lambda$ or $\rho$? (Or re-choose $G$ so that $\lambda(\bar s)\geq \delta$ for some $\delta>0$ and any $\bar s\in G$)\\
% M,T: $\rho$ (and thus $\lambda$, since it is basically a reparametrization of $\rho$) is measurable from the construction in our previous paper, which is rather technical.\\
% K: OK!}
    In conclusion, up to restriction and reparametrization, we can find a geodesic $\eta:[0,1]\to \hei$ and a set $G \subset [0,1]$ with $\Leb^1(G):= m >\frac{7}{8}$, cf. \eqref{eq:7/8}, such that for every $\bar s\in G$, there exists $\lambda(\bar s)>0$ such that, for every $s\in [\bar s -\lambda(\bar s), \bar s + \lambda(\bar s)]\cap [0,1]\setminus\{\bar s\}$, the inverse geodesic map $I_{\eta(\bar s)}$ and the midpoint map $\M$ are well-defined and $C^1$ in a neighborhood of $\eta (s)$ and $(\eta(\bar s),\eta(s))$ respectively, and in addition
    \begin{equation}
        \big|\det \diff_{\eta( s)} \M(\eta(\bar s),\cdot) \big| \leq \frac 1{2^4}.
    \end{equation}
    Set $\lambda(s)=0$ on the set $[0,1]\setminus G$ and consider the set 
    \begin{equation}
        T:=\left\{(s,t)\in [0,1]^2\, : \, t\in [s-\lambda(s),s + \lambda(s)]\right\}. 
    \end{equation}
    Observe that, introducing for every $\epsilon>0$ the set
    \begin{equation}
        D_\epsilon:= \{(s,t)\in [0,1]^2\, : \, |t-s|<\epsilon\},
    \end{equation}
    we have that 
    
    \begin{equation}\label{eq:quasinoncicredo}
        \frac{\Leb^2(T \cap D_\epsilon)}{\Leb^2( D_\epsilon)} =\frac{\Leb^2(T \cap D_\epsilon)}{2\epsilon-\epsilon^2} \to m > \frac{7}{8}, \qquad \text{as }\,\epsilon \to 0.
    \end{equation}
    %%%% Proof of \eqref{eq:quasinoncicredo}
    % \begin{equation}
    %     \Leb^2(T\cap D_\epsilon)\geq \int_{\Lambda_\epsilon} \Leb^1([s-\lambda(s),s+\lambda(s)])\de s\geq 2\epsilon\Leb^1(\Lambda_\epsilon)\stackrel{\epsilon\to 0}{\sim}2\epsilon.  
    % \end{equation}
    On the other hand, we can find $\delta>0$ such that the set $
    \Lambda_\delta :=\{s\in [0,1] \, :\, \lambda(s)>\delta\}$ satisfies $\Leb^1(\Lambda_\delta) > \frac{13}{16}$. In particular, for every $\epsilon<\delta$ sufficiently small we have that 
    \begin{equation}\label{eq:quasinoncicredo2}
        \Leb^2\left(\left\{(s,t)\in [0,1]^2\,:\, \frac{s+t}{2}\in  \Lambda_\delta \right\} \cap D_\epsilon \right) > \frac{3}{2} \epsilon>\frac34\Leb^2(D_\epsilon).
    \end{equation}
    Therefore, putting together \eqref{eq:quasinoncicredo} and \eqref{eq:quasinoncicredo2}, we can find $\epsilon<\delta$ sufficiently small such that 
    \begin{equation}
        \Leb^2\left(T \cap D_\epsilon \cap \left\{(s,t)\in [0,1]^2\,:\, \frac{s+t}{2}\in  \Lambda_\delta \right\} \right) > \frac 12 \Leb^2( D_\epsilon).
    \end{equation}
    Then, since the set $D_\epsilon$ is symmetric with respect to the diagonal $\{s=t\}$, we can find $\bar s\neq\bar t$ such that 
    \begin{equation}
        (\bar s,\bar t),(\bar t,\bar s) \in T \cap D_\epsilon \cap \left\{(s,t)\in [0,1]^2\,:\, \frac{s+t}{2}\in  \Lambda_\delta \right\} .
    \end{equation}
    In particular, this tells us that: 
    \begin{itemize}
        \item[(i)] $\bar t\in [\bar s-\lambda(\bar s),\bar s + \lambda(\bar s)]$ and $\bar s\in [\bar  t-\lambda(\bar  t),\bar t + \lambda(\bar  t)]$;
        \item[(ii)] $|\bar t-\bar s|< \epsilon <\delta$;
        \item[(iii)] $ \frac{\bar s+\bar t}{2}\in  \Lambda_\delta$.
    \end{itemize}
    Now, on the one hand, (i) ensures that the midpoint map $\M$ is well-defined and $C^1$ in a neighborhood of $(\eta(\bar s), \eta(\bar t))$ with  
    \begin{equation}
        \big|\det \diff_{\eta( \bar t)} \M(\eta(\bar s),\cdot) \big| \leq \frac 1{2^4} \qquad \text{and} \qquad \big|\det \diff_{\eta( \bar s)} \M(\cdot, \eta(\bar t)) \big| \leq \frac 1{2^4}.
    \end{equation}
    While, on the other hand, the combination of (ii) and (iii) guarantees that the inverse geodesic map $I_{\eta(\frac{\bar s+\bar t}{2})}$ is well-defined and $C^1$ in a neighborhood of $\eta(\bar s)$ and in a neighborhood of $ \eta(\bar t)$ respectively. Indeed, we have:
    \begin{equation}
        \bar s,\bar t\in \left[\frac{\bar s+\bar t}{2}- \delta,\frac{\bar s+\bar t}{2}+ \delta\right] \subset \left[\frac{\bar s+\bar t}{2}- \lambda\left(\frac{\bar s+\bar t}{2}\right),\frac{\bar s+\bar t}{2}+ \lambda\left(\frac{\bar s+\bar t}{2}\right)\right],
    \end{equation}
    and, by the very definition of $\lambda(\cdot)$, we obtain the claimed regularity of the inverse geodesic map.
    
    Once we have these properties, we can repeat the same strategy used in the second part of the proof of \cite[Thm.\ 4.26]{magnabosco2023failure} and contradict the Brunn--Minkowski inequality $\bm(K,N)$ for every $K\in \R$ and every $N\in (1,\infty)$.
\end{proof}

\bibliographystyle{alpha} 
\bibliography{bibliography}

\end{document}